\documentclass{amsart}
\usepackage{fontenc}
\usepackage[english]{babel}
\usepackage{nicefrac}

\usepackage{amssymb, amsmath,mathtools}
\usepackage{mathrsfs,mathabx}
\usepackage{amscd}
\usepackage{verbatim}
\usepackage{stmaryrd}
\usepackage[dvipsnames]{xcolor}
\usepackage{subfig}

\usepackage{enumerate}
\usepackage{stmaryrd,esint}

\usepackage{tikz}
\usepackage{xcolor}
\usetikzlibrary{arrows.meta, decorations.pathmorphing, patterns}

\pgfdeclaredecoration{penciline}{initial}{
\state{initial}[width=+\pgfdecoratedinputsegmentremainingdistance,
    auto corner on length=1mm,]{
    \pgfpathcurveto
    {
        \pgfqpoint{\pgfdecoratedinputsegmentremainingdistance}
        {\pgfdecorationsegmentamplitude}
    }
    {
        \pgfmathrand
        \pgfpointadd{\pgfqpoint{\pgfdecoratedinputsegmentremainingdistance}{0pt}}
        {\pgfqpoint{-\pgfdecorationsegmentaspect
            \pgfdecoratedinputsegmentremainingdistance}
        {\pgfmathresult\pgfdecorationsegmentamplitude}
        }
    }
    {
        \pgfpointadd{\pgfpointdecoratedinputsegmentlast}{\pgfpoint{1pt}{1pt}}
    }
}
\state{final}{}
}

\usepackage[colorlinks,linkcolor={blue},citecolor={blue},urlcolor={purple},]{hyperref}

\usepackage[colorinlistoftodos,prependcaption,textsize=tiny]{todonotes}

\theoremstyle{plain}
\newtheorem{theorem}{Theorem}[section]
\theoremstyle{remark}
\newtheorem{remark}[theorem]{Remark}

\theoremstyle{plain}

\newtheorem{lemma}[theorem]{Lemma}
\newtheorem{proposition}[theorem]{Proposition}

\newtheorem{definition}[theorem]{Definition}

\newtheorem{assumption}[theorem]{Assumption}

\numberwithin{equation}{section}

\def\N{{\mathbb N}}
\def\Z{{\mathbb Z}}
\def\Q{{\mathbb Q}}
\def\R{{\mathbb R}}

\newcommand{\E}{{\mathbf E}}
\renewcommand{\P}{{\mathbf P}}
\newcommand{\F}{{\mathcal F}}
\newcommand{\Progress}{\mathcal{P}}

\newcommand{\g}{\gamma}
\newcommand{\om}{\omega}
\renewcommand{\O}{\Omega}
\renewcommand{\a}{\kappa}

\newcommand{\loc}{{\rm loc}}
\newcommand{\Tor}{\mathbb{T}}
\newcommand{\T}{\Tor}
\newcommand{\calL}{\mathscr{L}}

\newcommand{\one}{{{\bf 1}}}
\newcommand{\embed}{\hookrightarrow}
\renewcommand{\H}{\mathcal{H}}
\newcommand{\M}{\overline{\mathcal{M}}}
\newcommand{\NN}{\mathcal{N}}
\renewcommand{\L}{\mathcal{T}}

\newcommand{\Xap}{X^{\mathrm{Tr}}_{\a,p}}
\newcommand{\Yar}{Y^{\mathrm{Tr}}_{\alpha,r}}

\newcommand{\dd}{\mathrm{d}}
\newcommand{\gap}{\mathsf{Exc}}

\newcommand{\ggap}{\mathsf{exc}}
\renewcommand{\r}{\ell}
\newcommand{\X}{\mathscr{X}}
\newcommand{\Y}{\mathscr{Y}}

\newcommand{\xx}{\mathcal{Z}}

\newcommand{\re}{\mathscr{T}_{\mathrm{Reg}}}
\newcommand{\reg}{\re}
\newcommand{\si}{\mathscr{T}_{\mathrm{Sin}}}

\newcommand{\wt}{\widetilde}

\newcommand{\set}{\mathbb{X}}
\newcommand{\setfull}{(X_0,X_1,p,\a)}
\newcommand{\sety}{\mathbb{Y}}
\newcommand{\setfully}{(Y_0,Y_1,r,\alpha)}
\newcommand{\setfullzero}{(X_0,X_1,r_1,\kappa_1)}

\newcommand{\Hs}{\mathbb{H}}
\newcommand{\Bs}{\mathbb{B}}
\newcommand{\Ls}{\mathbb{L}}
\newcommand{\p}{\mathbb{P}}
\newcommand{\q}{\mathbb{Q}}
\newcommand{\D}{\mathcal{D}}
\newcommand{\Borel}{\mathcal{B}}
\newcommand{\Dom}{\mathcal{O}}

\newcommand{\Iz}{N_u}

\newcommand{\ulambda}{u_{\lambda}}

\newcommand{\moll}{\mathcal{S}}
\newcommand{\cc}{\mu}
\newcommand{\Cw}{C_{{\rm w}}}

\newcommand{\A}{\mathcal{A}}

\newcommand{\Abin}{\mathscr{E}}
\newcommand{\reim}{\mathtt{R}_\varepsilon}

\allowdisplaybreaks
\makeatletter
\@namedef{subjclassname@2020}{%
  \textup{2020} Mathematics Subject Classification}
\makeatother

\begin{document}

\author{Antonio Agresti}
\address{Department of Mathematics Guido Castelnuovo, Sapienza University of Rome,
P.le Aldo Moro 5, 00185 Rome, Italy}
\email{antonio.agresti@uniroma1.it}

\thanks{The author is a member of GNAMPA (INdAM) and acknowledges support from INdAM through the GNAMPA 2026 project ``Fluidodinamica stocastica: irregolarità, trasporto e fenomeni di regolarizzazione''}

\date\today

\title[{Fractal dimension of singular times for SPDE\lowercase{s}}]{Fractal dimension of singular times for SPDE\lowercase{s}: Energy bounds, criticality, and\\ weak-strong uniqueness}



\begin{abstract}
For several physically relevant SPDEs, it is known that global weak solutions coexist with local strong ones. Typically, weak-strong uniqueness results are known and ensure that the global and strong solutions coincide as long as the latter exist. Times at which a weak solution does not coincide with a strong one are called \emph{singular times}. Determining their fractal dimension is fundamental to capturing the regularity of weak solutions.

We introduce a probabilistic notion of singular times for a wide class of semilinear SPDEs. We show that sets of singular times have fractal dimension (i.e., Hausdorff and/or Minkowski) at most $ 1-\r\, \gap$, where $\r$ and $\gap$ are the \emph{time integrability} and the \emph{excess of spatial regularity} compared to the critical regularity of the \emph{energy bound} associated with weak solutions, respectively. Moreover, their corresponding $(1-\r\,\gap )$-dimensional measure is zero. We formulate and apply our theory to annealed strong Leray-Hopf solutions of 3D Navier-Stokes equations (NSEs) with physically relevant noises, including rough Kraichnan and Lie transport. In particular, we extend the fundamental $1/2$-dimensional bound of Leray and Scheffer on singular times for 3D NSEs to the stochastic setting, and we prove new conditional results under supercritical Serrin's conditions, irrespective of the roughness of the noise. Our abstract framework is new even in the deterministic case, and provides the first partial regularity results for weak solutions to SPDEs with multiplicative noise.
\end{abstract}

\maketitle

\tableofcontents

\section{Introduction}
\label{s:introduction}
Weak solutions to stochastic PDEs (SPDEs) appear in a variety of contexts in applied sciences, including fluid mechanics, biology, and chemistry (e.g., reaction-diffusion equations and phase-separation processes). Understanding the regularity of weak solutions to SPDEs is often a challenging task, and the global smoothness of such solutions cannot be expected in general. In this scenario, it is of key interest to quantify the size of the set where possible singularities might arise.
The prototypical example is the 3D Navier-Stokes equations (NSEs), which model the motion of an incompressible fluid. Their physically motivated stochastic variants will be the guiding examples of our investigation.

The set of singularities of a weak solution to an SPDE can present a rather complicated structure. Thus, it is central to obtain sharp bounds on the fractal dimension (e.g., Hausdorff) of the singular set. 
In the absence of noise, such results are well-known for many PDEs, and they are usually referred to as \emph{partial regularity}. 
The current manuscript appears to be the first work in the largely unexplored field of partial regularity for SPDEs with multiplicative noise, which is typically of primary physical interest, as discussed for the NSEs in Subsection \ref{ss:NSEs_time_intro} below.

\smallskip

There are various ways to define the set of singularities. 
One of these, whose extension to the stochastic setting is one of the main contributions of this work, is via the so-called \emph{singular times}. The key idea behind this concept is as follows. A time $t_0>0$ is called \emph{regular} for a solution $u$ to a given PDE on a domain $\Dom$ if there exists an open interval $I_0$ containing $t_0$ such that 
\begin{equation}
\label{eq:smoothness_regular_times}
u|_{I_{0}\times \Dom}\in C^{\infty}(I_0\times \Dom).
\end{equation}
\emph{Singular times} are those times that are not regular. 

\smallskip
 
In this manuscript, we provide a robust framework for studying the fractal dimensions of singular times for SPDEs with \emph{multiplicative noise}.
Our setting offers a unified treatment of singular times, relying on the recent theory of SPDEs in critical spaces (see e.g., \cite{AV25_survey}), which is not limited to stochastic 3D NSEs. We refer to the end of Subsection \ref{ss:comparison} for comments on further applications. 
We illustrate the main results of this manuscript as follows:
\begin{itemize}
\item Section \ref{ss:NSEs_time_intro} -- Bounds on the fractal dimension of \emph{singular times} of \emph{annealed} strong Leray-Hopf solutions to 3D NSEs with multiplicative noise of transport and Lie type (see Theorem \ref{t:NSE_intro}). In particular, we extend the classical results by Leray \cite{Leray} and Scheffer \cite{S76_original} to the stochastic setting. 
\item Section \ref{ss:abstract_intro} -- Bounds on the fractal dimension of singular times for abstract SPDEs; see Theorem \ref{t:abstract_intro}. The latter encompasses all situations between global irregularity and global regularity, i.e., the \emph{partial} regularity regime (see Figure \ref{fig:regularity_regimes}). 
\end{itemize}

The results on the NSEs are a consequence of the ones for abstract SPDEs; see Figure \ref{fig:dimension_bounds}.
Comments on space-time singular sets in the spirit of the celebrated work of Caffarelli, Kohn, and Nirenberg \cite{CKN82} and related literature are given in Subsection \ref{ss:comparison}. Further applications of our results to, e.g., reaction-diffusion equations and other models from fluid dynamics are discussed at the end of Subsection \ref{ss:comparison}.

\subsection{Singular times for stochastic 3D Navier-Stokes equations}
\label{ss:NSEs_time_intro}
The NSEs with multiplicative noise of transport and Lie type on $\T^3=\R^3/\Z^3$ read as follows:
\begin{align}
\label{eq:NSE_intro}
\partial_t u 
= \Delta u -\nabla p -(u\cdot \nabla)u +\sum_{n\geq 1} \big[-\nabla \wt{p}_n + (\sigma_n\cdot \nabla) u + \cc_n\cdot u \big] \circ \dot{W}^n,
\end{align}
together with the incompressibility condition $\nabla\cdot u=0$, and initial condition $u(0)=u_0$.
In the above, $u:[0,\infty)\times \O\times \T^3\to \R^3$ denotes the unknown velocity field, $p,\wt{p}_n:[0,\infty)\times \O\times \T^3\to \R$ the unknown pressures, $(W^n)_n$ a sequence of independent standard Brownian motions and $\circ$ the Stratonovich integration. Moreover, $\sigma_n$ is a divergence-free vector field for all $n$, and for some $\g>0$,
\begin{equation}
\label{eq:regularity_NSE_coefficients}
(\sigma_n)_n\in C^{\g}(\T^3;\ell^2(\N;\R^3))\quad \text{ and }\quad (\cc_n)_n\in C^{\g}(\T^3;\ell^2(\N;\R^{3\times 3})).
\end{equation}

\subsubsection{Physical motivations and related literature}
The stochastic NSEs \eqref{eq:NSE_intro} cover the following two cases of physical interest: 

\begin{itemize}
\item
\emph{(Rough) transport noise:} $\cc_n=0$ and 
$(\sigma_n)_n\in C^{\g}(\T^3;\ell^2(\N;\R^3))$;
\vspace{0.05cm}
\item
\emph{Stochastic Lie transport:} $\cc_n=\nabla \sigma_n$ and $(\sigma_n)_n\in C^{1+\g}(\T^3;\ell^2(\N;\R^3))$;
\end{itemize}
for arbitrary $\g>0$.
Physical derivations and comments on stochastic 3D NSEs \eqref{eq:NSE_intro} with the above noises can be found in e.g., \cite{BrCaFl,DP22_two_scale,FlaPa21,MR01,MR04,BPW25,M14_derivation} and \cite{FL32_book,H15_SVP,DM25_PhysD}, respectively. 
Notably, transport noise of Kraichnan type can reproduce the \emph{Kolmogorov spectrum} of turbulence for $\g=\frac{2}{3}$, see \cite[pp. 426-427 and 436]{MK99_simplified} (and also \cite[Remark 5.3]{GY25} or \cite[eq.\ (1.4)-(1.5)]{Primitive3}).
Recalling the derivation of \eqref{eq:NSE_intro} with $\cc_n\equiv 0$ in \cite{MR04} via the stochastic Lagrangian approach, i.e., the trajectory $(x_t)_t$ of a fluid particle located at $x$ at the initial time satisfies
$$
\textstyle
 \dot{x}_t =u(t,x_t)+ \sum_{n\geq 1}\sigma_n(x_t)\circ \dot{W}^n,\quad x_0=x\in \T^3.
$$
The case of \emph{rough} Kraichnan noise $\g\in (0,1)$ can be thought of as an intermediate turbulent situation in which the small-scale part of the velocity $\sum_{n\geq 1}\sigma_n\circ  \dot{W}^n$ is fully turbulent, while the large-scale part $u$ is not.  
Nowadays, there is an extensive and still growing literature on stochastic NSEs; see e.g., \cite{AV21_NS,BrCaFl,BzMo13,DGHT11,FL19,FL32_book,G26_smooth,GC24_Lie,MR04,MR05,HLN21,KX26_JDE}, and it is not possible to give a complete account here.

\smallskip

In the deterministic setting, the study of singular times of weak solutions to the NSEs goes back to the fundamental work of Leray \cite{Leray}, where the existence of global weak solutions for 3D NSEs satisfying an energy inequality (usually referred to as Leray-Hopf solutions, see \cite{Hopf} and \eqref{eq:annealed_energy_inequality_intro}) was proven. His work contains the ingredients for the bound on the Hausdorff dimension of the set of singular times by $1/2$, later made explicit by Scheffer \cite{S76_original} (see also \cite[p.\ 536]{S76_partial}). The reader is referred to Subsection \ref{ss:fractal} and \cite{Falconer_book} for basic notions on fractal dimensions and measures. In \cite{S76_original,S76_partial}, Scheffer additionally proved that such a set has zero $1/2$-dimensional Hausdorff measure. More recently, these results have been refined by Robinson and Sadowski \cite[Corollary 3.1]{RS1_07} and Kukavica \cite[Theorem 2.10]{K09_singular_times} using the Minkowski (or box-counting) content and dimensions. 
It is worth noting that singular times have also been considered in the recent work \cite{BCV22} by Buckmaster, Colombo, and Vicol, where they proved the existence of very weak solutions to the 3D NSEs whose singular times have box-counting dimension less than $1$.

\subsubsection{Bounds on singular times for stochastic 3D NSEs}
To extend \eqref{eq:smoothness_regular_times} to the stochastic setting, let us recall that, under the regularity assumption  \eqref{eq:regularity_NSE_coefficients}, it follows from \cite[Theorem 2.4]{AV21_NS} that strong solutions of \eqref{eq:NSE_intro} have a.s.\ paths in 
\begin{equation}
\label{eq:reguality_of_strong_solutions}
C^{1/2-,(1+\g)-}_{\loc}((0,T)\times \T^3),
\end{equation} 
see Subsection \ref{ss:notation} for the notation. The thresholds $1/2$ and $1+\g$ are optimal for regularity of solutions to \eqref{eq:NSE_intro} even for strong ones, making the latter space the natural replacement for the smoothness condition in \eqref{eq:smoothness_regular_times}. 
Let $u:[0,\infty)\times \O\to H^{1}(\T^3;\R^3)$ be a weak solution to \eqref{eq:NSE_intro}, that is, the latter being satisfied in the weak PDE sense in space, and in the natural integral form in time (see Definition \ref{def:Leray_Hopf_strong}\eqref{it:Leray_Hopf_strong1}-\eqref{it:Leray_Hopf_strong2} for the precise formulation).   
For $\varepsilon\in (0,1)$, a time $t_0>0$ is said to be an $\varepsilon$-regular time for $u$ (or, briefly, $t_0\in \re^\varepsilon$), if there exists a \emph{stochastic interval} $(t,\tau)$ where $t<t_0$ and $\tau:\O\to [t,\infty]$ is a stopping time such that
\begin{equation}
\label{eq:condition_singular_times}
\P(\tau>t_0)>1-\varepsilon \quad \text{ and }\quad  u|_{(t,\tau)}\in C_{\loc}^{1/2-,(1+\g)-}((t,\tau)\times \T^3;\R^3)\text{ a.s. }
\end{equation}
For each $\varepsilon\in (0,1)$, the sets of \emph{$\varepsilon$-singular and singular times of $u$} are given by  
\begin{equation}
\label{eq:singular_sets_def_intro}
\si^\varepsilon=[0,\infty)\setminus \re^\varepsilon \quad \text{ and }\quad \si=\textstyle{\bigcup_{\varepsilon\in (0,1)} } \si^\varepsilon,
\end{equation}
respectively (see Definition \ref{def:singular_times_NS}).
As in the deterministic case, the notion of weak solutions as given above is too weak to allow for estimating the size of singular times (see e.g., \cite[Theorem 13.5]{LePi}). Here, we consider \emph{annealed strong Leray-Hopf solutions} to the 3D NSEs \eqref{eq:NSE_intro} (see Definition \ref{def:Leray_Hopf_strong}\eqref{it:Leray_Hopf_strong3}), i.e., in addition to being weak solutions in the PDE-sense, $u$ also satisfies the \emph{annealed strong energy inequality}: 

For a.a.\ $t\in \R_+$ and all stopping times $\tau:\O\to [t,\infty)$, it holds that 
\begin{align}
\label{eq:annealed_energy_inequality_intro}
&\E\|u( \tau)\|_{L^2}^2
+2\, \E \int_{t}^{\tau} \int_{\T^3}|\nabla u|^2\,\dd x \,\dd r \\
\nonumber
&\qquad \leq \E \|u(t)\|_{L^2}^2
+\E\int_{t}^{ \tau} \int_{\T^3}  \big[((\sigma_n\cdot\nabla)+ \cc_n ) u \big]\cdot \p\big[(\cc_n+\cc_n^\top) u\big]  \,\dd x \,\dd r,
\end{align}
where $\p$ denotes the Helmholtz projection, see Subsection \ref{sss:function_spaces_divfree}.
In the above, \emph{strong} refers to the fact that the initial time in the above inequality can be chosen in a set of full measure, cf.\ \cite[Proposition 12.1]{LePi}.
We formulate the notion of annealed strong Leray-Hopf solutions--a natural stochastic counterpart to the deterministic strong Leray-Hopf framework--and prove their existence in Proposition \ref{prop:existence_strong_Leray} (see also \cite{gess2023landau,GL26_strong} and the comments below the proposition for further comments). 

\smallskip

The annealed strong energy inequality \eqref{eq:annealed_energy_inequality_intro} seems to be the weakest condition under which the following partial regularity results for the stochastic 3D NSEs \eqref{eq:NSE_intro} hold.
In particular, a pathwise energy inequality, which is expected in the case of pure transport $\cc_n\equiv 0$, is not needed.
We refer to Subsection \ref{ss:fractal} for the Hausdorff and Minkowski dimensions, and measure/content $\H$ and $\M$.

\begin{theorem}[Bounds on singular times for 3D stochastic NSEs -- Informal version of Theorems \ref{t:singular_times_SNS} and \ref{t:singular_times_SNS_2}]
\label{t:NSE_intro}
Let $u$ be an annealed strong Leray-Hopf solution to the stochastic {{\normalfont 3D NSEs}} \eqref{eq:NSE_intro}. Let $\si^{\varepsilon}$ and $\si$ be the $\varepsilon$-singular and singular times of $u$ with $\varepsilon\in (0,1)$, respectively. Then, the following assertions hold.
\begin{enumerate}[{\rm(1)}]
\item\label{it:NSE_intro1} {\rm ($1/2$-bounds)} For all $\varepsilon\in (0,1)$, it holds that
\begin{align*}
\dim_{\M}(\si^\varepsilon)&\leq 1/2, \qquad 
\M^{1/2}(\si^\varepsilon)=0, & & \text{\normalfont{ (Minkowski/Box-counting)}}\\
\dim_{\H}(\si)&\leq 1/2, \qquad \ 
\H^{1/2}(\si)=0.& &  \text{ \normalfont{(Hausdorff)}}
\end{align*}
\item\label{it:NSE_intro2} {\rm (Conditional bounds -- Supercritical Serrin conditions)} Assume further that there exist $p_0\in (2,\infty)$ and $q_0\in (3,\infty)$ such that $\frac{2}{p_0}+\frac{3}{q_0}>1$, and 
$$
\E\int_0^T \|u\|_{L^{q_0}(\T^3;\R^3)}^{p_0}\,\dd t <\infty\ \ \text{ for all }\ T<\infty.
$$ 
Set $\delta_{0}=\frac{p_0}{2}(\frac{2}{p_0}+\frac{3}{q_0}-1)$. 
Then, for all $\varepsilon\in (0,1)$, it holds that
\begin{align*}
\dim_{\M}(\si^\varepsilon)&\leq \delta_0, \qquad 
\M^{\delta_0}(\si^\varepsilon)=0, & & \text{\normalfont{ (Minkowski/Box-counting)}}\\
\dim_{\H}(\si)&\leq \delta_0, \qquad \ 
\H^{\delta_0}(\si)=0.& &  \text{ \normalfont{(Hausdorff)}}
\end{align*}
\end{enumerate}
\end{theorem}

The above provides an extension of the results by Leray and Scheffer \cite{Leray,S76_partial}, and Kukavica and Robinson-Sadowski \cite{K09_singular_times,RS1_07} to the stochastic setting.
Moreover, the obtained bound is \emph{independent} of the noise regularity $\g>0$ in \eqref{eq:regularity_NSE_coefficients}. 
Together with the foundational works \cite{FG95_PTRF,MR05}, Theorem \ref{t:NSE_intro} partially closes the gap between the deterministic and stochastic theories of Leray-Hopf solutions to 3D NSEs.

It is worth mentioning that, in light of the recent breakthrough results \cite{ABC22_annals,hou2025nonuniqueness} on non-uniqueness of Leray-Hopf solutions in the deterministic setting, one may expect the bounds in \eqref{it:NSE_intro1} and \eqref{it:NSE_intro2} to be optimal.

Theorem \ref{t:NSE_intro} is actually a special case of a general result on abstract SPDEs presented in Theorem \ref{t:abstract_intro} below. For a schematic application of the latter leading to Theorem \ref{t:NSE_intro}, see Figure \ref{fig:dimension_bounds}. For the whole-space case, see Remark \ref{r:equilvalent_singular_times_SNS_3D}.

\smallskip

It remains an open problem to determine whether the bounds on the singular times $\si$ in Theorem \ref{t:NSE_intro} can be obtained in terms of the Minkowski/box-counting content or dimension. 
This is due to the lack of $\sigma$-subadditivity of the Minkowski content (see Remark \ref{r:lack_subadditivity}) and the definition of the set of singular times in \eqref{eq:singular_sets_def_intro}.

\smallskip

The conditional bounds under supercritical Serrin's conditions in \eqref{it:NSE_intro2} are an extension to the stochastic setting of the ones obtained by Giga \cite[Theorem 5]{G86} (see \cite{L19_PhysD} for a related result).
In Theorem \ref{t:singular_times_SNS_2}, we also consider $L^{q_0}$ replaced by a Besov space of \emph{negative smoothness}. The latter seems unknown even in the deterministic case.
Unfortunately, it is unclear whether the assumed bound in \eqref{it:NSE_intro2} is satisfied by an annealed Leray-Hopf solution. 
It is worth noticing that the dimensional bound $\delta_0$ on the singular times in \eqref{it:NSE_intro2} vanishes in case the Serrin condition $\frac{2}{p_0}+\frac{3}{q_0}=1$ holds, in which case global smoothness is expected. Further comments are also given in Subsection \ref{ss:comparison} below.

\subsubsection{Towards an abstract theory for singular times of SPDEs}
\label{sss:towards_abstract_NSE}
One of the key ideas behind the proof of Theorem \ref{t:NSE_intro}\eqref{it:NSE_intro1}-\eqref{it:NSE_intro2} exploits the strong annealed energy inequality \eqref{eq:annealed_energy_inequality_intro} to `connect' at a.a.\ times $t>0$ the annealed Leray-Hopf solution $u$  to \eqref{eq:NSE_intro} to a strong solution $v$ (if it exists) to the same SPDE which has a.s.\ paths in \eqref{eq:reguality_of_strong_solutions} (see \cite{AV21_NS}). This allows us to conclude that times $t_0>t$ ``sufficiently close'' to $t$ are \emph{regular}. Here, sufficiently close should be understood in terms of the \emph{lifetime} of the strong solution $v$ with initial data $u(t)$. In this way, we essentially `transfer' the regularity of strong solutions to the weak ones. These types of uniqueness results are referred to as 
\begin{itemize}
\item
\emph{Weak-strong} uniqueness -- Proposition \ref{prop:weak_strong_uniqueness};
\end{itemize}
for the deterministic case, see e.g., \cite[Theorem 13.5]{LePi}. Interestingly, Proposition \ref{prop:weak_strong_uniqueness} extends weak-strong uniqueness for stochastic NSEs \eqref{eq:NSE_intro} to encompass strong solutions with critical regularity (see the comments below Proposition \ref{prop:weak_strong_uniqueness} for related literature on this).

Before discussing the connection of annealed stochastic Leray-Hopf solutions to strong solutions, we briefly comment on the existence of the latter. To this end, we discuss the scaling invariance of the 3D stochastic NSEs, from which, as is well-known in the context of PDEs, one can deduce the critical regularity threshold for well-posedness of the corresponding SPDE. 
As discussed in \cite[Subsection 1.1]{AV21_NS} (see e.g., \cite{Can04,LePi,PW18} for the deterministic case), for all $\lambda>0$, the mapping
$$
u\mapsto u_{\lambda} (t,x)= \lambda^{1/2}u(\lambda t,\lambda^{1/2}x)\quad \text{ for }(t,x)\in \R_+\times \R^3,
$$
leaves locally invariant the set of solutions to \eqref{eq:NSE_intro}.
In particular, by setting $t=0$ in the above, it induces the following mapping on the initial data
\begin{equation}
\label{eq:invariant_maps_NSE}
u_0\mapsto u_{0,\lambda}\stackrel{{\rm def}}{=}\lambda^{1/2} u_0(\lambda^{1/2}\cdot).
\end{equation}
\emph{Critical spaces} are those spaces that are (locally) invariant under the mapping $u_0\mapsto u_{0,\lambda}$. Typical examples are $L^3(\T^3)$, $H^{1/2}(\T^3)$, or more generally, $B^{3/q-1}_{q,p}(\T^3)$ for all $q,p\in (1,\infty)$. Note that the Sobolev index\footnote{The Sobolev index of $H^{\sigma,q}(\T^d)$ or $B^{\sigma}_{q,p}(\T^d)$ is $\sigma-d/q$ and rules the local scaling of the space: $\|f(\lambda\cdot)\|_{\dot{H}^{\sigma,q}(\R^d)}\eqsim \lambda^{\sigma-d/q}\|f\|_{\dot{H}^{\sigma,q}(\R^d)}$ for $\lambda>0$, and similar for Besov spaces.}  
of all the previously mentioned spaces is $-1$, and the latter is the critical regularity threshold for NSEs.

From the energy balance \eqref{eq:annealed_energy_inequality_intro}, it follows that 
\begin{equation}
\label{eq:H1_regularity_NSEs}
\E\|u(\cdot,t)\|_{H^1(\T^3;\R^3)}^2<\infty\quad \text{ for a.a. $t>0$}.
\end{equation}
Thus, the discussion below \eqref{eq:invariant_maps_NSE} shows that the space $H^{1}(\T^3;\R^3)$ is \emph{subcritical}, and therefore \eqref{eq:H1_regularity_NSEs} implies the existence of a strong solution to the stochastic 3D NSEs \eqref{eq:NSE_intro}, say $v$ with lifetime $\tau$ and initial data $u(t)$ at time $t$, and by the above-mentioned weak-strong uniqueness, we have $u=v$ a.e.\ on $[t,\tau)\times \O$. Now, \cite[Theorem 2.4]{AV21_NS} ensures that the stopping time $\tau$ is such that the second condition in \eqref{eq:condition_singular_times} holds, while to check the first in \eqref{eq:condition_singular_times}, we prove \emph{quantitative} lower bounds on the lifetime $\tau$ depending only on (see Section \ref{s:quantifying_main_proof} for the abstract result)
\begin{itemize}
\item \emph{Excess} of Sobolev regularity/index of $H^1(\T^3)$ over the critical threshold $-1$, that is $\frac{1}{4}$, see Figure \ref{fig:dimension_bounds}. 
\end{itemize}
This and a covering argument prove Theorem \ref{t:NSE_intro}\eqref{it:NSE_intro1}. For Theorem \ref{t:NSE_intro}\eqref{it:NSE_intro2}, a similar argument applies, where instead of $H^1(\T^3)$, one considers $L^{q_0}(\T^3)$ with a corresponding change of the excess from criticality; see Figure \ref{fig:dimension_bounds}. 

\smallskip

Next, we discuss the abstract theory behind Theorem \ref{t:NSE_intro}. In doing so, we generalize the argument sketched above for the stochastic NSEs \eqref{eq:NSE_intro} to encompass a much larger class of SPDEs.

\subsection{The abstract formulation and the role of criticality}
\label{ss:abstract_intro}
Consider the following abstract SPDE:
\begin{equation}
\label{eq:SEE_intro}
\dd u + A u \,\dd t = F(\cdot,u)\,\dd t + Bu \,\dd W, \qquad u(0)=u_0,
\end{equation}
in a Banach space $X_0$. Here, $A $ and $B$ are jointly parabolic linear operators defined on $X_1\subseteq X_0$, $F$ is a given nonlinearity, and $W$ is a cylindrical Gaussian noise (see Subsection \ref{ss:critical_stochastic_evol} for details). 
SPDEs of the form \eqref{eq:SEE_intro} have been widely studied in the literature. Here, the setting of critical spaces of 
\cite{AV19_QSEE1,AV19_QSEE2,AV25_survey} (see also \cite{addendum,CriticalQuasilinear} for the deterministic setting) plays a central role.
The 3D NSEs \eqref{eq:NSE_intro} fit into the setting, see either \cite[Section 4]{AV21_NS} or Section \ref{s:NS}.

As explained at the beginning of \cite[Section 4]{AV25_survey}, in the critical space approach to abstract SPDEs \eqref{eq:SEE_intro}, paths of solutions to \eqref{eq:SEE_intro} typically belong to 
\begin{equation}
\label{eq:path_spaces_for_SEE}
L^p_{\loc}([0,\tau),t^\a \,\dd t ;X_1)\cap C([0,\tau);\Xap) \ \ \text{ with }  \ \  \tau>0,
\end{equation}
where $p\in (2,\infty)$ denotes the time integrability, $\a\in [0,\frac{p}{2}-1)$ the time weight, and $\Xap$ is the space of the initial data $u_0$ (so called \emph{trace spaces}) and is given by
\begin{equation}
\label{eq:trace_space}
\Xap\stackrel{{\rm def}}{=}(X_0,X_1)_{1-\frac{1+\a}{p},p}. \quad \text{ (Real interpolation)}
\end{equation}
Although it might not be clear at the moment, the $L^p$-weighted theory is crucial in the proof of Theorem \ref{t:NSE_intro}; see Remarks \ref{rem:necessity_Lp_setting} and \ref{rem:necessity_time_weights}. 

\smallskip

In the context of SPDEs of the form \eqref{eq:SEE_intro}, criticality is understood in an abstract sense, as an optimal `balance' between the regularity of the space $\Xap$ in which local well-posedness holds, and the roughness of the nonlinearity $F$. The latter is encoded in the following local Lipschitz assumption: 

There exist $\rho>0$ and $\beta\in (0,1)$ such that, for all $v,v'\in X_1$,
\begin{equation}
\label{eq:condition_F}
\|F(v)-F(v')\|_{X_0}\lesssim (1+\|v\|_{X_{\beta}}^\rho+\|v'\|_{X_\beta}^{\rho})\|v-v'\|_{X_\beta} ,
\end{equation}
where $X_\beta=[X_0,X_1]_\beta$ (complex interpolation). Note that $\rho$ and $\beta$ determine the growth and space roughness of $F$, respectively. The balance between the regularity of initial data and the roughness of $F$ is expressed as (see e.g., \cite[Section 4]{AV25_survey})
\begin{equation}
\label{eq:crtiticality_condition}
\frac{1+\a}{p}\leq \frac{\rho+1}{\rho}(1-\beta).
\end{equation}
This condition imposes a lower bound on the smoothness of the trace space \eqref{eq:trace_space} in terms of the regularity of the nonlinearity \eqref{eq:condition_F}, i.e., $1-\frac{1+\a}{p}\geq 1-\frac{\rho+1}{\rho}(1-\beta)$. In particular, the balance (or \emph{criticality}) is attained when the equality holds in \eqref{eq:crtiticality_condition}:
\begin{equation}
\label{eq:critical_roughness}
1-\frac{\rho+1}{\rho}(1-\beta). \quad \text{ (Critical regularity)}
\end{equation}
It is now well-established that, if the condition \eqref{eq:crtiticality_condition} holds for a concrete SPDE such as \eqref{eq:NSE_intro}, then the corresponding trace space is \emph{critical} from a PDE point of view. The reader is referred to \cite{AV21_NS} for the NSEs case, where the above condition captures the scaling \eqref{eq:invariant_maps_NSE}, and \cite{PW18} for the deterministic setting. Further examples can be found in \cite{Primitive3,Primitive1,Primitive2,AV19_QSEE1,RD_AV23,AV24_variational,AV25_survey,CriticalQuasilinear}.

\smallskip

Coming back to the study of singular times for \eqref{eq:SEE_intro}, let $\xx$ be a Banach space, and let $u:[0,\infty)\times \O\to \xx$ be a progressively measurable process, for which there exists $\ell\geq 1$ such that the following holds:
\begin{equation}
\label{eq:energy_bound_xx_intro}
\E\int_0^T\|u\|_{\xx}^\ell\,\dd t<\infty\ \text{ for all }T<\infty. \quad \text{(Energy bound)}
\end{equation}
The definition of singular times for the process $u$ is similar to the one given below \eqref{eq:condition_singular_times} in which one replaces the space \eqref{eq:reguality_of_strong_solutions} with the natural path space for \eqref{eq:SEE_intro}, i.e., \eqref{eq:path_spaces_for_SEE}; see Definition \ref{def:singular_regular}.
At first sight, this choice might look inconvenient, as singular times for \eqref{eq:SEE_intro} naturally depend on the choice of the `setting' 
$$
\set\stackrel{{\rm def}}{=}\setfull.
$$ 
However, in applications to SPDEs such as \eqref{eq:NSE_intro}, by parabolic regularization (see Subsection \ref{sss:independence_set}), one can check that the corresponding notion of singular times is independent of the specific choice of the setting $\set$; see Subsections \ref{sss:proof_singular_times_SNS_1} and \ref{sss:proof_singular_times_SNS_2} for the case of stochastic NSEs \eqref{eq:NSE_intro}.

\smallskip

To investigate singular times for $u$ satisfying the energy bound \eqref{eq:energy_bound_xx_intro}, we provide a convenient abstraction of the argument in Subsection \ref{sss:towards_abstract_NSE}. Assume the existence of $p\in (2,\infty)$ and $\a\in [0,\frac{p}{2}-1)$ such that 
\begin{equation}
\label{eq:xx_trace}
\xx\embed \Xap \text{ and \eqref{eq:SEE_intro} is \emph{locally well-posed} on $\Xap$}.
\end{equation}
An example of the above is given by the 3D NSEs \eqref{eq:NSE_intro}, where \eqref{eq:energy_bound_xx_intro} holds with $\r=2$ and $\xx=H^1(\T^3;\R^3)$, which is smoother than the critical one for the corresponding SPDE; see the comments below \eqref{eq:invariant_maps_NSE}.  Therefore, from \eqref{eq:critical_roughness} and \eqref{eq:xx_trace}, it is natural to define the \emph{excess from the criticality} of $\xx$ in the setting $\set$ as  
\begin{equation}
\label{eq:excess_criticality_intro}
\gap_\set\stackrel{{\rm def}}{=} 
 \underbrace{\Big(1-\frac{1+\a}{p}\Big)}_{\text{Regularity of }\xx}
- \underbrace{\Big(1- \frac{\rho+1}{\rho}(1-\beta)\Big)}_{\text{Critical regularity}}
=\frac{\rho+1}{\rho}(1-\beta)-\frac{1+\a}{p}.
\end{equation}
In applications to concrete SPDEs, the excess $\gap_\set$ is independent of the choice of $\set$; see Subsections \ref{sss:proof_singular_times_SNS_1} and \ref{sss:proof_singular_times_SNS_1} and Figure \ref{fig:dimension_bounds} for its computation in the case of the stochastic 3D NSEs \eqref{eq:NSE_intro}. 

\smallskip

The following illustrates our main result on \eqref{eq:SEE_intro}. For the Hausdorff and Minkowski dimensions, and measure/content $\H$ and $\M$, see Subsection \ref{ss:fractal}.

\begin{theorem}[Bounds on singular times: Abstract formulation -- Informal version of Theorem \ref{t:singular_times_SPDEs}]
\label{t:abstract_intro}
Let $u$ be a stochastic process satisfying \eqref{eq:energy_bound_xx_intro}, and assume that \eqref{eq:xx_trace} holds. Moreover, assume that $u$ has the \emph{strong weak-strong uniqueness} property with respect to strong solutions of \eqref{eq:SEE_intro} in the $\set$-setting (see Assumption \ref{ass:abstract2}).
Suppose that the excess of $\xx$ from criticality $\gap_\set$ as defined in \eqref{eq:excess_criticality_intro} satisfies
\begin{align}
\label{eq:spatial_subcriticality}
\gap_\set&> 0, \qquad \text{{\normalfont{ (Spatial subcriticality)}}}\\
\label{eq:spatial_subcriticality2}
\gap_{\set}&<\tfrac{1}{\r}. \qquad \text{{\normalfont{  (Space-time supercriticality)}}}
\end{align}
Then, for all $\varepsilon\in (0,1)$, 
\begin{align*}
\dim_{\M}(\si^\varepsilon)&\leq 1-\r\,\gap_\set, \quad 
\M^{1-\r\,\gap_\set}(\si^\varepsilon)=0, & & \text{\normalfont{ (Minkowski/Box-counting)}}\\
\dim_{\H}(\si)&\leq 1-\r\,\gap_\set, \quad \
\H^{1-\r\,\gap_\set}(\si)=0.& &  \text{ \normalfont{(Hausdorff)}}
\end{align*}
\end{theorem}

The above yields the bounds on the set of singular times of solutions to the stochastic 3D NSEs \eqref{eq:NSE_intro} in Theorem \ref{t:abstract_intro} (see Figure \ref{fig:dimension_bounds}).
Let us point out that the process $u$ is connected to the abstract SPDE \eqref{eq:SEE} only via the strong weak-strong uniqueness property.

\smallskip

\begin{figure}[htbp]
\centering
\renewcommand{\arraystretch}{2.1}
\makebox[\textwidth][c]{%
    \scalebox{0.95}{
        \begin{tabular}{|c|c|c|c|c|c|}
        \hline
        Thm \ref{t:NSE_intro} & {\renewcommand{\arraystretch}{1.0}\begin{tabular}{@{}c@{}}Energy \end{tabular}} & {\renewcommand{\arraystretch}{1.0}\begin{tabular}{@{}c@{}}Integrability \\ in time $\r$ \end{tabular}} & {\renewcommand{\arraystretch}{1.0}\begin{tabular}{@{}c@{}}Spatial \\ regularity \end{tabular}} & {\renewcommand{\arraystretch}{1.0}\begin{tabular}{@{}c@{}}Excess from \\ criticality $\gap$\end{tabular}} & {\renewcommand{\arraystretch}{1.0}\begin{tabular}{@{}c@{}}Singular time \\ dim.\ $\leq 1-\r\,\gap$\end{tabular}} \\
        \hline
        \eqref{it:NSE_intro1} &  $L^2_t(H^1(\T^3))$ & 2 & $\displaystyle{1-\frac{3}{2} }$  & $\displaystyle{\frac{1}{2}\Big(-\frac{1}{2}+1\Big)}$ & $\displaystyle{ \frac{1}{2}}$ \\[0.2cm] 
        \hline
        \eqref{it:NSE_intro2} & $L^{p_0}_t(L^{q_0}(\T^3))$ & $p_0$ & $\displaystyle{-\frac{3}{q_0}}$ & $\displaystyle{\frac{1}{2}\Big(-\frac{3}{q_0}+1\Big)}$ & $ \displaystyle{\frac{p_0}{2}\Big(\frac{2}{p_0}+\frac{3}{q_0}-1\Big) }$ \\[0.2cm] 
        \hline
        \end{tabular}%
    }
} 
\vspace{0.1cm} 
\caption{Derivation of Theorem \ref{t:NSE_intro} from Theorem \ref{t:abstract_intro}. We used that the critical Sobolev threshold for 3D NSEs is equal to $-1$ (see below \eqref{eq:invariant_maps_NSE}), and the factor $\frac{1}{2}$ in the excess formula $\gap$ is because in \eqref{eq:NSE_intro} the leading operators are of second order (or, in other words, parabolic scaling with time counted as the unit).}
\label{fig:dimension_bounds}
\end{figure}

As mentioned above, \eqref{eq:spatial_subcriticality} roughly ensures that $\xx$ has \emph{more} regularity than the critical regularity; see \eqref{eq:excess_criticality_intro}. Meanwhile, the condition \eqref{eq:spatial_subcriticality2}, which is equivalent to $1-\r\,\gap_\set>0$, precisely identifies the regime in which the energy space $L^\ell(\xx)$ in \eqref{eq:energy_bound_xx_intro}-\eqref{eq:xx_trace} has \emph{less} regularity than the critical one  (again, in terms of space-time Sobolev index\footnote{The space-time Sobolev index of $L^p(0,T;(X_0,X_1)_{\sigma,r})$ is given by $\sigma-1/p$.}, or space-time scaling):
\begin{equation}
\label{eq:supercriticality_Lellbound}
\frac{1}{\r}- \gap_\set = 
\underbrace{\Big(1- \frac{\rho+1}{\rho}(1-\beta)\Big)}_{\text{Critical regularity}}
-\Big[ \underbrace{-\frac{1}{\r}+\Big(1-\frac{1+\a}{p}\Big)}_{\text{Regularity of }L^\r(\xx)}\Big]
>0.
\end{equation}
The previous condition is natural in our context, as if $\frac{1}{\r}\leq \gap_\set$, then $L^\ell(\xx)$ has larger or equal regularity compared to the critical one for \eqref{eq:SEE_intro}, and in this situation, it is expected that $u$ is a \emph{global strong} solution to \eqref{eq:SEE_intro} in the $\set$-setting by Serrin-type blow-up criteria, see e.g., \cite[Theorem 4.11]{AV19_QSEE2}. 
For the sake of comparison, coming back to the NSEs \eqref{eq:NSE_intro} once more, the condition $\frac{1}{\r}\leq \gap_\set$ turns out to correspond to $\frac{2}{p_0}+\frac{3}{q_0}\leq 1$ (see Section \ref{s:NS}), which is the well-known Serrin's criteria for global regularity of Leray-Hopf solutions to 3D NSEs, see \cite[Theorem 12.4]{LePi} and \cite[Theorem 2.9]{AV21_NS} for the deterministic and stochastic case, respectively.

\smallskip

Recalling that the Hausdorff measures $\H^0$ and $\H^1$ coincide with the counting and the one-dimensional Lebesgue measures, respectively, Theorem \ref{t:abstract_intro} covers all the intermediate situations between the following two extreme cases (see Figure \ref{fig:regularity_regimes}): 
\begin{itemize}
\item Global irregularity -- Energy bound \eqref{eq:energy_bound_xx_intro} with $\gap_\set\leq 0$.
\item Global regularity -- Energy bound \eqref{eq:energy_bound_xx_intro} with $\gap_{\set}\geq \tfrac{1}{\r}$.
\end{itemize}
To conclude, we point out that Theorem \ref{t:singular_times_SPDEs_2} provides an extension of Theorem \ref{t:abstract_intro} to the case $\gap_\set=0$ (i.e., $\xx$ is critical). However, the proof of the latter is simpler.

\begin{figure}[htbp]
    \centering
    \begin{tikzpicture}[>=stealth, font=\small, transform shape]
        \def\xmax{6.5}
        \def\ymax{5}
        \def\xmin{-3}
        \def\ymin{-0.2} 
       
        \fill[gray!60] (\xmin, 0) rectangle (0, \ymax);
        \node at (\xmin/2, \ymax/2) [align=center, font=\sffamily, text=white] {Global\\irregularity\\[1ex] $\gap_\set \leq 0$};

        \begin{scope}
            \clip (0,0) rectangle (\xmax, \ymax);
            \fill[gray!20] (0,0) -- (\xmax, 0) -- (\xmax, \xmax) -- cycle;
        \end{scope}
        
        \node at (0.6*\xmax, 0.3*\ymax) [align=center, font=\sffamily] {Global\\ regularity\\[1ex] $\gap_\set \geq \tfrac{1}{\r}$};

        \begin{scope}
            \clip (0,0) -- (\ymax, \ymax) -- (0, \ymax) -- cycle;
            \foreach \x in {0.3, 0.6, ..., 4.8} {
                \draw[black!50, thin] (0,0) -- (\x, \ymax);
            }
        \end{scope}
        
        \node[fill=white, inner sep=2pt, rounded corners=1pt, font=\sffamily\bfseries] 
            at (0.35*\ymax, 0.7*\ymax + 0.25) {Partial regularity};
            
        \node[fill=white, inner sep=2pt, rounded corners=1pt, font=\sffamily\bfseries] 
            at (0.35*\ymax, 0.7*\ymax - 0.25) {(this paper)};

        \draw[->, thick] (4, \ymax + 0.3) -- (1, \ymax + 0.3) node[midway, above] {$1-\r\, \gap_\set$ increases};

        \draw[thick, gray!20] (0,0) -- (\ymax, \ymax);

        \draw[->, thick] (\xmin, 0) -- (\xmax + 0.5, 0) node[right] {$\gap_\set$};
        
        \draw[->, thick] (0, \ymin) -- (0, \ymax + 0.5) node[above] {$1/\r$};

        \draw[thick] (2pt, \ymax) -- (-2pt, \ymax) node[above left] {1};

        \node[below left] at (0,0) {0};

    \end{tikzpicture}
    \caption{The striped region is the area of applicability of Theorem \ref{t:abstract_intro}. Here, $\ell$ is as in \eqref{eq:energy_bound_xx_intro}, and $\gap_\set$ is the excess of (spatial) regularity of the latter bound over the critical threshold; see \eqref{eq:excess_criticality_intro}. }
    \label{fig:regularity_regimes}
\end{figure}

\subsection{Pathwise approach, space-time singular sets and further applications}
\label{ss:comparison}
It is natural to ask whether singular and regular times for the stochastic 3D NSEs \eqref{eq:NSE_intro} (or, more generally, for the abstract SPDE \eqref{eq:SEE_intro}) can be investigated pathwise, rather than through the annealed notion in \eqref{eq:condition_singular_times}--\eqref{eq:singular_sets_def_intro}. There are several challenges in pursuing such an approach. Indeed, the restarting argument outlined below \eqref{eq:H1_regularity_NSEs} would require restarting strong solutions simultaneously for almost every initial time along almost every noise realization, together with sufficiently sharp lower bounds on their lifetimes to recover the correct dimensional exponent (cf.\ Proposition \ref{prop:quantitative_lower_bound} and the proof of Theorem \ref{t:singular_times_SPDEs}). Although rough-path techniques have been recently applied to 3D NSEs for certain spatially smooth transport noises \cite{HLN19,HLN21}, the required well-posedness and sharp estimates on the lifetime of solutions are currently not available and are very difficult to obtain with the existing techniques. 

\smallskip

From the definitions \eqref{eq:smoothness_regular_times}, \eqref{eq:condition_singular_times}-\eqref{eq:singular_sets_def_intro} (or more generally, Definition \ref{def:singular_regular}), it is clear that being a regular/singular time involves a condition that is local in time, but \emph{global} in space. 
To some extent, this conflicts with the parabolic scaling, as the spatial direction becomes the favorite one when measuring smoothness. This is also visible in Theorem \ref{t:NSE_intro}\eqref{it:NSE_intro2} where, by taking $p_0\to \infty$ and $q_0=\frac{3p_0}{p_0-1}$, one has $\delta_0=\frac{p_0}{2}(\frac{2}{p_0}+\frac{3}{q_0}-1)\equiv \frac{1}{2}$ as well as $\frac{2}{p_0}+\frac{3}{q_0}\to 1$. In particular, it is possible to get arbitrarily close to the Serrin condition for global smoothness of Leray-Hopf solutions, while keeping the bound on the singular uniformly bounded from below. 

Space-time sets of (possible) singularities for deterministic PDEs are also well-studied in the literature; see e.g.,  \cite{Currents2}. In the context of 3D NSEs, space-time variants of singular sets were first studied by Scheffer in \cite{Scheffer_partial_regularity}, and afterwards refined in the celebrated work by Caffarelli, Kohn, and Nirenberg \cite{CKN82}, where the authors proved that the parabolic Hausdorff dimension of the space-time singular set of so-called \emph{suitable} solutions (see e.g., \cite[Definition 13.5]{LePi}) to the 3D NSEs has dimension $\leq 1$ with a null corresponding measure (see also \cite{L98_1}).
This result refines the one by Leray and Scheffer on singular times, but it concerns the more restrictive class of suitable solutions rather than the strong Leray-Hopf solutions to 3D NSEs; see above \eqref{eq:annealed_energy_inequality_intro} or \cite[Proposition 12.1]{LePi}. 

The result of Caffarelli-Kohn-Nirenberg was later extended in many directions (e.g., by replacing the Hausdorff by the Minkowski/box-counting dimension), and it is still an active line of research; see e.g., \cite{Breit_partial,CY19_1,CDLM20,KY16_1,K09_singular_times,K09_singular_times2,RS2_09,WW24_1}.

\smallskip

In the stochastic setting, to the best of our knowledge, the only partial regularity for an SPDE was obtained by Flandoli and Romito in \cite{FR_partial}, where, exploiting an extension of the Caffarelli-Kohn-Nirenberg to the stochastic setting, they proved that for stationary solutions to the stochastic 3D NSEs with additive noise, for each fixed (deterministic) time, the set of spatial singularities is a.s.\ empty. This, in a certain sense, can be seen as an improvement of the deterministic theory. 

Unfortunately, the approach in \cite{FR_partial} cannot be generalized to SPDEs with \emph{multiplicative noise}, as their approach relies on the fact that, in the case of additive noise, the corresponding SPDE can be reduced to a random PDE by subtracting the solution to the linear problem with the same additive noise (often referred to as the ' `Da Prato-Debussche trick'). Moreover, this reduction allows for a pathwise analysis, which is not possible for \eqref{eq:NSE_intro}.

\smallskip

In the case of multiplicative noise, there are several difficulties in extending the Caffarelli-Kohn-Nirenberg theory to the case of stochastic NSEs \eqref{eq:NSE_intro} in the presence of a \emph{transport}-type noise. 
Indeed, in this case, the existence of \emph{suitable} weak solutions (see e.g.,  \cite[Definition 13.5]{LePi}) is not known in the stochastic setting; partial results can be found in recent works \cite{B08_local_energy,CD_suitable_Weak}. 
From an analytic point of view, the main difficulty lies in proving certain iteration lemmas over families of shrinking balls (see e.g., \cite[Subsection 13.9]{LePi}), which are central to understanding the energy decay. 
There are serious obstructions in extending this to the case of multiplicative noise of transport type, and therefore it seems out of reach of the current techniques (related issues appear in the context of De Giorgi-Nash-Moser estimates in which analogous arguments are needed, see e.g., \cite{ASV25,DG17_continuity,HWZ17}).

\smallskip

Finally, we discuss further applications of Theorem \ref{t:abstract_intro} to obtain bounds on singular times for weak solutions to other SPDEs. It will be clear from the proof of Theorem \ref{t:NSE_intro} that the core arguments are widely applicable. Possible examples include the 3D Boussinesq systems, hyper- and hypoviscous NSEs, magnetohydrodynamics, surface growth models \cite{BR09_growth,BR15_growth,OR19}, and reaction-diffusion equations with supercritical $L^\zeta$-coercivity (i.e., in the setting of \cite[Theorem 3.2]{AV23} but with the coercivity parameter $\zeta$ in the supercritical range), or with entropy-dissipation \cite{F15_renormalized,F17_weak_strong}. Among others, we emphasize that with our methods one can recover the results on the fractional NSEs by Jiu and Wang \cite{JW14}, and also obtain extensions of the latter results to the stochastic setting.

\subsection{Notation}
\label{ss:notation}

Here, we collect the basic notation used in the manuscript.  
For two quantities $x$ and $y$, we write $x\lesssim y$, if there exists a constant $C$ such that $x\le Cy$. If such a $C$ depends on the parameters $p_1,\dots,p_n$ we either mention it explicitly or indicate this by writing $C_{p_1,\dots,p_n}$ and correspondingly $x\lesssim_{p_1,\dots,p_n}y$ whenever $x\le C_{p_1,\dots,p_n}y$. We write $x\eqsim_{p_1,\dots,p_n} y$ whenever $x\lesssim_{p_1,\dots,p_n} y$ and $y\lesssim_{p_1,\dots,p_n}x$.

\smallskip

\emph{Probabilistic setting.}
Below, $(\O, \mathcal{A},(\mathcal{F}_t)_{t\geq 0}, \P)$ denotes a filtered probability space carrying a sequence of independent standard Brownian motions which changes depending on the SPDE under consideration, and $\E[\cdot]=\int_{\O} \cdot \,\dd \P$ for the associated expected value. Throughout the paper, we assume that $\F_0$ contains all $\P$-null sets in $\mathcal{A}$.
A process $\phi:[0,\infty)\times \O\to X$ is progressively measurable if $\phi|_{[0,t]\times \O}$ is $\Borel([0,t])\otimes \F_t$-measurable for all $t\geq 0$, where $\Borel$ is the Borel $\sigma$-algebra on $[0,t]$ and $X$ a Banach space. Moreover, a stopping time $\tau$ is a measurable map $\tau:\O\to [0,\infty]$ such that $\{\tau\leq t\}\in \F_t$ for all $t\geq 0$. Finally, a stochastic process $\phi:[0,\tau)\times \O\to X$ is progressively measurable if $\one_{[0,\tau)\times \O}\,\phi$ is progressively measurable where $
[0,\tau)\times \O\stackrel{{\rm def}}{=}\{(t,\om)\in[0,\infty)\times \O\,:\,\,0\leq t<\tau(\om)\}$
and $\one_{[0,\tau)\times \O}$ (or simply $\one_{[0,\tau)}$) stands for the extension by zero outside $[0,\tau)\times \O$. The definitions of the stochastic intervals $(0,\tau)\times \O$ and $[0,\tau]\times \O$ are similar.

\smallskip

\emph{Interpolation.} For $\vartheta\in (0,1)$, $p\in (1,\infty)$, $(\cdot,\cdot)_{\vartheta,p}$ and $[\cdot,\cdot]_{\vartheta}$ are the real and complex interpolation functors, respectively; see e.g., \cite{BeLo,InterpolationLunardi,Analysis1}.

\smallskip

\emph{Function spaces.}
Let $X$ be a Banach space. 
We write $L^p(S,\mu;X)$ for the Bochner space of strongly measurable, $p$-integrable $X$-valued functions for a measure space $(S,\mu)$ and $p\in (1,\infty)$, see e.g., \cite[Section 1.2b]{Analysis1}. As usual, $\one_A$ denotes the indicator function of $A\subseteq S$. 

Fix $\a\in \R$ and $t_0\in \R$. We denote by $w^{t_0}_{\a}$ the associated shifted power weight:
$$
w_\a^{t_0}(t)\stackrel{{\rm def}}{=}|t-t_0|^{\a} \qquad \text{ and }\qquad w_\a(t)\stackrel{{\rm def}}{=}w_\a^0(t)=|t|^\a,
$$ 
for $t\in \R$.
If $S=(t_0,t)$ for some $-\infty \leq t_0<t\leq \infty$ and $\mu= w_\a^{t_0} \,\dd x $, we simply write either $L^p((t_0,t),w_{\a}^{t_0} ;X )$ or  $L^p(t_0,t,w_{\a}^{t_0} ;X )$ instead of $L^p((t_0,t),w_{\a}^{t_0}\,\dd t ;X )$.

For $T\in (0,\infty]$ and $p\in (1,\infty)$, we denoted by $W^{1,p}(0,T,w_\a;X)$ the set of all $f\in L^p(0,T,w_{\a};X)$ such that $f'\in L^p(0,T,w_{\a};X)$ endowed with the natural norm, see \cite[Section 2.5]{Analysis1} for distributional derivative of $X$-valued maps. Moreover, for $\vartheta\in (0,1)$, we define the Bessel-potential space $H^{\vartheta,p}(0,T,w_{\a};X)$ as
$$ 
H^{\vartheta,p}(0,T,w_{\a};X)\stackrel{{\rm def}}{=}[L^p(0,T,w_{\a};X),W^{1,p}(0,T,w_{\a};X)]_{\vartheta},
$$
As usual, for $I\subseteq [0,\infty)$, we say $f\in H^{\vartheta,p}_{\loc}(I;X)$ if $f\in H^{\vartheta,p}(J;X)$ for all compact sets $J\subseteq I$ (a similar notation is employed if $H^{\vartheta,p}$ is replaced by either $W^{1,p}$ or $L^p$). 

Finally, for all $\vartheta_0,\vartheta_1>0$ and $-\infty< s<t< \infty$, we let 
$$
C^{\vartheta_0,\vartheta_1}((s,t)\times \T^d)\stackrel{{\rm def}}{=}C^{\vartheta_0}(s,t;C(\T^d))\cap C([s,t];C^{\vartheta_1}(\T^d))
$$ 
and $C^{\vartheta_0,\vartheta_1}_{\loc}((s,t)\times \T^d)\stackrel{{\rm def}}{=}\cap_{\varepsilon>0}C^{\vartheta_0,\vartheta_1}((s+\varepsilon,t-\varepsilon)\times \T^d)$.

\section{Preliminaries}

\subsection{Fractal measures and dimensions}
\label{ss:fractal}
Here we collect basic facts on Hausdorff measures and Minkowski contents, which will be needed in the following. We emphasize that we restrict ourselves to the one-dimensional case, as it is the one needed here. For a more detailed treatment, the reader is referred to, e.g., \cite{EG_measure,Falconer_book} and \cite[Chapter 1]{Currents1}.
  
For $A\subseteq \R$, $\eta>0$ and $s\in [0,1]$, we define the $s$-dimensional $\eta$-pre-measure as:
\begin{equation}
\label{eq:pre_measure_HS}
\H^s_\eta(A)\stackrel{{\rm def}}{=}
\inf\Big\{{\textstyle{\sum_{j\in J}}}\,  \big(\mathrm{diam}(I_j)\big)^{s}\,:\,\text{ with } A\subset \textstyle{\bigcup_{j\in J}}I_j \text{ and }\mathrm{diam} (I_j)<\eta \Big\}
\end{equation}
where the $\inf$ is taken over all possible at most countable coverings of $A$ by arbitrary subsets $I_j$ of $\R$, and $\mathrm{diam} (I_j)\stackrel{{\rm def}}{=} \sup_{x,y\in I_j}|x-y|$ its diameter. 

\smallskip

The (outer) $s$-dimensional \emph{Hausdorff measure} on $\R$ is defined as:
\begin{equation}
\label{eq:measure_HS}
\H^s(A)=\lim_{\eta\downarrow 0}\H^s_\eta(A).
\end{equation}
Note that the limit exists for all subsets $A$ of $\R$ as the assignment $\eta\mapsto \H^s_{\eta}(A)$ is non-increasing. As noticed in \cite[p.\ 82]{EG_measure}, the limit $\eta\downarrow 0$ forces the coverings in the definition of $\H^s_\delta(A)$ to follow the geometry of $A$. 

It is routine to show that the measure $\H^0$ agrees with the counting measure \cite[p.\ 14]{Currents1}. While a deep result from real analysis yields that $\H^1$ coincides with the one-dimensional (outer) Lebesgue measure $|\cdot|$, see e.g., \cite[Theorem 2.5]{EG_measure}.

\smallskip

The following result is a straightforward consequence of the definition of the Hausdorff outer measure $\H^s$; see e.g.,   \cite[Lemma 2.2]{EG_measure} or \cite[p.\ 14]{Currents1}, and is the key to introducing fractional dimensions.

\begin{lemma}
\label{l:basic_property_H}
Let $A\subseteq \R$ be a set, and $0\leq s<t\leq 1$.
\begin{itemize}
\item
If $\H^s(A)<+\infty$, then $\H^{t}(A)=0$.
\item 
If $\H^t(A)>0$, then $\H^{s}(A)=+\infty$.
\end{itemize}
\end{lemma}

Given the above result, the \emph{Hausdorff dimension} of a set $A\subseteq \R$ is denoted by 
$\dim_{\H}(A)$ and defined by the infimum over $s\in [0,1]$ for which $A$ has null $s$-dimensional Hausdorff measure, i.e.,
\begin{equation}
\label{eq:Hausdorff_dimension}
\dim_{\H}(A)\stackrel{{\rm def}}{=}\inf\{s\in [0,1]\,:\, \H^s(A)=0\}\ \  \text{ where } \ \ 
\inf\emptyset \stackrel{{\rm def}}{=}1.
\end{equation}
In particular, if $\dim_{\H}(A)>0$, then the mapping $s\mapsto \H^s(A)$ transitions from $+\infty$ to $0$ at the value $\dim_{\H}(A)$. Note that $\H^{\dim_\H(A)}(A)\in [0,\infty]$, where the extreme values cannot be excluded.

\smallskip

Next, we discuss the upper Minkowski contents and related dimensions, also known as the box-counting dimension. 
For a \emph{bounded} set $A\subseteq \R$ and $\eta\in (0,1)$, let $\NN(A,\eta)$ denote the infimum number of balls of radius $\eta$ needed to cover $A$. 
The upper $s$-dimensional Minkowski content of the bounded set $A$ is defined as:
\begin{equation}
\label{eq:Feta_definition0}
\M^s (A)
\stackrel{{\rm def}}{=}\limsup_{\eta\downarrow 0} \eta^s \NN(A,\eta).
\end{equation}
In this manuscript, we often deal with possibly unbounded sets. Therefore, a natural extension of the Minkowski content is as follows. We define the upper $s$-dimensional Minkowski content of \emph{any} $A\subseteq \R$ as
\begin{equation}
\label{eq:Feta_definition}
\M^s(A)
\stackrel{{\rm def}}{=}\sup_{0<t<\infty}\M^s(A\cap (-t,t)),
\end{equation}
where the terms in the supremum are defined as in \eqref{eq:Feta_definition0}.
Clearly, the above formula is consistent with the previous definition in the case of bounded sets.

From \eqref{eq:pre_measure_HS}, it follows that for any bounded set $A\subseteq \R$ and $s\in [0,1]$,  it holds that $\H^s_\eta(A)\leq \eta^s \NN(A,\eta)$ and therefore
\begin{equation}
\label{eq:HF_comparison}
\H^s(A)\leq 2^s\M^s(A) .
\end{equation}
The reverse inequality is, in general, \emph{false}.
Moreover, one can check that a variant of Lemma \ref{l:basic_property_H} also holds for $\H^s$ replaced by $\M^s$. Thus, for a bounded set $A\subseteq \R$, we can define the \emph{upper Minkowski dimension} as 
\begin{equation}
\label{eq:upper_box_dim}
\dim_{\M}(A)
\stackrel{{\rm def}}{=}\inf\{s\in [0,1]\,:\, \M^s(A)=0\} \ \  \text{ where } \ \ 
\inf\emptyset \stackrel{{\rm def}}{=}1.
\end{equation}
One can check that the above formula coincides with the \emph{box-counting} dimension usually defined via the more standard formula $\limsup_{\eta \downarrow 0}\log \NN(A,\eta)/\log(1 /\eta)$; see \cite[Chapter 3]{Falconer_book} for further discussion. It follows from \eqref{eq:HF_comparison} that
$$
\dim_{\H}(A)\leq \dim_{\M}(A)\ \text{ for all bounded sets $A\subseteq \R$}.
$$
Interestingly, there are many examples of bounded sets where the upper Minkowski dimension is larger than the Hausdorff one \cite[Chapter 3]{Falconer_book}.

\begin{remark}[Lack of $\sigma$-subadditivity of the Minkowski content]
\label{r:lack_subadditivity}
A well-known limitation of the upper Minkowski content is that it is not countably subadditive (e.g., $\M^s(\Q\cap [0,1])>0$, while $\M^s(\{q\})=0$ for all $q\in \R$ and $s\in (0,1]$). This is the primary reason why our main results (cf.\ Theorems \ref{t:NSE_intro} and \ref{t:abstract_intro}) are formulated in terms of both Hausdorff and Minkowski dimensions and measures.
\end{remark}

\subsection{Stochastic maximal $L^p$-regularity and SPDEs in critical spaces}
\label{ss:critical_stochastic_evol}
The aim of this subsection is to give a brief introduction to stochastic maximal $L^p$-regularity and the theory of stochastic evolution equations in critical spaces as in \eqref{eq:SEE_intro}. Here, we mainly follow the exposition in \cite[Sections 3 and 4]{AV25_survey}. For a more comprehensive treatment, the reader is referred to \cite{AV19_QSEE1,AV19_QSEE2} (see also \cite{CriticalQuasilinear,addendum} for the deterministic case). Throughout this manuscript, we enforce the following condition. For UMD and type $2$ spaces, the reader is referred to, e.g., \cite[Chapter 5]{Analysis1} and \cite[Chapter 7]{Analysis2}.

\begin{assumption}[Setting]
The setting is denoted by 
$\set=\setfull$ where
\label{ass:setting}
\begin{itemize}
\item $X_0$ and $X_1$ are {\normalfont{UMD}} Banach spaces with type $2$, and $X_1\embed X_0$ densely.
\item $p\in [2,\infty)$ and $\a\in [0,\frac{p}{2}-1)\cup \{0\}$. 
\item $W$ is a cylindrical Brownian motion on a separable Hilbert space $H$ w.r.t.\ the filtered probability space $(\O,(\F_t)_t,\F,\P)$ (see e.g., \cite[Definition 2.11]{AV19_QSEE1}).
\end{itemize}
\end{assumption}

Recall that the UMD and type $2$ assumptions on the Banach spaces are needed to have suitable stochastic integration, see e.g., \cite{NVW1} and \cite[Theorem 4.7]{NVW13}. 
With the notation of Assumption \ref{ass:setting}, we let 
\begin{equation}
\label{eq:def_trace_space}
X_{\theta}\stackrel{{\rm def}}{=}[X_0,X_1]_{\theta}\quad \ \text{ and }\quad \ 
\Xap\stackrel{{\rm def}}{=}(X_0,X_1)_{1-\frac{1+\a}{p},p},
\end{equation}
be complex and real interpolation spaces, respectively; see Subsection \ref{ss:notation}.
The last ingredient needed below is the $\gamma$-radonifying operators denoted by $\g(H,X)$ for a Banach space $X$. In recent years, they have played a central role in stochastic integration \cite{NVW1,NVW13} and \cite[Subsection 2.5]{AV25_survey}, and they can be defined as follows. Let $(h_n)_n$ and $(\wt{\g}_n)_n$ be an orthonormal basis of $H$ and a sequence of standard independent Gaussian random variables on a probability space $(\wt{\O},\wt{\A},\wt{\P})$, respectively. 
A bounded linear operator $T:H\to X$ belongs to $\g(H,X)$ if $\sum_{n} \wt{\g}_n  T h_n $ converges in $L^2(\wt{\O};X)$ and we let 
$$
\textstyle
\|T\|_{\g(H,X)}=\big(\wt{\E}\|\sum_{n}\wt{\g}_n T h_n\|_{X}^2\big)^{1/2}.
$$
A detailed treatment of $\g(H,X)$ can be found in \cite[Chapter 9]{Analysis2}. Descriptions in case $L^q$-spaces are given in \cite[Theorem 9.4.8]{Analysis2} (see also \cite[Proposition A.2]{AgrSau}).

\subsubsection{Stochastic maximal $L^p$-regularity}
Here, we briefly comment on stochastic maximal $L^p$-regularity, which is an essential tool in the theory of stochastic evolution equations in critical spaces and serves as a linearization of \eqref{eq:SEE_intro}, cf.\ the text below \cite[Theorem 1.1]{AV25_survey}. Brief overview on this can be found in \cite[Section 3]{AV19_QSEE1} and \cite[Section 3]{AV25_survey}. 

Suppose that Assumption \ref{ass:setting} holds, and consider the linear abstract SPDE:
\begin{equation}
\label{eq:SMR_recap}
\dd v+ A v \,\dd t = f\,\dd t + (Bv + g)\,\dd W,\qquad v(t_0)=0.
\end{equation}
where $t_0\geq 0$ and, for $p\in [2,\infty)$ and $\a\in [0,\frac{p}{2}-1)\cup\{0\}$ as above,   
\begin{equation}
\label{eq:forcing_assumption_fg}
f\in L^p((t_0,\tau)\times\O,w_{\a}^{t_0};X_0 )\quad  \ \text{ and }\ \quad g \in L^p((t_0,\tau)\times\O,w_{\a}^{t_0};\g(H,X_{1/2}))
\end{equation}
are progressively measurable, $\tau:\O\to [t_0,\infty]$ is a given stopping time. Finally, the couple $(A,B)$ satisfies the following 

\begin{assumption}[Measurability and boundedness of $(A,B)$]
\label{ass:measurability_boundedness_AB}
The mappings are $\Progress$-strongly measurable (in the operator sense, see \cite[Definition 1.1.27]{Analysis1})
$$
A:\R_+\times \O\to \calL(X_1,X_0)\quad\  \text{ and }\ \quad B:\R_+\times \O\to \calL_2(X_1,\g(H,X_{1/2})).
$$ 
Moreover, $\|A\|_{\calL(X_1,X_0)} +\|B\|_{\calL(X_1, \g(H,X_{1/2}))}\leq M$ a.e.\ on $\R_+\times \O$ for some $M>0$.
\end{assumption}

The SPDE \eqref{eq:SMR_recap} is understood in the integral form (cf.\ Definition \ref{def:solutions_SEE} below). 
Assumptions \ref{ass:setting}, \ref{ass:measurability_boundedness_AB} and \eqref{eq:forcing_assumption_fg}, all the corresponding integrals are well-defined in case $v\in L^p_{\loc}([0,\infty),w_{\a};X_1)$ a.s.\ by e.g., \cite[Theorem 4.7]{NVW13}. 

\smallskip

We say that the couple $(A,B)$ has \emph{stochastic maximal $L^p$-regularity} in the setting $\set=\setfull$ if the following condition holds. For all $T\in (0,\infty)$, there exists a constant $C_0>0$ such that, for all $t_0\in (0,T)$, all stopping times $\tau:\O\to [t_0,T]$, and all progressively measurable process as in \eqref{eq:forcing_assumption_fg}, there exists $v\in C([t_0,\tau];\Xap)\cap L^p(t_0,\tau,w_{\a}^{t_0};X_1)$ a.s.\ that is a solution to \eqref{eq:SMR_recap}, and satisfies
\begin{align}
\label{eq:stochastic_maximal_Lp_regularity}
&\E \sup_{t\in [t_0,\tau]}\|v(t)\|_{\Xap}^p 
+ \E\int_{t_0}^{\tau} \|v(t)\|_{X_1}^p\,w_{\a}^{t_0}\,\dd t \\
\nonumber
&\qquad\quad \leq C_0\, \E\int_{t_0}^{\tau} \|f(t)\|_{X_0}^p\,w_{\a}^{t_0}\,\dd t 
+ C_0 \, \E\int_{t_0}^{\tau} \|g(t)\|_{\g(H,X_{1/2})}^p\,w_{\a}^{t_0}\,\dd t;
\end{align}
see \eqref{eq:def_trace_space} for $\Xap$. 
Let us point out that the above is slightly weaker than the stochastic maximal $L^p$-regularity as claimed in \cite[Definition 3.8]{AV25_survey}, where additional optimal fractional time-regularity is assumed (see also \cite[Proposition 2.1]{AV25_survey}). However, the above is enough for our purposes. Finally, we emphasize that the bound \eqref{eq:stochastic_maximal_Lp_regularity} together with the Assumption \ref{ass:measurability_boundedness_AB} also implies existence, uniqueness, and a corresponding estimate in case of non-trivial initial data $v(t_0)\in L^p_{\F_{t_0}}(\O;\Xap)$; see e.g., \cite[Proposition 3.10]{AV19_QSEE1} for details. 

\smallskip

There are several examples of a couple satisfying stochastic maximal $L^p$-regularity. In particular, it is worth mentioning the seminal works by Krylov \cite{Kry96,Kry,KryOverview} and the semigroup approach by Van Neerven, Veraar, and Weis \cite{MaximalLpregularity}. Further references can be found in \cite[Subsection 3.2]{AV19_QSEE1} and \cite[Sections 3.4--3.6]{AV25_survey}.
Finally, we mention that stochastic maximal $L^p$-regularity for the so-called `turbulent Stokes' couple appearing in the stochastic 3D NSEs \eqref{eq:NSE_intro} is proven in \cite[Section 3]{AV21_NS}.
 
\subsubsection{Stochastic evolution equations in critical spaces}
Here, we briefly comment on the local well-posedness for the SEE: 
\begin{equation}
\label{eq:SEE_critical}
\dd u +A u\,\dd t= F(\cdot,u)\,\dd t + (Bu + G(\cdot,u))\,\dd W,\qquad u(0)=u_{0},
\end{equation}
in the so-called \emph{critical setting} \cite{AV19_QSEE1,AV19_QSEE2}.  
The following condition rules the relation between the nonlinearities $F$ and $G$ in \eqref{eq:SEE_intro} and the order of the operators $(A,B)$ encoded in the spaces $X_0$ and $X_1$. 

\begin{assumption}[(Sub)criticality]
\label{ass:criticality}
Let $p\in [2,\infty)$ and $\a\in[0,\frac{p}{2}-1)\cup\{0\}$.
The following mappings are strongly progressively measurable
$$
F:\R_+\times \O\times X_1 \to X_0\quad \text{ and }\quad G:\R_+\times \O\times X_1\to \g(H,X_{1/2})
$$ 
and $\|F(\cdot,0)\|_{X_0}, \|G(\cdot,0)\|_{\g(H,X_{1/2})}\in L^\infty_{\loc}([0,\infty)\times \O)$. Moreover, there exist $m\geq 1$, positive numbers $(\rho_j)_{j=1}$, and $(\beta_j)_{j=1}^m\in (1-\frac{1+\a}{p},1)$ such that 
\begin{equation}
\label{eq:sub_critical_condition}
\frac{1+\a}{p}\leq \frac{\rho_j+1}{\rho_j}(1-\beta_j).
\end{equation}
and for all $v,v'\in X_1$, 
$$
\|F(v)-F(v')\|_{X_0}
+
\|G(v)-G(v')\|_{\g(H,X_{1/2})}
\lesssim \textstyle{\sum_{j=1}^m} (1+\|v\|_{X_{\beta_j}}^{\rho_j}+\|v'\|_{X_{\beta_j}}^{\rho_j})\|v-v'\|_{X_{\beta_j}}.
$$
\end{assumption}

In case the condition \eqref{eq:sub_critical_condition} is satisfied with equality, we say that the setting $\set$ is \emph{critical} for \eqref{eq:SEE_critical}. 
In applications to concrete SPDEs, the condition \eqref{eq:sub_critical_condition} has proved to capture the \emph{scaling} of the underlying SPDEs, see \cite{RD_AV23,AV24_variational,AV21_NS,AV25_survey}. In particular, the space for initial data in the corresponding local well-posedness results (see Theorem \ref{t:lwp} below) enjoys the scaling of the corresponding SPDE.

Next, we define solutions to \eqref{eq:SEE_critical} as in \cite[Definitions 4.5 and 4.6]{AV25_survey}. 

\begin{definition}[Local, unique and maximal in the $\set$-setting]
\label{def:solutions_SEE}
Suppose Assumptions \ref{ass:setting}, \ref{ass:measurability_boundedness_AB} and \ref{ass:criticality} hold.
Let $\sigma: \O\to [0,\infty]$ and $u:[0,\sigma)\times \O\to X_1$ be a stopping time\footnote{Here, $[0,\sigma)\times \O= [0,\sigma)\times \{\sigma>0\}$.} and a progressively measurable process, respectively.
\begin{itemize}
\item The pair $(u,\sigma)$ is called a strong solution to \eqref{eq:SEE_critical} in the $\set$-setting if 
\[u\in L^p(0,\sigma,w_{\kappa};X_1)\cap C([0,\sigma];\Xap) \text{ a.s., }\]
and the following identity holds a.s.\ for all $t\in [0,\sigma)$:
\begin{equation*}
\textstyle u(t) - u_0 + \int_0^t A u(s) \,\dd s = \int_0^t F(u(s)) \,\dd s + \int_{0}^t \one_{[0,\sigma]}(s) [B u(s) + G(u(s))] \,\dd  W(s).
\end{equation*}
\item The pair $(u,\sigma)$ is called a local solution to \eqref{eq:SEE_critical} in the $\set$-setting if there exists a sequence of stopping times $(\sigma_n)_n$ such that $\sigma_n\uparrow \sigma$ a.s.\ and $(u|_{[0,\sigma_n]\times \O},\sigma_n)$ is a strong solution to \eqref{eq:SEE_critical} in the $\set$-setting.
\item A local solution $(u,\sigma)$ to \eqref{eq:SEE_critical} in the $\set$-setting is called \emph{maximal} if for any other local solution $(v, \tau)$ to \eqref{eq:SEE_critical} in the $\set$-setting, it holds that a.s.\ $\tau\leq \sigma$ and $u = v$ on $[0,\tau)\times \O$.
\end{itemize}
\end{definition}

As commented below \cite[Definitions 4.5]{AV25_survey}, all the integrals appearing in the definition of local solutions are well-defined due to pathwise regularity of solutions. 

\begin{theorem}[Local well-posedness in the $\set$-setting]
\label{t:lwp}
Under the Assumptions \ref{ass:measurability_boundedness_AB}, \ref{ass:criticality}, and that $(A,B)$ has stochastic maximal $L^p$-regularity in the $\set$-setting, then for each $u_0\in L^0_{\F_0}(\O;\Xap)$ the abstract {{\normalfont{SPDE}}} \eqref{eq:SEE_critical} has a unique maximal solution $(u,\tau)$ in the $\set$-setting with $\tau>0$ a.s.
\end{theorem}

The previous is proven in \cite[Theorem 4.8]{AV19_QSEE1} (for further comments, see \cite[Section 4]{AV25_survey}). In Section \ref{s:quantifying_main_proof}, we will sharpen some of the arguments in \cite{AV19_QSEE1}. In particular, we partially obtain a self-contained proof of \cite[Theorem 4.8]{AV19_QSEE1}.

Finally, throughout this work, we say that \emph{local well-posedness of \eqref{eq:SEE} in the $\set$-setting holds} if the assumptions of Theorem \ref{t:lwp} hold.

\section{Singular times for SPDEs and their fractal dimensions}
\label{s:singular_times_SPDEs}
In this section, we introduce sets of singular times for stochastic processes, and we establish bounds on their Hausdorff and Minkowski dimensions. As in Theorem \ref{t:abstract_intro}, the key assumptions are that such stochastic processes satisfy an energy bound and a weak-strong uniqueness property at a.a.\ $t> 0$ with respect to strong solutions of the abstract SPDE
\begin{equation}
\label{eq:SEE}
\dd u + A u \,\dd t = F(\cdot,u)\,\dd t + (Bu + G(\cdot,u))\,\dd W,\qquad
u(t)=u_{t},
\end{equation}
in the $\set=\setfull$-setting. The precise formulation of the main results in the above abstract context is given in Subsection \ref{ss:main_results}. Some extensions and further comments are provided in Subsection \ref{ss:extension_strategy}. Before doing so, we first rigorously define \emph{singular times} for a stochastic process satisfying appropriate measurability conditions. Some basic properties are also discussed. 

\subsection{Regular and singular times} 
Here, we define regular and singular times. We start with a basic assumption concerning the measurability of the process $u$ under consideration. 

\begin{assumption}[Standing assumptions] 
\label{ass:standing_assumption}
We assume that:
\begin{itemize}
\item The abstract {\normalfont{SPDE}} \eqref{eq:SEE} is locally well-posed in the setting $\set=\setfull$, i.e., Assumptions \ref{ass:setting}, \ref{ass:measurability_boundedness_AB} and \ref{ass:criticality} hold, and $(A,B)$ has stochastic maximal $L^p$-regularity in the $\set$-setting (see also the text below Definition \ref{def:solutions_SEE}). 
\item The process $u:\O\times [0,\infty)\to \xx$ is progressively measurable, where $\xx$ is a separable Banach space such that 
\begin{equation}
\label{eq:xx_xap}
\xx\embed \Xap.
\end{equation}
\end{itemize}
\end{assumption}

As recalled in \eqref{eq:path_spaces_for_SEE}, $\Xap=(X_0,X_1)_{1-\frac{1+\a}{p},p}$ is natural in the theory of critical spaces for SPDEs \eqref{eq:SEE}, and is the optimal one for the initial data $u_t$ in \eqref{eq:SEE} for the $L^p$-approach to SPDEs, see e.g., \cite{ALV21} or \cite[Sections 3 and 4]{AV25_survey}.
As will become clear later on, the weaker assumption that $\xx$ and $X_0$ are compatible (or form an interpolation couple, see e.g., \cite[Definition C.1.1]{Analysis1}) is sufficient for defining the singular and regular times of the stochastic process. However, for simplicity, we adopt here the stronger condition \eqref{eq:xx_xap}. This simplifies the presentation, as this assumption is later required to estimate the size of the set of singular times.

\smallskip

Following the heuristic idea outlined around \eqref{eq:smoothness_regular_times}, we define regular times as those times $t_0$ for which there exists a stochastic interval $I_0$ containing the time $t_0$ such that $u|_{I_0\times \O}$ belongs to the regularity class of \emph{strong} solutions to \eqref{eq:SEE} (and hence is, roughly speaking, smooth) in the $\set$-setting with high probability; see Definition \ref{def:solutions_SEE}. By the local well-posedness part of Assumption \ref{ass:standing_assumption}, strong solutions to \eqref{eq:SEE} are unique in the $\set$-setting.

\begin{definition}[Regular and singular times]
\label{def:singular_regular} 
Let Assumption \ref{ass:standing_assumption} be satisfied.
\begin{itemize}
\item A time $t_0>0$ is said to be \emph{regular} (for $u$ in the $\set$-setting) if for all $\varepsilon\in (0,1)$ there exist a time $t<t_0$ and a stopping time 
$\tau:\O\to [t,\infty]$ such that 
$$
\P(\tau>t_0)>1-\varepsilon
$$ 
and the $\P\otimes \dd t$-equivalent class of $u|_{(t,\tau)\times \O}$ satisfies
$$
u|_{(t,\tau)\times \O}\in L^p_{\loc}((t,\tau);X_{1})\cap C((t,\tau);\Xap)\text{ a.s.\ }
$$
\item A time $t_0\geq 0$ is said to be \emph{singular} (for $u$ in the $\set$-setting) if it is not regular.
\end{itemize}
The sets of regular and singular times are denoted by $\reg^{\set}$ and $\si^{\set}$, respectively. 
\end{definition}

A visualization of Definition \ref{def:singular_regular} is given in Figure \ref{fig:1} with a connection to the notion of regular times in the deterministic setting \eqref{eq:smoothness_regular_times}; see Remark \ref{r:deterministic_case} for further comments.
Note that  
$$
\reg^{\set}\subseteq \R_+, \qquad \si^{\set}\subseteq [0,\infty), \qquad \si^{\set}=[0,\infty)\setminus \reg^{\set},
$$
and $t=0$ is always singular. While the latter is to some extent arbitrary, it is natural in applications to SPDEs where the initial data are typically more irregular compared to what is needed for local well-posedness (for instance, in the case of Leray-Hopf solutions of 3D NSEs, one assumes $u_0\in L^2(\T^3)$, while local well-posedness is only available for initial data in $L^3(\T^3)$, see Subsection \ref{ss:NSEs_time_intro} and Section \ref{s:NS}). 
In particular, to accommodate the roughness of the initial data, weak solutions are expected to be irregular near $t=0$. 
When the setting $\set = \setfull$ is clear from the context, we will simply write $\reg$ (or $\si$) instead of $\reg^{\set}$ (or $\si^{\set}$). 

\smallskip

To connect our definition of regular times with the more familiar ones in the context of 3D NSEs (see e.g., \cite[Subsection 13.6]{LePi}), let us note that in all relevant contexts (i.e., when $u$ solves an SPDE in an appropriate sense), further regularity around regular times can be a posteriori bootstrapped. Let us emphasize that bootstrap of regularity is typically only possible when solutions to \eqref{eq:SEE} are considered in a setting for which local well-posedness holds (this motivates the first requirement in Assumption \ref{ass:standing_assumption}).
We make this observation precise in Subsection \ref{sss:independence_set}, where we show that the space $L^p_t(X_1) \cap C_t(\Xap)$ can, for instance, be replaced with the maximal $L^p$-regularity space in the $\set$-setting (see Lemma \ref{lem:instantaneous_reg}): 
$$
\textstyle{
\bigcap_{\theta \in [0, 1/2)}}\, H_{\loc}^{\theta,p}((t, \tau); X_{1-\theta}).
$$ 
As we will see in the context of the stochastic 3D NSEs, bootstrapping arguments show that singular times are \emph{independent} of $\set$, and near a regular time an annealed Leray-Hopf solution $u$ to \eqref{eq:NSE_intro} is as smooth as the noise coefficients \eqref{eq:regularity_NSE_coefficients} allow, i.e., \eqref{eq:reguality_of_strong_solutions}. Thus, the above coincide with the one given in Subsection \ref{s:introduction}. 
As a byproduct of the bootstrapping, one also obtains the \emph{independence} of the singular and regular times from the setting $\set$ initially chosen. 

\smallskip

\begin{figure}[htbp]
    \centering
    \scalebox{0.75}{
    \begin{tikzpicture}[font=\Large] 
    
    \colorlet{myblue}{blue!60!black}
    \colorlet{myorange}{orange!80!red!80!black}
    \colorlet{myred}{red!60!black}

    \draw[thick, -{Latex[length=3mm]}] (0,0) -- (10,0) node[right] {$t$};
    \draw[thick, -{Latex[length=3mm]}] (0,0) -- (0,-8) node[below left] {$\Omega$};

    \node at (4, 0.5) {$t$};
    \node at (6, 0.5) {$t_0$};

    \draw[thick] (4, 0.2) -- (4, -0.2);
    \draw[thick] (6, 0.2) -- (6, -0.2);
    \draw[dashed, thick] (4, -8) -- (4, 0);
    \draw[dashed, thick] (6, -8) -- (6, 0);

    \draw[black, thick,dashed] (0, -2) -- (10, -2);
    \draw[black,thick, dashed] (0, -6.5) -- (10, -6.5);

    \filldraw[myred, opacity=0.3, pattern=north west lines, pattern color=myorange] (4, -6.5) rectangle (6.35, -2);
    \draw[myred, thick] (4, -6.5) rectangle (6.35, -2);
    \node at (5.2, -4.25) {$I_0^\varepsilon$};

    \draw[thick, myblue] (4, 0)
    .. controls (7, -1) and (7.5, -4) .. (7.5, -6.5)
    .. controls (7.2, -7) and (5.5, -7.2) .. (5, -7.5);

    \draw[dashed, thick, black] (0,-7.5) -- (10,-7.5);

    \node[black] at (6.5, -1.4) {$\tau$};
    
    \node[black] at (-0.5, -4.1) {$\Omega_0^\varepsilon$};

    \node[black, right] at (8, -4.25) {$\P(\O_0^\varepsilon)> 1-\varepsilon$};

    \end{tikzpicture}
    }
    \caption{Visualization of the behaviour of $u$ around the regular time $t_0$, where $\varepsilon\in (0,1)$. The red box $I_0^\varepsilon$ provides a region in the time-sample space where $u$ is regular for the $\set$-setting.}
    \label{fig:1}
\end{figure}

From Definition \ref{def:singular_regular}, the set of regular and singular times can naturally be approximated by a family of sets describing the time regularity or irregularity with $\varepsilon>0$ fixed. 

\begin{definition}[$\varepsilon$-regular and $\varepsilon$-singular times]
\label{def:singular_regular2}
Let $\varepsilon\in (0,1)$, and suppose that Assumption \ref{ass:standing_assumption} holds. 
\begin{itemize}
\item We say that a time $t_0>0$ is \emph{$\varepsilon$-regular} (for $u$ in the $\set$-setting) if there exist a time $t<t_0$ and a stopping time $\tau:\O\to [t,\infty]$ such that
$$ 
\P(\tau>t_0)>1-\varepsilon
$$ 
and the $\P\otimes \dd t$-equivalent class of $u|_{(t,\tau)\times \O}$ satisfies
$$
u|_{(t,\tau)\times \O}\in L^p_{\loc}((t,\tau);X_{1})\cap C((t,\tau);\Xap)\text{ a.s.\ }
$$
\item We say that a time $t\geq 0$ is a \emph{$\varepsilon$-singular} (for $u$ in the $\set$-setting) if it is \emph{not} $\varepsilon$-regular.
\end{itemize}
The sets of regular and singular times are denoted by $\reg^{\set,\varepsilon}$ and $\si^{\set,\varepsilon}$, respectively.
 \end{definition}

As above, we do not display the dependence on the setting $\set$ if clear from the context. 
From the above definitions, it readily follows that  
\begin{equation}
\label{eq:intersection_representation}
\reg^{\set}=\textstyle{\bigcap_{\varepsilon\in (0,1)}} \reg^{\set,\varepsilon} \quad \text{ and }\quad 
\si^{\set}=\textstyle{\bigcup_{\varepsilon\in (0,1)}} \si^{\set,\varepsilon}.
\end{equation}
Moreover, the sets $\reg^{\set,\varepsilon} $ (resp.\ $\si^{\set,\varepsilon}$) are decreasing (resp.\ increasing) as $\varepsilon\downarrow 0$. In particular, the intersection and union in \eqref{eq:intersection_representation} over $\varepsilon\in (0,1)$ can be replaced by corresponding operations over any sequence $(\varepsilon_k)_{k\in\N}$ satisfying $\varepsilon_k\downarrow0$.
In particular, topological and measurable properties of the singular and regular sets can be deduced from their $\varepsilon$-approximations.

As common practice, we denote by $\mathsf{G}_\delta$ (resp.\ $\mathsf{F}_\sigma$) the Borel sets that can be obtained via an intersection  (resp.\ union) of countable open (resp.\ closed) sets.

\begin{lemma}[Topological and measurability properties of singular and regular times]
\label{l:Borel_measurable}
In the setting of Definition \ref{def:singular_regular} and \ref{def:singular_regular2}, the sets of $\varepsilon$-regular and $\varepsilon$-singular times $\reg^{\set,\varepsilon}$ and $\si^{\set,\varepsilon}$ of the process $u$ are open and closed, respectively. In particular, the sets of regular and singular times $\reg^{\set}$ and $\si^{\set}$ of the process $u$ are Borel measurable in the class $\mathsf{G}_\delta$ and $\mathsf{F}_\sigma$, respectively. 
\end{lemma}

\begin{proof}
From \eqref{eq:intersection_representation} and the comments below it, it suffices to show that $\reg^{\varepsilon}$ is open for each fixed $\varepsilon\in (0,1)$.
To see this, let $t_0\in \reg^{\varepsilon}$ and let $t$ and $\tau$ be as in Definition \ref{def:singular_regular}. First, note that the set $\reg^\varepsilon$ is open on the left as any $t_1\in (t,t_0)$ also satisfies $t_1\in \reg^{\varepsilon}$, cf.\ Figure \ref{fig:1}. Second, from the fact that $\P(\tau>t_0)>1-\varepsilon$ it follows that there exists $\delta>0$ such that $\P(\tau>t_0+\delta)>1-\varepsilon$. Thus, we also have $t_1\in \reg^{\varepsilon}$ provided $t_0<t_1<t_0+\delta$. 
Hence, $\reg^{\varepsilon}$ is open, and the proof is complete. 
\end{proof}

We conclude this subsection by comparing singular times introduced in Definitions \ref{def:singular_regular} and \ref{def:singular_regular2} to the one usually employed in the deterministic setting.

\begin{remark}[Comparison with the deterministic]
\label{r:deterministic_case}
Note that in the absence of noise in \eqref{eq:SEE}, it holds that $\reg^{\varepsilon}$ (and so $\si^{\varepsilon}$) is independent of $\varepsilon\in (0,1)$, and therefore the latter coincide with the heuristic idea given in \eqref{eq:smoothness_regular_times} in the deterministic setting. In particular, without noise, singular times are always \emph{closed}, while this does not hold in general in the stochastic case; see Lemma \ref{l:Borel_measurable}.
\end{remark}

\subsection{Bounds on the fractal dimension of singular times for SPDEs}
\label{ss:main_results}
In this subsection, we state our main result on singular times for the abstract SPDE \eqref{eq:SEE}. We begin by listing the main assumptions. The following connects the process $u$ in Assumption \ref{ass:standing_assumption} to the abstract SPDE \eqref{eq:SEE}.

\begin{assumption}[Strong weak-strong uniqueness property in the $\set$-setting]
\label{ass:abstract2}
Let $\set$ and $u$ be as in Assumption \ref{ass:standing_assumption}. 
We say that $u$ satisfies the \emph{strong weak-strong uniqueness property in the $\set$-setting} if there exists a set $\Iz\subseteq \R_+$ of zero Lebesgue measure such that, for all $t\in (0,\infty)\setminus\Iz$, 
\begin{equation}
\label{eq:weak_strong_uniqueness}
u=v \   \text{ a.e.\ on }\ [t,\tau)\times \O,
\end{equation}
where $(v,\tau)$ is the maximal solution to \eqref{eq:SEE} in the $\set$-setting with initial data $u(t,\cdot)\in L^0_{\F_t}(\O;\Xap)$ at time $t$ in the $\set$-setting (see Theorem \ref{t:lwp}). 
\end{assumption}

The terminology \emph{strong} weak-strong uniqueness is borrowed from the fluid dynamics literature related to Leray-Hopf solutions, where strong is related to the properties that hold for almost all times; see e.g., \cite[Proposition 12.1]{LePi}, Subsection \ref{ss:NSEs_time_intro} or \ref{ss:annealed_strong_energy}.
The final ingredient in the main results of this section is as follows.

\begin{assumption}[Energy bound]
\label{ass:abstract1}
Let $\set$ and $u$ be as in Assumption \ref{ass:standing_assumption}. Suppose there exists $\r\in [1,\infty)$ and $(\O_n)_n\subseteq \F_0$ satisfying $\O_n\uparrow \O$ such that, for all $T<\infty$,  
$$
\E \Big[\one_{\O_n}\int_0^T \|u(t)\|^{\r}_{\xx}\,\dd t\Big]<\infty  \ \text{ for all }n\geq 1, \ T<\infty.
$$
\end{assumption}

The sets $(\O_n)_n$ can be used to deal with non-$\O$-integrable initial data $u_0$. Localization via stopping times is also possible. However, it will not be needed here. 

\smallskip

Similar to Subsection \ref{ss:abstract_intro}, motivated by \eqref{eq:sub_critical_condition} in Assumption \ref{ass:criticality}, we define the \emph{excess} of $\xx$ from critical regularity for \eqref{eq:SEE} in the $\set$-setting as:
\begin{align}
\label{eq:excess_criticality_1}
\gap_\set&\stackrel{{\rm def}}{=} 
\Big(\min_{j\in \{1,\dots,m\}}\ggap_{\set,j}\Big) \Big(1+ \frac{1}{\max_{j\in \{1,\dots,m\}}\rho_j}\Big),\\
\label{eq:excess_criticality_2}
\ggap_{\set,j}&\stackrel{{\rm def}}{=} 1-\beta_j-\frac{\rho_j}{\rho_j+1} \frac{1+\a}{p}\ \  \text{ for } \ j\in \{1,\dots,m\}.
\end{align}
From the well-posedness of \eqref{eq:SEE} in the $\set$-setting in Assumption \ref{ass:abstract2}, it follows that $\ggap_{\set,j}\geq 0$ for all $j$. 
Moreover, if $m=1$, then the above coincides with \eqref{eq:excess_criticality_intro}.
Under additional assumptions, a closer version to \eqref{eq:excess_criticality_intro} of the excess of $\xx$ from the criticality can be found in Remark \ref{r:sum_vs_decomposition}.

\smallskip

Next, we formulate the main result of the current manuscript for the abstract SPDE \eqref{eq:SEE}. 
Below, $\H^{s}$ (resp.\ $\M^s$) and $\dim_{\H}$ (resp.\ $\dim_{\M}$) are the Hausdorff measure (resp.\ Minkowski/box-counting content) and corresponding dimension, respectively. The reader is referred to Subsection \ref{ss:fractal} for the notation.

\begin{theorem}[Bounds on fractal dimension of singular times for SPDEs] 
\label{t:singular_times_SPDEs}
Let $u$ be a stochastic process satisfying Assumptions \ref{ass:standing_assumption}, \ref{ass:abstract1} and \ref{ass:abstract2}. Suppose that 
\begin{align}
\label{eq:singular_times_SPDEs_ass1}
\gap_\set &>0,\quad \text{ (Spatial subcriticality)} \\ 
\label{eq:singular_times_SPDEs_ass2}
 \gap_\set& < \tfrac{1}{\r}. \quad \text{ (Space-time supercriticality)}
\end{align}
For $\varepsilon\in (0,1)$, let $\si^{\set}$ and $\si^{\set,\varepsilon}$ be the set of singular and $\varepsilon$-singular times of $u$ with respect to the setting $\set$, respectively; see Definition \ref{def:singular_regular} and \ref{def:singular_regular2}. 
Then
\begin{align}
\label{eq:singular_times1}
\dim_{\M}(\si^{\set,\varepsilon})&\leq 1-\r\,\gap_\set, \ \quad\text{ and }\ \quad 
\M^{1-\r\,\gap_\set}(\si^{\set,\varepsilon})=0,  \\ 
\label{eq:singular_times2}
\dim_{\H}(\si^{\set})&\leq 1-\r\,\gap_\set, \ \quad\text{ and }\ \quad \ 
\H^{1-\r\,\gap_\set}(\si^{\set})=0. 
\end{align}
\end{theorem}

The proof of the above is postponed to Subsection \ref{ss:proofs_main_abstract}. The comments on \eqref{eq:singular_times_SPDEs_ass1} and \eqref{eq:singular_times_SPDEs_ass2} given below Theorem \ref{t:abstract_intro} partially extend to the current situation; see also Remark \ref{r:sum_vs_decomposition} below. For brevity, we do not repeat them here.  

It seems a challenging problem to determine whether \ref{eq:singular_times2} can be obtained with $\H$ and $\dim_{\H}$ replaced by $\M$ and $\dim_{\M}$, respectively. 
This is due to the structure of the set of singular times $\si^\set$, see \eqref{eq:intersection_representation}, and the lack of $\sigma$-additivity of the upper Minkowski content $\M^s$. 

\smallskip

In the following, we formulate a variant of Theorem \ref{t:singular_times_SPDEs} where the definition of excess of $\xx$ from the criticality in \eqref{eq:excess_criticality_1} can be improved under additional conditions.

\begin{remark}
\label{r:sum_vs_decomposition}
From the proof of Theorem \ref{t:singular_times_SPDEs} (arguing componentwise in the proof of Proposition \ref{prop:quantitative_lower_bound}), it follows that if $F=\sum_{j=1}^{m_F} F_j$ and $G=\sum_{k=m_F+1}^{m} G_k$ for some $m\geq m_F$, and $F_j,G_k$ satisfy
\begin{align*}
\|F_j(\cdot,u)-F_j(\cdot,u')\|_{X_0}
&\lesssim (1+\|u\|_{X_{\beta_j}}^{\rho_j}+\|u'\|_{X_{\beta_j}}^{\rho_j})\|u-u'\|_{X_{\beta_j}},\\
\|G_k(\cdot,u)-G_k(\cdot,u')\|_{\g(H,X_{1/2})}
&\lesssim (1+\|u\|_{X_{\beta_k}}^{\rho_k}+\|u'\|_{X_{\beta_k}}^{\rho_k})\|u-u'\|_{X_{\beta_k}},
\end{align*}
where all components are strongly progressively measurable and $(\rho_j,\beta_j)_{j=1}^m$ are as in Assumption \ref{ass:criticality}, then the results of Theorem \ref{t:singular_times_SPDEs} hold with $\gap_\set$ replaced by
$$
\min_{j\in \{1,\dots,m\}}\Big(
\frac{\rho_j+1}{\rho_j}(1-\beta_j)-\frac{1+\a}{p}\Big).
$$
\end{remark}

Next, we address the case of spatially critical energy bounds, i.e., $\gap_\set=0$. Below, $|\cdot|$ denotes the $d$-dimensional Lebesgue measure.

\begin{theorem}[Bounds on fractal dimension of singular times for SPDEs -- Spatially critical case]
\label{t:singular_times_SPDEs_2}
Let $u$ be a stochastic process satisfying Assumptions \ref{ass:standing_assumption} and \ref{ass:abstract2}. 
Assume that the setting $\set=(X_0,X_1,p,\a)$ is spatial criticality, i.e.,   
$
\gap_\set = 0.
$
Let $\si^{\set}$ be the set of singular times of $u$ with respect to the setting $\set$; see Definition \ref{def:singular_regular}. Then 
\begin{equation}
\label{eq:statement_critical_case_singular_set}
|\si^\set|=0 \quad \text{ and }\quad \M^{1}(\si^{\set,\varepsilon})=0 \text{ for all }\ \varepsilon\in (0,1).
\end{equation}
\end{theorem}

The proof of the above is given in Subsection \ref{ss:proofs_main_abstract_2}.
We emphasize that, in the spatially critical case, Assumption \ref{ass:abstract1} is not needed. 
As the one-dimensional Hausdorff measure $\H^1$ on $\R$ coincides with the Lebesgue measure (see e.g., \cite[Theorem 2.5]{EG_measure}), the above result is the endpoint version of \eqref{eq:singular_times2} in Theorem \ref{t:singular_times_SPDEs} in the critical case $\gap_\set=0$.

\subsection{Further properties and results on singular times}
\label{ss:extension_strategy}
In this subsection, we discuss some further properties of singular and regular times.

\subsubsection{Instantaneous regularization and $\set$-independence of singular times}
\label{sss:independence_set}
From Definitions \ref{def:singular_regular} and \ref{def:singular_regular2}, it follows that $\si^\set$ depend on the setting $\set$. 
In applications to SPDEs, there might be no preferred setting $\set$, as is the case, for instance, with the 3D NSEs analysed in Subsection \ref{ss:NSEs_time_intro}. 
In many situations, it is known that the regularity on $(t_0,\infty)$ of solutions to the abstract SPDEs \eqref{eq:SEE} is \emph{independent} of the chosen setting $\set$. 
This is a consequence of the so-called \emph{instantaneous regularization} phenomenon, which is typical in parabolic PDEs, although challenging to obtain for critical problems. For the abstract setting, the reader is referred to \cite[Section 6]{AV19_QSEE2} and \cite[Subsection 5.3]{AV25_survey}, while consequences for concrete SPDEs can be found in \cite{Primitive3,AgrSau,RD_AV23} and \cite[Section 8]{AV25_survey}. For instance, in the 3D stochastic NSEs, we will employ the results in \cite[Subsection 2.3]{AV21_NS}; see particularly Theorem 2.7 there.

The following result yields independence of the singular times $\si^\set$ from the setting $\set$, when instantaneous regularization is known for the abstract SPDE \eqref{eq:SEE}.

\begin{lemma}[Instantaneous regularization and regular times]
\label{lem:instantaneous_reg} 
Let Assumption \ref{ass:standing_assumption} and \ref{ass:abstract2} be satisfied.
Let $(\mathscr{X}_{(s,t)})_{s<t}$ be a family of function spaces such that 
\begin{equation}
\label{eq:embedding_MRX_mathscrX}
\mathscr{X}_{(s,t)} \subseteq L^p_{\loc}(s,t;X_1)\cap C((s,t);\Xap) \ \text{ for all }\ 0\leq s<t<\infty.
\end{equation}
Suppose that, for any $t\geq 0$, and any initial data $u_{t}\in L^0_{\F_{t}}(\O;\Xap)$, the maximal solution $(v,\tau)$ to \eqref{eq:SEE} in the $\set$-setting satisfies 
\begin{equation}
\label{eq:instantaneous_regularization_X}
v\in \mathscr{X}_{(t,\tau)} \text{ a.s. }\quad (\text{Instantaneous regularization})
\end{equation}
Then, for each $\varepsilon\in (0,1)$ and $t_0\in \reg^{\set,\varepsilon}$, there exists $t<t_0$ and a stopping time $\tau:\O\to [t,\infty]$ such that 
$$ 
\P(\tau>t_0)>1-\varepsilon\qquad \text{ and }\qquad
u\in \mathscr{X}_{(t,\tau)} \text{ a.s.}
$$
\end{lemma}
 
Lemma \ref{lem:instantaneous_reg} is proven at the end of this subsection. 
In case the couple of $(A,B)$ in \eqref{eq:SEE} is such that the linearized problem \eqref{eq:SMR_recap} has optimal space-time $L^p$-regularity estimates (see e.g., \cite[Definition 3.8 and Proposition 3.18]{AV25_survey}, and the comments below \eqref{eq:stochastic_maximal_Lp_regularity}), the above result is always applicable with 
$$
\mathscr{X}_{(s,t)}= 
 \textstyle{\bigcap_{\theta\in [0,1/2)}}\, H^{\theta,p}_{\loc}((s,t);X_{1-\theta}).
$$ 
Combining Lemma \ref{lem:instantaneous_reg} and \cite[Theorem 2.7]{AV21_NS}, one obtains that the singular and regular times for the stochastic 3D NSEs \eqref{eq:NSE_intro} in any setting $\set$ in which the well-posedness holds coincide with those in \eqref{eq:condition_singular_times}-\eqref{eq:singular_sets_def_intro}; see Section \ref{s:NS}.

\smallskip

An abstract version of this argument reads as follows.

\begin{lemma}[Setting independence of singular times via instantaneous regularization]
\label{lem:independence}
Let Assumption \ref{ass:standing_assumption} and \ref{ass:abstract2} be satisfied.
Let $\sety=\setfully$ be another setting for which local well-posedness for \eqref{eq:SEE} holds in the setting $\sety$ such that 
\begin{equation}
\label{eq:connection_density}
X_1\cap Y_1 \stackrel{{\rm d}}{\embed} Y_1\qquad \text{ and }\qquad X_1\cap Y_1 \stackrel{{\rm d}}{\embed} X_1,
\end{equation}
and whose realizations of $A,B,F,G$ agree with those in the $\set$-setting on $X_1\cap Y_1$, on a common ambient Banach space $V$ such that $X_0, Y_0\embed V$.
Let $(\mathscr{X}_{(s,t)})_{s<t}$ be a family of function spaces such that \eqref{eq:embedding_MRX_mathscrX} holds.
Suppose that any maximal local solution $(v,\tau)$ to \eqref{eq:SEE} in the $\set$-setting with initial data $u_t\in L^0_{\F_t}(\O;\Xap)$ at a time $t\geq0$ satisfies  \eqref{eq:instantaneous_regularization_X}, and 
\begin{equation}
\label{eq:connection_X_Y}
\mathscr{X}_{(s,t)} \subseteq L^r_{\loc}(s,t;Y_1)\cap C((s,t);\Yar) \ \text{ for all }\ 0\leq  s<t<\infty,
\end{equation} 
where $\Yar=(Y_0,Y_1)_{1-\frac{1+\alpha}{r},r}$ is the natural space for initial data in the $\sety$-setting.
Then, for all $\varepsilon\in (0,1)$, 
\begin{equation}
\label{eq:independence_sety}
\si^{\sety,\varepsilon}\subseteq \si^{\set,\varepsilon}\qquad \text{ and }\qquad 
\si^{\sety}\subseteq \si^{\set}.
\end{equation}
\end{lemma}

Note that the density assumption \eqref{eq:connection_density} alone does not ensure consistency of the two formulations;
their agreement on a common space is assumed above. Moreover, \eqref{eq:connection_X_Y} connecting the $\mathscr{X}_{(0,t)}$ to the (maximal) regularity in the $\sety$-setting is essential.  
Clearly, \eqref{eq:connection_X_Y} and Lemma \ref{lem:instantaneous_reg} ensure 
$$
\re^{\set,\varepsilon}\subseteq 
\re^{\sety,\varepsilon}  \  \text{ for all }\ \varepsilon\in (0,1).
$$
In particular, \eqref{eq:independence_sety} follows from the above and Definitions \ref{def:singular_regular} and \ref{def:singular_regular2}.
We conclude this subsection by proving Lemma \ref{lem:instantaneous_reg}.

\begin{proof}[Proof of Lemma \ref{lem:instantaneous_reg}]
Fix $\varepsilon\in(0,1)$ and $t_0\in\reg^{\set,\varepsilon}$. From Definition \ref{def:singular_regular2}, there exists $t<t_0$ and a stopping time $\tau:\O\to [t,\infty]$ such that, a.s., 
\begin{equation*}
u\in L^p_{\loc}((t,\tau);X_1)\cap C((t,\tau);\Xap) \qquad \text{ and }\qquad \P(\tau>t_0)>1-\varepsilon.
\end{equation*}
Let $\Iz$ be as in Assumption \ref{ass:abstract2}. Since $\Iz$ has Lebesgue measure zero, then  $\Iz^{{\rm c}}=\R_+\setminus\Iz$ is dense in $\R_+$. Therefore $ (t,t_0)\cap \Iz^{{\rm c}}$ is not empty. Fix $t_1\in (t,t_0)\cap \Iz^{{\rm c}}$. It follows from the previously displayed formula that 
\begin{equation}
\label{eq:u_regularity_1}
u\in L^p_{\loc}([t_1,\tau_1);X_1)\cap C([t_1,\tau_1);\Xap) , \quad \text{ where }\quad \tau_1\stackrel{{\rm def}}{=}\tau\vee t_1.
\end{equation}
Moreover, from $t_0>t_1$, we get 
\begin{equation}
\label{eq:tau_1_is_a_good_stopping_time}
\P(\tau_1>t_0)>1-\varepsilon.
\end{equation}
As $t_1\in \Iz^{{\rm c}}$, from the strong weak-strong uniqueness of $u$ in the $\set$-setting (see Assumption \ref{ass:abstract2}), we have
\begin{equation}
\label{eq:weak_strong_uniqeuness_independence}
u=v \text{ a.e.\ on } \ [t_1,\sigma)\times \O,
\end{equation}
where $(v ,\sigma )$ is the maximal solution to \eqref{eq:SEE} in the $\set$-setting starting with initial data $u_{t_1}\stackrel{{\rm def}}u(t_1)$ at time $t=t_1$. Note that, due to \eqref{eq:u_regularity_1}, $u_{t_1}\in L^{0}_{\F_{t_1}}(\O;\Xap)$. 
From \eqref{eq:weak_strong_uniqeuness_independence} and the instantaneous regularization assumption \eqref{eq:instantaneous_regularization_X}, we infer
$
u\in  \mathscr{X}_{(t_1,\tau_1)} \text{ a.s.}
$
This and \eqref{eq:tau_1_is_a_good_stopping_time} conclude the proof.
\end{proof}

\subsubsection{Weakening weak-strong uniqueness via setting compatibility}
\label{sss:compatibility_bounds}
In certain situations, it is easier to prove weak-strong uniqueness in a certain setting $\sety$, while the energy bound of Assumption \ref{ass:abstract1} can only be connected via \eqref{eq:xx_xap} to another setting $\set$. 
This situation appears, for instance, in the case of stochastic 3D NSEs as Theorem \ref{t:NSE_intro}\eqref{it:NSE_intro2} (see also Theorem \ref{t:singular_times_SNS_2}). In the former case, $\xx=L^{q_0}$ for which \eqref{eq:xx_xap} forces the choice $\set=(H^{-1-s,q_0},H^{1-s_0,q_0},r_0,\a_0)$ for some $s_0>0$ and suitable $r_0,\a_0$ (see Subsection \ref{sss:proof_singular_times_SNS_2} and \cite[Theorem 2.4 and Remark 2.5(4)]{AV21_NS}). 
In particular, the regularity available in the $\xx$-setting does not, in general, guarantee finite energy dissipation (i.e., they might not belong to $L^2_t (H^1_x)$; see \eqref{eq:annealed_energy_inequality_intro}). In the context of infinite energy solutions of 3D NSEs, weak-strong uniqueness seems a difficult task. However, strong weak-strong uniqueness $\sety=(H^{-1,q},H^{1,q},p,\a)$ for suitable $p,\a$ hold (see Proposition \ref{prop:weak_strong_uniqueness}).

We resolve this by requiring that the settings $\set$ and $\sety$ are in some sense `compatible', which in applications can be checked by means of instantaneous regularization.

\smallskip

Let $\set=\setfull$ and $\sety=\setfully$ be such that well-posedness of \eqref{eq:SEE} holds in both settings (see below Theorem \ref{t:lwp}) and \eqref{eq:connection_density} holds. 
We say that $\sety$ is \emph{compatible with} $\set$ if for all $t\geq 0$ and $u_t\in L^0_{\F_t}(\O;\Xap)\cap L^0_{\F_t}(\O;\Yar)$, it holds that 
\begin{equation}
\label{eq:setting_compatibility}
\tau^\sety \geq \tau^{\set}\text{ a.s. }\qquad \text{ and }\qquad v^{\sety}=v^{\set}\text{ a.e.\ on }[t,\tau^{\set})\times \O,
\end{equation}
where $(v^\set,\tau^\set)$ and $(v^\sety,\tau^\sety)$ are the maximal solutions to \eqref{eq:SEE} in the $\set$-setting and $\sety$-setting, respectively.

\begin{proposition}[Bounds on fractal dimension of singular times for SPDEs -- Compatible settings]
\label{prop:singular_times_SPDEs_weaker_weak_strong} 
Let $u$ be a stochastic process satisfying Assumptions \ref{ass:standing_assumption} and \ref{ass:abstract1}.
Suppose that the strong weak-strong uniqueness of Assumption \ref{ass:abstract2} is satisfied with $\set$ replaced by another setting $\sety$ that is compatible with $\set$. The following assertions hold.
\begin{itemize}
\item If $
\r\,\gap_\set<1,
$ then the bounds \eqref{eq:singular_times1} and \eqref{eq:singular_times2} on dimension and measures of the sets of singular times of $u$ in the $\set$-setting  hold.
\item If $\gap_\set=0$, then the singular times of $u$ in the $\set$-setting has Lebesgue measure zero, i.e.,  \eqref{eq:statement_critical_case_singular_set} hold.
\end{itemize} 
\end{proposition}

The above follows from the proofs of Theorems \ref{t:singular_times_SPDEs} and \ref{t:singular_times_SPDEs_2} with minor modifications, and its proof is given in Subsection \ref{ss:proof_singular_times_SPDEs_weaker_weak_strong}.

\section{Quantifying local well-posedness and singular times for SPDEs}
\label{s:quantifying_main_proof}
In this section, we prove Theorems \ref{t:singular_times_SPDEs} and \ref{t:singular_times_SPDEs_2}, and Proposition \ref{prop:singular_times_SPDEs_weaker_weak_strong}. 
The key ingredient is the following result, which gives quantitative estimates on the lifetime of solutions to the SEEs \eqref{eq:SEE} under the assumptions in Subsection \ref{ss:critical_stochastic_evol}.

\begin{proposition}[Lower bounds of lifetime of solutions for SEEs]
\label{prop:quantitative_lower_bound}
Let the assumptions of Theorem \ref{t:lwp} be satisfied. Let $\gap_\set\geq 0$ denote the excess from the criticality; see \eqref{eq:excess_criticality_1}.
For all $R>0$, there exists a constant $C_R>0$ for which the following assertion holds. For all 
\begin{equation*}
t\in [0,R], \qquad \text{ and }\qquad  u_t\in L^0_{\F_t}(\O;\Xap), 
\end{equation*}
the maximal local solution $(v,\tau)$ to \eqref{eq:SEE} in the $\set$-setting with initial data $u_t$ at time $t$ satisfying $\tau>t$ a.s.\ and 
\begin{equation}
\label{eq:sigma_zero_lower_bound}
\P(\tau-t\leq T)\leq C_R T^{p\, \gap_\set} (1+N^p)
+ \P(\|u_t\|_{\Xap}> N) 
\end{equation}
 for all $T>0$ and $N\geq 0$.
\end{proposition}

Note that the maximal unique solution $(u,\sigma)$ to \eqref{eq:SEE} in the $\set$-setting provided in Subsection \ref{ss:critical_stochastic_evol} satisfies $\sigma>\tau$ a.s.\ and $u=v$ a.e.\ on $[t,\tau)\times \O$ (see Definition \ref{def:solutions_SEE}). Hence, \eqref{eq:sigma_zero_lower_bound} also holds with $\tau$ replaced by $\sigma$.

Clearly, \eqref{eq:sigma_zero_lower_bound} is useful only in the spatially subcritical case $\gap_\set>0$. Although in the latter case the $\set$-setting is subcritical, sharp estimates are needed to capture the optimal rate $T^{p\,\gap_\set}$ (see Subsection \ref{ss:critical_stochastic_evol}). The absence of a rate in the critical setting  $\gap_\set=0$ is natural, as even in the deterministic case, the dependence of the lifetime of local solutions is very subtle. In general, if $\gap_\set=0$, then even in the deterministic setting, it is \emph{not} possible to obtain a uniform lower bound on the time of solutions to the SEE \eqref{eq:SEE}; see 
e.g., \cite[Chapter 10]{RRS_3D} for comments on the 3D NSEs.
In particular, in the case $\gap_\set=0$, then Proposition \ref{prop:quantitative_lower_bound} is a consequence of the well-posedness results in \cite[Theorem 4.8]{AV19_QSEE1} (see also \cite[Section 4]{AV25_survey}).

If $\gap_\set>0$, then letting $N\uparrow \infty$ and $T\downarrow 0$, the estimate \eqref{eq:sigma_zero_lower_bound} provides a lower bound with \emph{explicit rate} for the probability of $\tau>t+T$ in terms of the size of the initial data, and on the excess from criticality of the $\set$-setting. 
In the proof of Proposition \ref{prop:quantitative_lower_bound} we obtain the following variant of \eqref{eq:sigma_zero_lower_bound}: 
$$
\P(\tau-t\leq T)\leq C_0\, T^{p\,\gap_\set}(1+ \E\|u_t\|_{\Xap}^p) \  \text{ for all }\ T> 0.
$$
In the proof of Proposition \ref{prop:quantitative_lower_bound}, we will show that the estimate \eqref{eq:sigma_zero_lower_bound} is localized in the $\O$ version of the above inequality. We emphasize that the tail probability formulation in \eqref{eq:sigma_zero_lower_bound} is essential to handle the case $\r\ll p$ in Theorem \ref{t:singular_times_SPDEs} (see Assumption \ref{ass:abstract1}), which is typical in applications to SPDEs.

\smallskip

This section is organized as follows. Firstly, assuming the validity of Proposition \ref{prop:quantitative_lower_bound}, 
we prove Theorems \ref{t:singular_times_SPDEs} and \ref{t:singular_times_SPDEs_2}, and Proposition \ref{prop:singular_times_SPDEs_weaker_weak_strong} in Subsections \ref{ss:proofs_main_abstract}, \ref{ss:proofs_main_abstract_2} and \ref{ss:proof_singular_times_SPDEs_weaker_weak_strong}, respectively.
Secondly, in Subsection \ref{ss:proof_lower_bound_existing_times}, we show Proposition \ref{prop:quantitative_lower_bound}.

\subsection{Proof of Theorems \ref{t:singular_times_SPDEs} and \ref{t:singular_times_SPDEs_2}, and Proposition \ref{prop:singular_times_SPDEs_weaker_weak_strong}}
\label{ss:proof_abstract_together}

\subsubsection{Proof of Theorem \ref{t:singular_times_SPDEs}}
\label{ss:proofs_main_abstract}
Before going into the proofs, we collect some useful facts. Firstly,  from \eqref{eq:intersection_representation} and the comments below, it follows that 
\begin{equation*}
\si^{\set}=\textstyle{\bigcup_{k\in \N} \si^{\set,2^{-k}}}.
\end{equation*}
Thus, \eqref{eq:singular_times2} follows from \eqref{eq:singular_times1}, the $\sigma$-subadditivity of the Hausdorff measures and
$$
\H^{1-\r\,\gap }(\si^{\set,2^{-k}})\stackrel{\eqref{eq:HF_comparison}}{\leq} \M^{1-\r\,\gap }(\si^{\set,2^{-k}})=0.
$$
Hence, to conclude the proof of Theorem \ref{t:singular_times_SPDEs}, it remains to show that \eqref{eq:singular_times1}. From \eqref{eq:Feta_definition}, it suffices to prove that
\begin{equation}
\label{eq:reduction_claim_Minkowski}
\M^{1-\r\,\gap_\set}(\si^{\varepsilon,T})=0 \ \text{ for all }\ \varepsilon\in (0,1), \ T\in (0,\infty).
\end{equation}
with
\begin{equation}
\label{eq:sin_varepsilon_T}
\si^{\varepsilon,T}\stackrel{{\rm def}}{=}\si^{\set,\varepsilon} \cap [0,T] \quad \text{ for }T<\infty.
\end{equation}
In the previous and in the proof below, for notational convenience, we do not display the dependence on the setting $\set=\setfull$.

\begin{proof}[Proof of Theorem \ref{t:singular_times_SPDEs}]
As commented above, it suffices to prove \eqref{eq:reduction_claim_Minkowski}. Hence, throughout this proof we fix $\varepsilon\in (0,1)$ and $T\in (0,\infty)$. Recall that $\xx\embed \Xap$ by Assumption \ref{ass:standing_assumption}, and therefore Assumption \ref{ass:abstract1} holds with $\xx$ replaced by $\Xap$.

We split the proof into three steps. For convenience of exposition, in the first two steps, we prove the claim \eqref{eq:reduction_claim_Minkowski} under the additional assumption $\O_n\equiv \O$ in Assumption \ref{ass:abstract1} (i.e., when an $L^\ell(\O;\Xap)$-bound for $u$ holds), and in Step 3 we comment on the (minor) modifications needed to obtain \eqref{eq:reduction_claim_Minkowski} in the general case.

\smallskip

\emph{Step 1: Suppose that Assumption \ref{ass:abstract1} holds with $\O_n\equiv \O$. Then there exist constants $R,r>0$ such that, for all $\eta\in (0,r)$ and $t\in \si^{\varepsilon,T}$ satisfying $t\geq \eta$, it holds that}
$$
\E\int_{t-\eta}^t \|u(t_0)\|_{\Xap}^{\r}\,\dd t_0\geq R \,\eta^{1-\r \,\gap },
$$
\emph{where $\gap=\gap_\set$.}
Set $I_0=(0,T)\setminus\Iz$, where $\Iz$ is as in Assumption \ref{ass:abstract2}. By Tonelli's theorem, we may also assume that 
$$
\E\|u(t_0)\|^\ell_{\Xap}<\infty \text{ for all }t_0\in I_0.
$$ 
Note that, for all $t_0\in I_0\cap (t-\eta,t)$ (recall $t\geq \eta$), it follows from Proposition \ref{prop:quantitative_lower_bound} that there exists a maximal solution $(v,\tau)$ to \eqref{eq:SEE} with initial data $u(t_0)$ at time $t_0$, where $\tau>t_0$ is a stopping time satisfying
\begin{equation}
\label{eq:local_quantification_time_existence}
\P(\tau\leq t_0+ \g)\leq C_0\, \g^{p\,\gap } (1+N^p)
+ N^{-\r}\E\|u(t_0)\|_{\Xap}^{\r} \  \text{ for all } \g,N>0.
\end{equation}
In the above, $C_0$ is as in Proposition \ref{prop:quantitative_lower_bound}, and therefore independent of $t_0$.

From the weak-strong uniqueness property of Assumption \ref{ass:abstract2} in the $\set$-setting, we also have 
\begin{equation}
\label{eq:u_equal_ustrong}
u=v \ \ \text{ a.e.\ on }[t_0,\tau)\times \O.
\end{equation} 
Since, $(v,\tau)$ is a maximal solution to \eqref{eq:SEE}, from the above, it follows that 
\begin{equation}
\label{eq:regularity_strong_solution}
u=v\in L^p_{\loc}((t_0,\tau);X_1)\cap C((t_0,\tau);\Xap)\text{ a.s.\ }
\end{equation}

The above and the definition of $\si^{\varepsilon,T}$ show that, if $\g$ and $N$ satisfy the conditions 
\begin{equation}
\label{eq:implication_t_t0}
C_0\, \g^{p\, \gap} (1+N^p)\leq \tfrac{\varepsilon}{4} \qquad \text{ and }\qquad N^{-\r}\E\|u(t_0)\|_{\Xap}^{\r}\leq \tfrac{\varepsilon}{4},
\end{equation}
then 
\begin{equation}
\label{eq:implication_t_t01}
t-t_0\geq  \g.
\end{equation}
Indeed, by contradiction, assume that $t<t_0+\g$ and the conditions in \eqref{eq:implication_t_t0} are satisfied at the same time. Then \eqref{eq:local_quantification_time_existence} implies $\P(\tau>t_0+\g)\geq 1-\varepsilon/2>1-\varepsilon$. Since $t<t_0+\g$, there exists a time $t_0<t$ and a stopping time $\tau$ such that $\tau>t_0$ a.s., $\P(\tau>t)>1-\varepsilon$  and $u$ is regular in the $\set$-setting, see \eqref{eq:regularity_strong_solution}. This contradicts the fact that $t\in \si^{\varepsilon,T}$, see Definition \ref{def:singular_regular2}. Thus, the implication \eqref{eq:implication_t_t0}$\ \Rightarrow\ $\eqref{eq:implication_t_t01} holds.

Note that the second condition in \eqref{eq:implication_t_t0} is satisfied if 
$$
N=1+(\tfrac{4}{\varepsilon})^{1/\r} \big(\E\|u(t_0)\|_{\Xap}^{\r}\big)^{1/\r}.
$$ 
Using this choice in the first of \eqref{eq:implication_t_t0}, the latter condition holds provided
$$
\g =\big(\tfrac{\varepsilon}{4C_0 }(1+N^p)^{-1}\big)^{ 1/(p\, \gap)} \eqsim_{\varepsilon,p} 
\big(\E\|u(t_0)\|_{\Xap}^{\r}+ 1\big)^{-1/(\r\, \gap)}.
$$
Now, the implication \eqref{eq:implication_t_t0}$\ \Rightarrow \ $\eqref{eq:implication_t_t01} and the above yield
$$
1+\E\|u(t_0)\|_{\Xap}^{\r}\gtrsim_\varepsilon (t-t_0)^{-\r\,\gap },
$$
where the implicit constant in the above is independent of $t$ and $t_0$. Integrating over $t_0\in (t-\eta,t)\cap I_0$ (recall that $I_0$ has full Lebesgue measure), we obtain
$$
\eta+
\int_{t-\eta}^t \E\|u(t_0)\|_{\Xap}^{\r}\,\dd t_0\gtrsim_\varepsilon \eta^{1-\r\, \gap }.
$$
The claim of Step 1 now follows from Fubini's theorem, and taking $r$ sufficiently small depending only on the implicit constant in the above lower bound, $\r$ and $\gap$ to absorb the term $\eta\in (0,r)$ on the right-hand side of the previous lower bound.

\smallskip

\emph{Step 2: If Assumption \ref{ass:abstract1} is satisfied with $\O_n\equiv \O$, then \eqref{eq:reduction_claim_Minkowski} holds.}
In light of Step 1, the conclusion now follows by a covering argument similar to the one used in the deterministic case, cf.\ \cite[Theorem 13.5]{LePi} and \cite[Theorem 2.10]{K09_singular_times}. For the reader's convenience, we include some details. For each $\eta\in (0,r)$ (here $r$ is as in Step 1), we can write 
$$
\si^{\varepsilon,T} \subseteq  [0,\eta)\cup  \textstyle{\bigcup_{t\in \si^{\varepsilon,T}, t\geq \eta}}\, I_{\eta}(t), 
$$
where $I_\eta(t)=(t-\eta,t+\eta)$ denotes the interval with center $t$ and radius $\eta$. 
From Vitali's lemma (see e.g., \cite[Theorem 1.24]{EG_measure}), there exists a subset $(t_{k})_{k=1}^K\subseteq \si^{\varepsilon,T}\cap [\eta,\infty)$ such that $|t_{k}-t_{h}| \geq 2\eta$ for each $h\neq k$ (in other words, the intervals $I_\eta(t_h)$ and $I_{\eta}(t_k)$ are disjoint) and 
\begin{equation}
\label{eq:covering_property_subset}
\si^{\varepsilon,T}\subseteq   [0,\eta)\cup \textstyle{\bigcup_{k=1}^K} I_{5\eta} (t_{k}).
\end{equation}
Since $\si^{\varepsilon,T}\subseteq [0,T]$ has finite length, the covering set $(I_{\eta} (t_{k}))_{k=1}^K$ consists of finite intervals.
Recall that $\NN(\si^{\varepsilon,T},\eta)$ denotes the infimum number of balls needed to cover $\si^{\varepsilon,T}$ of radius $\eta\in (0,1)$, and thus $\NN(\si^{\varepsilon,T},5\eta)\leq 1+ K$ by \eqref{eq:covering_property_subset}. Thus, from the disjointness of $(I_{\eta} (t_{k}))_{k=1}^K$, it follows that 
\begin{align}
\label{eq:main_estimate_gap_l}
\eta^{1-\r\,\gap}
\NN(\si^{\varepsilon,T},5\eta) 
&\stackrel{(i)}{\lesssim}  \eta^{1-\r\,\gap}
+\sum_{k=1}^K \E \int_{t_k-\eta}^{t_k} \|u(t_0)\|_{\Xap}^{\r}\,\dd t_0\\
\nonumber
&\stackrel{(ii)}{\leq}  \eta^{1-\r\,\gap}
+\E \int_0^T \one_{\{\mathrm{dist}(\si^{\varepsilon,T},t_0)\leq \eta\}} \|u(t_0)\|_{\Xap}^{\r}\,\dd t_0,
\end{align}
where in $(i)$ and $(ii)$ we used Step 1 and $\bigcup_{k=1}^K I_{\eta}(t_k)\subseteq \{t_0\,:\,\mathrm{dist}(\si^{\varepsilon,T},t_0)\leq \eta\}$, respectively. In the previous, $\mathrm{dist}(\si^{\varepsilon,T},t_0)=\inf_{t\in \si^{\varepsilon,T}} |t-t_0|$ denotes the distance between $t_0$ and set $\si^{\varepsilon,T}$. 
From the compactness of $\si^{\varepsilon,T}$ (Lemma \ref{l:Borel_measurable} and \eqref{eq:sin_varepsilon_T}), we can infer 
$$
\lim_{\eta\downarrow0} \one_{\mathrm{dist}(\si^{\varepsilon,T},t_0)\leq \eta}=\one_{\si^{\varepsilon,T}}(t_0) \ \ \text{ for all }\tau\in [0,T];
$$
where the existence of the pointwise limit is a consequence of the monotonicity of $\big(\one_{\{\mathrm{dist}(\si^{\varepsilon,T},t_0)\leq \eta\}}(\tau)\big)_{\eta>0}$.
By letting $\limsup_{\eta\downarrow 0}$ in \eqref{eq:main_estimate_gap_l}, we obtain
\begin{equation}
\label{eq:bound_fondamental_hausdorff}
\M^{1-\r\,\gap}(\si^{\varepsilon,T})\lesssim\E \int_{\si^{\varepsilon,T}} \|u(t_0)\|_{\Xap}^{\r}\,\dd t_0.
\end{equation} 
From Assumption \ref{ass:abstract1} and \eqref{eq:HF_comparison}, it follows that 
$$\H^{1-\r\,\gap}(\si^{\varepsilon,T})\leq 
\M^{1-\r\,\gap}(\si^{\varepsilon,T})<\infty.
$$
Hence, $\H^{s}(\si^{\varepsilon,T})=0$ for all $s>1-\r\,\gap$ by Lemma \ref{l:basic_property_H}. In particular, as $\gap_\set\neq 0$, we have $|\si^{\varepsilon,T}|\leq \H^1(\si^{\varepsilon,T})=0$ (see e.g., \cite[Theorem 2.5]{EG_measure}).
From the previous and Assumption \ref{ass:abstract1}, we deduce that  
$$
 \int_{\si^{\varepsilon,T}}\E \|u(t_0)\|_{\Xap}^{\r}\,\dd t_0=0.
 $$ 
Combining this with \eqref{eq:bound_fondamental_hausdorff}, we proved \eqref{eq:reduction_claim_Minkowski}.
The bound on the corresponding dimension follows from the definition \eqref{eq:upper_box_dim}.
This concludes the proof of Step 2.

\smallskip

\emph{Step 3: The general case -- \eqref{eq:reduction_claim_Minkowski} holds.}
Suppose that Assumption \ref{ass:abstract1} holds with $\O_n\not\equiv \O$. Recall that $\varepsilon\in (0,1)$ and $T\in (0,\infty)$ are fixed. 
By assumption, there exists $n\geq 1$ such that 
\begin{equation}
\label{eq:choice_n_prob}
\P(\O\setminus \O_n)<\varepsilon/2.
\end{equation}
Let $I_0=(0,T)\setminus \Iz$ be as in Step 1 (see Assumption \ref{ass:abstract1}. For a.e.\ $t_0\in I_0$, we have
$
\E\big[\one_{\O_n}\|u(t_0)\|_{\Xap}^{\r}\big]<\infty.
$
Applying Proposition \ref{prop:quantitative_lower_bound} with initial datum $u(t_0)$, we obtain
\begin{align*}
\P(\tau\leq t_0+\g) &\leq C_0\g^{p\gap}(1+N^p) +\P(\|u(t_0)\|_{\Xap}>N)\\
&\leq C_0\g^{p\gap}(1+N^p) +\P(\O\setminus\O_n)+N^{-\r} \E\big[\one_{\O_n}\|u(t_0)\|_{\Xap}^{\r}\big].
\end{align*}
Arguing as in Step 1 and using \eqref{eq:choice_n_prob}, one obtains the existence of $R,r>0$ such that, for all $\eta\in (0,r)$ and $t\in \si^{\varepsilon,T}$ satisfying $t\geq \eta$, it holds that 
$$
\E\int_{t-\eta}^t \one_{\O_n} \|u(t_0)\|_{\Xap}^{\r}\,\dd t_0\geq R \,\eta^{1-\r \,\gap }.
$$
The conclusion follows as in Step 2, using Assumption \ref{ass:abstract1}.
\end{proof}

\subsubsection{Proof of Theorem \ref{t:singular_times_SPDEs_2}}
\label{ss:proofs_main_abstract_2}
As commented below Proposition \ref{prop:quantitative_lower_bound}, in the critical case $\gap_\set=0$, no explicit quantification of the lifetime of solutions to \eqref{eq:SEE} is possible. Therefore, the proof of Theorem \ref{t:singular_times_SPDEs_2} follows a different path compared to the one of Theorem \ref{t:singular_times_SPDEs}.

\begin{proof}[Proof of Theorem \ref{t:singular_times_SPDEs_2}]
Fix $\varepsilon\in (0,1)$ and $T<\infty$. Similar to the proof of Theorem \ref{t:singular_times_SPDEs}, to show \eqref{eq:statement_critical_case_singular_set}, it is enough to show that
\begin{equation}
\label{eq:critical_case_full_measure}
|\reg^{\varepsilon,T}|=T\qquad \text{ where }\qquad\reg^{\varepsilon,T}=\reg^{\set,\varepsilon} \cap [0,T].
\end{equation}
Indeed, the above readily implies the first in \eqref{eq:statement_critical_case_singular_set}. 
Moreover, combining this with Lemma \ref{l:Borel_measurable}, the set $\si^{\set,\varepsilon}\cap[0,T]$ is compact and has zero Lebesgue measure. Hence, a standard covering argument yields $\M^1(\si^{\set,\varepsilon}\cap[0,T])=0$ for all $T<\infty$. The second assertion in \eqref{eq:statement_critical_case_singular_set} now follows from \eqref{eq:Feta_definition}.

Let $\Iz$ be as in Assumption \ref{ass:abstract2}, and set $I_0\stackrel{{\rm def}}{=}(0,T)\setminus \Iz$.
By Proposition \ref{prop:quantitative_lower_bound} and strong weak-strong uniqueness, it follows that for each $t\in I_0$ there exists $\delta_t >0$ such that $(t,t+\delta_t)\subseteq \reg^{\varepsilon,T}$ (see \eqref{eq:u_equal_ustrong}-\eqref{eq:regularity_strong_solution} for a similar situation). Hence, 
\begin{equation}
\label{eq:reg_supset_critical_case}
\reg^{\varepsilon,T}\supseteq \textstyle{\bigcup_{t\in \mathcal{I}_T} }(t,t+\delta_t). 
\end{equation}
From the Lebesgue differentiation theorem (adapted to the one-sided maximal function $M_+f (t)=\sup_{r>0}\frac{1}{r}
\int_t^{t+r}f(t')\,\dd t' $), it follows that there exists a set of full Lebesgue measure $I_0' \subseteq I_0$ such that, for all $t_0\in I_0'$,  
\begin{align*}
\one_{\reg^{\varepsilon,T}} (t_0)
= \lim_{r\downarrow 0} 
\frac{1}{r}\int_{t_0}^{t_0+r} 
\one_{\reg^{\varepsilon,T}} (t')\,\dd t' 
&\stackrel{\eqref{eq:reg_supset_critical_case}}{\geq}
 \limsup_{r\downarrow 0} 
\frac{1}{r}\int_{t_0}^{t_0+r} 
\one_{(t_0,t_0+\delta_{t_0})} (t')\,\dd t' \\
&\ =  \limsup_{r\downarrow 0} 
\frac{1}{r}|(t_0,t_0+(\delta_{t_0}\wedge r))|=1. 
\end{align*}
Thus, $
\one_{\reg^{\varepsilon,T}}\equiv 1$ a.e.\ on $(0,T)$ and \eqref{eq:critical_case_full_measure} follows.
\end{proof}

\subsubsection{Proof of Proposition \ref{prop:singular_times_SPDEs_weaker_weak_strong}}
\label{ss:proof_singular_times_SPDEs_weaker_weak_strong}
The proof of Proposition \ref{prop:singular_times_SPDEs_weaker_weak_strong} now follows from a minor modification of the just proven results. 

\begin{proof}[Proof of Proposition \ref{prop:singular_times_SPDEs_weaker_weak_strong}]
We content ourselves to comment on the modifications needed in Theorem \ref{t:singular_times_SPDEs} to obtain the statements \eqref{eq:singular_times1} and \eqref{eq:singular_times2} under the weaker weak-strong uniqueness assumption of Proposition \ref{prop:singular_times_SPDEs_weaker_weak_strong}. The case of Theorem \ref{t:singular_times_SPDEs_2} is completely analogous. 

To prove \eqref{eq:singular_times1} and \eqref{eq:singular_times2} in the current situation, it is enough to modify Step 1 in Theorem \ref{t:singular_times_SPDEs}. More precisely, we have to justify the validity of \eqref{eq:u_equal_ustrong} for all $t_0\in (t-\eta,t)\cap I_0$, where $I_0\subseteq (0,T)$ is a set of full Lebesgue measure (here $T<\infty$ is fixed, cf.\ the beginning of Theorem \ref{t:singular_times_SPDEs}).
Note that, by Assumption \ref{ass:abstract1} and Assumption \ref{ass:abstract2} with $\set$ replaced by $\sety$, there exists a set of full measure $I_0\subseteq (0,T)$ such that, for all $t_0\in I_0$,
$$
u(t_0)\in L^0_{\F_{t_0}}(\O;\Xap)\cap L^0_{\F_{t_0}}(\O;\Yar).
$$
Let $(v^\set,\tau^\set)$ and $(v^\sety,\tau^\sety)$ the maximal solution to \eqref{eq:SEE} with initial data $u(t_0)$ at time $t=t_0$ in the setting $\set$ and $\sety$, respectively.
From Assumption \ref{ass:abstract2} with $\set$ replaced by $\sety$, we have 
$
u=v^\sety$  a.e.\ on $[t_0,\tau^\sety)\times \O.$
From the compatibility of $\set$ and $\sety$ (see \eqref{eq:setting_compatibility}) and the previous, it follows that 
$$
u=v^\set\quad \text{  a.e.\ on }[t_0,\tau^\set)\times \O.
$$
That is exactly \eqref{eq:u_equal_ustrong}, as desired. The rest of the proof stays unchanged.
\end{proof}

\subsection{Lower bounds of lifetime of solutions -- Proof of Proposition \ref{prop:quantitative_lower_bound}}
\label{ss:proof_lower_bound_existing_times}
To prove Proposition \ref{prop:quantitative_lower_bound}, we partially retrace the proof of the local existence of \cite[Theorem 4.5]{AV19_QSEE1} (see also \cite[Section 4]{AV25_survey}). In particular, we sharpen the arguments in Steps 1-3 of \cite[Theorem 4.5]{AV19_QSEE1} to obtain information about the lifetime of local solutions built via contraction.
Before going into the details, let us discuss the motivations for the spaces appearing below. 
First, note that, if Assumption \ref{ass:criticality} holds, then for all $t>0$
\begin{align}
\label{eq:control_nonlinearities_FG}
\|F(\cdot,v)-F(\cdot,0)\|_{L^p(0,t,w_{\a};X_0)}&+\|G(\cdot,v)-G(\cdot,0)\|_{L^p(0,t,w_{\a};\g(H,X_{1/2}))}\\
\nonumber
&\lesssim \sum_{j=1}^m \big\|(1+\|v\|_{X_{\beta_j}}^{\rho_j})\|v\|_{X_{\beta_j}}\big\|_{L^p(0,t,w_{\a})}\\
\nonumber
&\lesssim \sum_{j=1}^m (1+ \|v\|_{L^{p(\rho_j+1)}(0,t,w_{\a};X_{\beta_j})}^{1+\rho_j}).
\end{align}
Hence, the weakest space that allows us to control the nonlinearities $F$ and $G$ in an $L^p(w_\a)$ is given by 
$$
\Y_t\stackrel{{\rm def}}{=}\textstyle{\bigcap_{j=1}^m}\, L^{p(\rho_j+1)}(0,t,w_{\a};X_{\beta_j}),
$$
endowed with the natural norm. Recall that if $(A,B)$ has stochastic maximal $L^p$-regularity in the $\set$-setting (see the text above \eqref{eq:stochastic_maximal_Lp_regularity} for the definition) the solution to the linearization of \eqref{eq:SEE} has a.s.\ paths in the space
$
C([0,t];\Xap)\cap L^p(0,t,w_{\a};X_1).
$
In particular, by standard interpolation inequalities, the sharp embedding of the latter into a space of spatial regularity $X_\beta$ reads as follows: 
\begin{equation}
\label{eq:embedding_sharp_theta_interpolation}
C([0,t];\Xap)\cap L^p(0,t,w_{\a};X_1)\embed L^{p/\theta_j}(0,t,w_{\a};X_{\beta_j})
\end{equation}
where the constant in the above embedding is independent of $t>0$ and
$$
\theta_j\stackrel{{\rm def}}{=}\frac{p}{1+\a}\Big(\beta_j-1+\frac{1+\a}{p}\Big).
$$
Note that the space on the RHS of \eqref{eq:embedding_sharp_theta_interpolation} has space-time Sobolev index 
$$-\theta_j\frac{1+\a}{p}-\beta_j=1-\frac{1+\a}{p};$$ 
see \eqref{eq:supercriticality_Lellbound}, which is that of $C([0,t];\Xap)$ and $L^p(0,t,w_{\a};X_1)$.
The following shows that the maximal $L^p$-regularity in the $\set$-setting embeds into $\Y_t$. Moreover, we obtain explicit bounds in terms of the excess of smoothness of $\Xap$ compared to the critical threshold $
\ggap_{\set,j}$, see \eqref{eq:excess_criticality_2}.

\begin{lemma}
\label{lem:criticality_gap}
Let $\rho>0$ and $\beta\in (1-\frac{1+\a}{p},1)$ be such that 
$
\beta\leq 1-\frac{\rho}{\rho+1} \frac{1+\a}{p}.
$
Let $\ggap_{\set}$ be the excess from the criticality, i.e.,
$$
\ggap_{\set}=1-\beta-\frac{\rho}{\rho+1} \frac{1+\a}{p}.
$$
Then, there exists a constant $C_0>0$ such that, for all $u\in L^{p/\theta}(0,T,w_\a;X_{\beta})$, 
$$
\|u\|_{L^{p(\rho+1)}(0,T,w_\a;X_\beta)}\leq C_0 T^{\ggap_{\set}}
\|u\|_{L^{p/\theta}(0,T,w_\a;X_{\beta})}
$$
where $\theta=\frac{p}{1+\a}(\beta_j-1+\frac{1+\a}{p})$.
\end{lemma}

\begin{proof}
Letting $\frac{1}{r}+\frac{\theta}{p}=\frac{1}{p(\rho+1)}$, the H\"older inequality yields
$$
\|u\|_{L^{p(\rho+1)}(0,T,w_\a;X_\beta)}
\lesssim T^{(1+\kappa)/r} \|u\|_{L^{p/\theta}(0,T,w_{\a};X_\beta)}.
$$
Since
\begin{align*}
\frac{1}{r}
= \frac{1}{p(\rho+1)}-\frac{1}{1+\a}\Big(\beta-1+\frac{1+\a}{p}\Big)
= \frac{1}{1+\a} \Big(1-\beta-\frac{\rho}{\rho+1} \frac{1+\a}{p}\Big),
\end{align*}
the claim Lemma \ref{lem:criticality_gap} follows.
\end{proof}

From \eqref{eq:control_nonlinearities_FG}, \eqref{eq:embedding_sharp_theta_interpolation} and Lemma \ref{lem:criticality_gap}, it follows that the space 
$$
\X_t := 
\textstyle{\bigcap_{j=1}^m}\,L^{p/\theta_j} (0,t,w_\a;X_{\beta_j})
$$
controls the nonlinearities $F$ and $G$, and moreover 
\begin{equation}
\label{eq:embedding_Ctrace_LpX1_Xt}
C([0,t];\Xap)\cap L^p(0,t,w_{\a};X_1)\embed \X_t,
\end{equation}
with an embedding constant independent of $t>0$. 

We are now in the position of proving Proposition \ref{prop:quantitative_lower_bound}.

\begin{proof}[Proof of Proposition \ref{prop:quantitative_lower_bound}]
As commented below the statement of Proposition \ref{prop:quantitative_lower_bound}, from \cite[Theorem 4.8]{AV19_QSEE1}, it is enough to consider the case 
$$
\gap_\set>0.
$$ 
For notational convenience, we prove the claim only for $t=0$ and $R=1$. The general case follows similarly. 
Following the proof of \cite[Theorem 4.5]{AV19_QSEE1}, we begin by considering the following cutoff version of \eqref{eq:SEE}:
\begin{equation}
\label{eq:SEE_cutoff}
\left\{
\begin{aligned}
&\dd\ulambda + A(\cdot) \ulambda\,\dd t =[ \phi_{\lambda} (\cdot,\ulambda)(F(\cdot,\ulambda)-F(\cdot,0))+f]\,\dd t \\
&\qquad \qquad \qquad+ [B(\cdot)\ulambda +  \phi_{\lambda} (\cdot,\ulambda)(G(\cdot,\ulambda)-G(\cdot,0))+g]\, \dd W,\\
&\ulambda(0)=u_0,
\end{aligned}
\right.
\end{equation}
where, $f=F(\cdot,0)$, $g=G(\cdot,0)$, 
$$
\phi_\lambda(t,u)= \phi(\|u\|_{\Y_t}/\lambda) \ \  \text{ for }  \ \lambda >0,\ t\in [0,T],\ u\in \X_T,
$$
and $\phi$ is a cutoff function on $[0,\infty)$ such that $\phi=1$ on $[0,1]$ and $\phi=0$ on $[2,\infty)$. 
Strong solutions to \eqref{eq:SEE_cutoff} in the $\set$-setting are defined as in Definition \ref{def:solutions_SEE}.
We now divide the proof into several steps.

\smallskip

\emph{Step 1: (Locally Lipschitz estimate for the truncated nonlinearities). For $v,v'\in \X_T$ for some $T>0$, let 
$$
F_\lambda(\cdot,v)\stackrel{{\rm def}}{=} \phi_{\lambda}(\cdot,v)(F(\cdot,v)-F(\cdot,0))\ \  \text{ and } \ \ 
G_\lambda(\cdot,v)\stackrel{{\rm def}}{=} \phi_{\lambda}(\cdot,v)(G(\cdot,v)-G(\cdot,0)).
$$
There exists a constant $C_1>0$ such that, for all $\lambda\in [1,\infty)$, $T\in (0,1]$ and $v,v'\in X_T$, it holds that}
\begin{align*}
\|F_{\lambda}(\cdot,v)-F_\lambda(\cdot,v')\|_{L^p(0,T,w_{\a};X_0)}
&+
\|G_{\lambda}(\cdot,v)-G_{\lambda}(\cdot,v')\|_{L^p(0,T,w_{\a};\g(H,X_{1/2}))}\\
&\leq C_1 \lambda^{\max_j\rho_j }T^{\min_j\ggap_{\set,j}} \|v-v'\|_{\X_T},
\end{align*}
The proof is similar to the one given in \cite[Lemma 4.13]{AV19_QSEE1}. However, given its central importance in our proofs, we provide some details. Below, we only consider $F_\lambda$, as the locally Lipschitz estimate for $G_\lambda$ is analogous.
Fix $v,v'\in \X_T$. For $w\in \{v,v'\}$,
$$
\tau_{w}\stackrel{{\rm def}}{=}
\inf\{t\in [0,T]\,:\, \|w\|_{\Y_t}\geq 2\lambda\},
$$
where $\inf\emptyset \stackrel{{\rm def}}{=}T$. Without loss of generality, we assume $\tau_v\leq \tau_{v'}$. We now decompose the truncated nonlinearity $F_{\lambda}(\cdot,v)$ as follows
\begin{align}
\label{eq:decomposition_F_lambda}
F_{\lambda}(\cdot,v)-
F_{\lambda}(\cdot,v')
&= 
\phi_{\lambda}(\cdot,v)\big( F(\cdot,v)-F(\cdot,v')\big)\\
\nonumber
&+\big(\phi_{\lambda}(\cdot,v)-\phi_{\lambda}(\cdot,v')\big)(F(\cdot,v')-F(\cdot,0)).
\end{align}
Now, we estimate each term separately. Firstly, as $\phi_{\lambda}(t,v)=0$ for $t\geq \tau_v$, we obtain
\begin{align*}
&\|\phi_{\lambda}(\cdot,v)\big( F(\cdot,v)-F(\cdot,v')\big)\|_{L^p(0,T,w_{\a};X_0)}\\
&=\|\phi_{\lambda}(\cdot,v)\big( F(\cdot,v)-F(\cdot,v')\big)\|_{L^p(0,\tau_v,w_{\a};X_0)}\\
&\lesssim \textstyle{\sum_{j=1}^m } (1+ \|v\|_{L^{p(\rho_j+1)}(0,\tau_v,w_{\a};X_{\beta_j})}^{\rho_j}
+ \|v'\|_{L^{p(\rho_j+1)}(0,\tau_v,w_{\a};X_{\beta_j})}^{\rho_j}) \\
&\qquad \qquad \qquad \qquad\qquad \qquad \qquad \qquad \qquad \qquad 
 \|v-v'\|_{L^{p(\rho_j+1)}(0,\tau_v,w_{\a};X_{\beta_j})}\\
&\stackrel{(i)}{\leq} \textstyle{\sum_{j=1}^m } (1+ 2(2\lambda)^{\rho_j}) 
\|v-v'\|_{L^{p(\rho_j+1)}(0,T,w_{\a};X_{\beta_j})}\\
&\stackrel{(i)}{\leq} \textstyle{\sum_{j=1}^m } (1+ 2(2\lambda)^{\rho_j}) 
T^{\ggap_{\set,j}}\|v-v'\|_{L^{p/\theta_j}(0,T,w_{\a};X_{\beta_j})}\\
&\stackrel{(ii)}{\lesssim} \lambda^{\max_j\rho_j}T^{\inf_j\ggap_{\set,j}} 
\|v-v'\|_{\X_T},
\end{align*}
where in $(i)$ we used $\tau_v\leq \tau_{v'}$ also implies $\|v'\|_{L^{p(\rho_j+1)}(0,\tau_v,w_{\a};X_{\beta_j})}\leq 2\lambda$, and in $(ii)$ that $\lambda\geq 1$ and $T\leq 1$. 
Secondly, recalling that $\tau_{v}\leq \tau_{v'}$, both the cutoffs $\phi_{\lambda}(\cdot,v)$ and $\phi_\lambda(\cdot,v')$ vanish on the time interval $(\tau_{v'},T)$, we have
\begin{align*}
&\|(\phi_{\lambda}(\cdot,v)-\phi_{\lambda}(\cdot,v')\big)(F(\cdot,v')-F(\cdot,0))\|_{L^p(0,T,w_\a;X_0)}\\
&\quad \quad= \|(\phi_{\lambda}(\cdot,v)-\phi_{\lambda}(\cdot,v'))(F(\cdot,v')-F(\cdot,0))\|_{L^p(0,\tau_{v'},w_\a;X_0)}\\
&\quad \quad\leq \big(\textstyle{\sup_{ [0,T]}}\big|\phi_{\lambda}(\cdot,v)-\phi_{\lambda}(\cdot,v')\big|\big)\, \|F(\cdot,v')-F(\cdot,0)\|_{L^p(0,\tau_{v'},w_\a;X_0)}\\
&\quad \quad\leq \lambda^{-1} \|v-v'\|_{\Y_T}  \big(\textstyle{\sum_{j=1}^m} (1+\|v'\|_{L^{p(\rho_j+1)}(0,\tau_{v'},w_\a;X_{\beta_j})}^{\rho_j})\|v'\|_{L^{p(\rho_j+1)}(0,\tau_{v'},w_\a;X_{\beta_j})}\big)\\
&\quad \quad\leq  \|v-v'\|_{\X_T}  \big(\textstyle{\max_{j} }T^{\ggap_{\set,j}}\big)  \big(\textstyle{\sum_{j=1}^m} \lambda^{\rho_j}\big)\\
&\quad \quad\leq  \|v-v'\|_{\X_T}  T^{\min_j\ggap_{\set,j}} \lambda^{\max_j\rho_j},
\end{align*}
where again we used $T\leq 1$ and $\lambda\geq 1$. 
Hence, the estimate in Step 1 follows by combining the previous findings.

\smallskip

\emph{Step 2: Suppose that $u_0\in L^p_{\F_0}(\O;\Xap)$. There exist constants $R_0>0$, $C_0\geq 1$ and $T_0\in (0,1]$ for which the following assertions hold. For all 
\begin{equation}
\label{eq:parameters_bounds_step_2}
\lambda= R_0 T^{-(\min_j \ggap_{\set,j})/ (\max_j\rho_j)}  \quad \text{ and }\quad T\in (0,T_0],
\end{equation}
then \eqref{eq:SEE_cutoff} has a global unique solution $\ulambda$ on $[0,T]$ in the $\set$-setting satisfying}
\begin{equation}
\label{eq:step_2_bound_ulambda}
\E\Big[\sup_{ [0,T]} \|\ulambda\|_{\Xap}^p\Big] + \E\int_0^T \|\ulambda\|^p_{X_1}\,w_\a \,\dd t 
\leq C_0 (1+ \,\E\|u_0\|^p_{\Xap}).
\end{equation}

To prove the claim of Step 2, we apply the contraction principle similar to Steps 1 and 2 in the proof of \cite[Theorem 4.5]{AV19_QSEE1}. 
For notational convenience, we use the following shorthand notation for the space on which we apply the contraction principle:
$$ 
\mathscr{M}_{\set}(T)\stackrel{{\rm def}}{=}L^p_{\Progress}(\O;C([0,T];\Xap))\cap L^p_{\Progress}(\O\times (0,T),w_{\a};X_1),
$$ 
where $\set=(X_0,X_1,p,\a)$ is the setting. 
For $\lambda\geq 1$ and $T\leq 1$, consider the mapping 
\begin{align*}
\Pi: \mathscr{M}_{\set}(T)\to \mathscr{M}_{\set}(T), \qquad
v\mapsto w, 
\end{align*}
where $w$ is the strong solution to  
\begin{equation}
\label{eq:SEE_cutoff_2}
\left\{
\begin{aligned}
&\dd w + A(\cdot) w\,\dd t =[ F_\lambda(\cdot,v)+f]\,\dd t + [B(\cdot)w + G_\lambda(\cdot,v)+g]\, \dd W.\\
&w(0)=u_0,
\end{aligned}
\right.
\end{equation}

To conclude the proof of this step, it suffices to show that $\Pi$ maps $\mathscr{M}_{\set}(T)$ into itself, and 
\begin{equation}
\label{eq:mapping_property_of_Pi}
\|\Pi \,v-\Pi \,v'\|_{\mathscr{M}_{\set}(T)}\leq C_1C_2 \lambda^{\max_j\rho_j }T^{\min_j\ggap_{\set,j}}
\|v- v'\|_{\mathscr{M}_{\set}(T)},
\end{equation}
where $C_2>0$ depends only on $m,p,\a$ and the constant in the stochastic maximal $L^p$-regularity constant for the couple $(A,B)$ on $[0,1]$ (here we used that $R=1$). 

Before proving \eqref{eq:mapping_property_of_Pi}, let us first show how this implies the claim of this step. 
To this end, letting $R_0=(2 C_2 C_1)^{-1/(\max_j \rho_j)}$ and $T_0= 1\wedge (2C_1 C_2)^{-1/(\min_j\ggap_{\set,j})}$, one can check that if $\lambda$ and $T$ satisfy \eqref{eq:parameters_bounds_step_2} with the previous choices, one has $\lambda\geq 1$ and
\begin{equation}
\label{eq:constant_bounds_onehalf}
C_1 C_2 \lambda^{\max_j\rho_j }T^{\min_j\ggap_{\set,j}}\leq 1/2.
\end{equation} 
Therefore, $\Pi$ is a contraction on $\mathscr{M}_{\set}(T)$. Let $\ulambda$ be the unique fixed point of $\Pi$. Thus, $\ulambda$ is the unique global solution to \eqref{eq:SEE_cutoff} on $[0,T]$. By writing $\ulambda=  \Pi \,0  + (\Pi\, \ulambda -\Pi\,0)$ and using \eqref{eq:mapping_property_of_Pi}--\eqref{eq:constant_bounds_onehalf}, one obtains 
\begin{align*}
\|\ulambda\|_{\mathscr{M}_{\set}(T)}
&
\leq 
 \|\Pi\,0\|_{\mathscr{M}_{\set}(T)}+\|\Pi\, \ulambda-\Pi\,0\|_{\mathscr{M}_{\set}(T)}\\
&
\leq \|\Pi\,0\|_{\mathscr{M}_{\set}(T)}+(1/2)\|\ulambda\|_{\mathscr{M}_{\set}(T)},
\end{align*}
which yields the estimate \eqref{eq:step_2_bound_ulambda}. In the remaining part of the step, we show \eqref{eq:mapping_property_of_Pi}. Analogously, one obtains that $\Pi$ maps $\mathscr{M}_{\set}(T)$ into itself.
Note that $V\stackrel{{\rm def}}{=}\Pi\,v-\Pi\,v'\in \mathscr{M}_{\set}(T)$ solves 
\begin{equation*}
\left\{
\begin{aligned}
&\dd V + A(\cdot) V\,\dd t =[ F_\lambda(\cdot,v)-F_\lambda(\cdot,v')]\,\dd t + [B(\cdot)V + G_\lambda(\cdot,v)-G_\lambda(\cdot,v')]\, \dd W.\\
&V(0)=0,
\end{aligned}
\right.
\end{equation*}
Hence, the estimate \eqref{eq:mapping_property_of_Pi} and the claims in Step 2 are a direct consequence of the stochastic maximal $L^p_\a$-regularity of $(A,B)$ and the estimate proven in Step 1.

\smallskip

\emph{Step 3: Suppose that $u_0\in L^p_{\F_0}(\O;\Xap)$. Let $T_0>0$ be as in Step 2.  
Then there exists a constant $C_3$ independent of $u_0$ for which the maximal solution $(v,\tau)$ to the abstract SPDE \eqref{eq:SEE} in the $\set$-setting satisfies}
$$
\P(\tau\leq T) \leq C_3 T^{p\, \gap_{\set}} (1+\E \|u_0\|_{\Xap}^p)\ \  \text{ for all }\ T\leq T_0.
$$
Fix $\lambda= R_0 T^{-(\min_j \ggap_{\set,j})/ (\max_j\rho_j)} $, where $R_0$ is as in Step 2. 
Let $\ulambda$ be the corresponding solution to \eqref{eq:SEE_cutoff} on $[0,T]$ in the $\set$-setting. Define
$$
\tau_\lambda \stackrel{{\rm def}}{=}\inf\{t\in [0,T]\,:\, \|\ulambda\|_{\Y_t}\geq \lambda\} \quad \text{ and }\quad \inf\emptyset\stackrel{{\rm def}}{=}T.
$$
Clearly, $\tau$ is a stopping time. Moreover, 
$$
\phi_{\lambda}(\cdot,\ulambda)|_{[0,\tau_\lambda]}=1 \text{ a.s.\ }
$$
Hence, it follows that $(\ulambda,\tau_\lambda)$ is a local strong solution to the original problem \eqref{eq:SEE} in the $\set$-setting (see also Steps 3 and 4 in the proof of \cite[Theorem 4.5]{AV19_QSEE1}). 
Let $(v,\tau)$ be the maximal solution to \eqref{eq:SEE}. From Definition \ref{def:solutions_SEE}, it follows that 
\begin{equation}
\label{eq:u_ulambda_identity_estimate}
\tau_\lambda\leq \tau \ \text{ a.s.\ } \quad \text{ and } \quad u_\lambda=v \ \text{ on } [0,\tau]\times \O.
\end{equation}
Moreover, from Lemma \ref{lem:criticality_gap} and the above choice of $\lambda$, we infer
\begin{align*}
\P(\tau_\lambda< T)
= \P(\|\ulambda\|_{\Y_T}\geq \lambda)
&\lesssim \max_j\, \P(\|\ulambda\|_{L^{p(\rho_j+1)}(0,T,w_{\a};X_{\beta_j})}\geq \lambda)\\
&\leq\P(T^{\max_j\ggap_{\set,j}}\|\ulambda\|_{L^{p/\theta_j}(0,T,w_{\a};X_{\beta_j})}\geq C_{0}\lambda)\\
&\lesssim T^{- p\, \max_j \ggap_{\set,j}}\lambda^{-p}\, \E \|\ulambda\|_{\X_T}^p\\
&\lesssim T^{p\, \gap_{\set}}(1+ \E \|u_0\|_{\Xap}^p),
\end{align*}
where we also used the definition of $\gap_{\set}$ in \eqref{eq:excess_criticality_1} and the embedding \eqref{eq:embedding_Ctrace_LpX1_Xt}. Thus, by choosing a sequence $(T_n)_n\subset \R_+$ such that $T_n\downarrow T$, it follows from \eqref{eq:u_ulambda_identity_estimate} that
$$
\P(\tau\leq T)=\lim_{n\to \infty} \P(\tau<T_n) \lesssim T^{p\, \gap_{\set}}(1+ \E \|u_0\|_{\Xap}^p).
$$

\smallskip

\emph{Step 4: Conclusion -- Proof of \eqref{eq:sigma_zero_lower_bound} in the case $t=0$.} 
Given the result of Step 3, it remains to apply a standard localization argument. Let $u_0\in L^0_{\F_0}(\O;\Xap)$, and for $N\geq 0$, set
$$
\O_N\stackrel{{\rm def}}{=}\big\{ \|u_0\|_{\Xap}\leq N \big\}\in \F_0.
$$
Clearly, $\one_{\O_N}u_0\in L^p_{\F_0}(\O;\Xap)$ with norm bounded by $N$.
Steps 1-3 yield that the maximal solution $(v_N,\tau_N)$ to \eqref{eq:SEE} in the $\set$-setting with initial data $\one_{\O_N}u_0$ satisfying 
\begin{equation}
\label{eq:tau_N_T_estimate}
\P(\tau_N\leq T)
\lesssim T^{p\, \gap_\set} (1+N^p).
\end{equation}
It follows that the maximal solution $(v,\tau)$ to \eqref{eq:SEE} in the $\set$-setting (see Definition \ref{def:solutions_SEE}) satisfies $\tau=\tau_N$ a.s.\ and $v=v_N$ a.e.\ on $[0,\tau_N)\times \O$. Moreover,
\begin{align*}
\P(\tau\leq T)
&\leq \P(\{\tau\leq T\}\cap \O_N)+ \P(\O\setminus\O_N)\\
& \leq \P(\tau_N \leq T)+ \P(\|u_0\|_{\Xap}>N)\\
&\lesssim T^{\gap_\set}(1+ N^p)+ \P(\|u_0\|_{\Xap}>N),
\end{align*}
where the last step follows from \eqref{eq:tau_N_T_estimate}.
This proves \eqref{eq:sigma_zero_lower_bound} in the case $t=0$ and $T\in (0,T_0]$. The remaining cases $T\geq T_0$ follow by enlarging the implicit constant in the above estimate, if necessary.
\end{proof}

\section{Singular times of stochastic 3D Navier-Stokes equations}
\label{s:NS}
Here, we apply the results of Section \ref{s:singular_times_SPDEs} to investigate the fractal dimension of \emph{singular times} of Leray-Hopf-type solutions to the following stochastic 3D NSEs:
\begin{equation}
\label{eq:SNS_introduction}
\left\{
\begin{aligned}
\partial_t u 
&= \Delta u -\nabla p -(u\cdot \nabla)u \\
&+\sum_{n\geq 1} \big[-\nabla \wt{p}_n + (\sigma_n\cdot \nabla) u + \cc_n\cdot u \big] \circ \dot{W}^n&\quad & \text{ on }\T^3,\\
\nabla \cdot u&=0&\quad & \text{ on }\T^3,\\
u(0)&=u_0& \quad& \text{ on }\T^3,
\end{aligned}
\right.
\end{equation}
where $u:[0,\infty)\times \O\times \T^3 \to \R^3$ denotes the unknown velocity field, $p:[0,\infty)\times \O\times \T^3 \to \R$ and $\wt{p}_n:[0,\infty)\times \O\times \T^3 \to \R$ denote the unknown deterministic and turbulent pressures, $(W^n)_{n\geq 1}$ is a sequence of standard independent Brownian motions on a filtered probability space $(\O,\F,(\F_t)_{t\geq 0}, \P)$, and $\circ$ stands for the Stratonovich integration. 
Moreover, on the coefficients $\sigma_n$ and $\cc_n$, we enforce the following conditions.

\begin{assumption}[Noise regularity] 
\label{ass:NSE_regularity_coefficients}
The following are satisfied.
\begin{enumerate}[{\rm(1)}]
\item\label{it:NSE_regularity_coefficients1} 
For all $n\geq 1$, the following vector fields are $\Borel([0,\infty)\times \T^3)$-measurable 
$$
\sigma_n: [0,\infty)\times \T^3\to \R^3\quad \text{ and }\quad 
\cc_n : [0,\infty)\times \T^3\to \R^{3\times 3}.
$$
\item\label{it:NSE_regularity_coefficients2} 
There exist constants $M\geq 1$ and $\g>0$ such that, for all $t\in \R_+$,  
$$
\|(\sigma_n(t))_{n}\|_{C^{\g}(\T^3;\ell^2(\R^3))}\leq M \quad \text{ and }\quad
\|(\cc_n(t))_{n}\|_{C^{\g}(\T^3;\ell^2(\R^{3\times 3}))}\leq M.
$$
\item\label{it:NSE_regularity_coefficients3} 
For all $n\geq 1$ and $t\in\R_+$, 
$$
\nabla \cdot \sigma_n(t) =0 \text{ in distributions on }\T^3.
$$
\end{enumerate}
\end{assumption}

We point out that Assumption \ref{ass:NSE_regularity_coefficients}\eqref{it:NSE_regularity_coefficients3} can be weakened to the requirement $\sup_{ \R_+}\|(\nabla \cdot \sigma_n)_{n\geq 1}\|_{\ell^2(\T^3;\ell^2)} <\infty$.
We leave the details to the interested reader. 

\smallskip

Physical motivations for the model \eqref{eq:SNS_introduction} are given in Subsection \ref{ss:NSEs_time_intro}.
Let us recall that the above setting covers the following two situations of physical interest:
\begin{itemize}
\item \emph{Rough Kraichnan noise:} $(\sigma_n)_n \in C^{\g}(\T^3;\ell^2)$ and $\cc_n\equiv 0$.
\item \emph{Advection by Lie transport:} $(\sigma_n)_n \in C^{1+\g}(\T^3;\ell^2)$ and $\cc_n = \nabla \sigma_n$. 
\end{itemize}
Moreover, in the case $\g=\frac{2}{3}$ and $\cc_n\equiv 0$, the above assumptions allow for Kraichnan noise reproducing the Kolmogorov spectrum of turbulence; see 
\cite[pp. 426-427 and 436]{MK99_simplified}, \cite[Remark 5.3]{GY25} or \cite[eq.\ (1.4)-(1.5)]{Primitive3}.

\smallskip

This section is organized as follows. In Subsection \ref{ss:preliminaries_SNS}, we collect some preliminary facts, including function spaces and energy inequalities. The main results on annealed stochastic Leray-Hopf solutions are stated in Subsection \ref{ss:annealed_strong_energy}. The corresponding proofs and a weak-strong uniqueness for such solutions are given in Subsection \ref{ss:annealed_LH}. Finally, in Subsection \ref{ss:existence_strong_Leray_proof}, we give a short proof of the existence of annealed stochastic Leray-Hopf solutions to the 3D NSEs \eqref{eq:SNS_introduction}. 

\subsection{Preliminaries}
\label{ss:preliminaries_SNS}
In this subsection, we discuss some preliminary facts needed to formulate our main results on stochastic 3D NSEs \eqref{eq:SNS_introduction} in Subsection \ref{ss:annealed_strong_energy}. More precisely, in Subsections \ref{sss:function_spaces_divfree} and \ref{sss:ito_stratonovich_operators}, we discuss the relevant function spaces to study the NSEs and the (formal) Stratonovich-to-It\^o transformation of \eqref{eq:SNS_introduction} and the associated energy balance.

\subsubsection{Function spaces of divergence-free vector fields}
\label{sss:function_spaces_divfree}
We begin by introducing the relevant function spaces on $\T^d$, which will be used throughout this section. As usual, 
Bessel-potential spaces are indicated by $H^{s,q}(\T^d)$ where $s\in \R$ and $q\in (1,\infty)$. 
We also use the standard short-hand notation $H^{s}(\T^d)$ for $H^{s,2}(\T^d)$. 
We sometimes also employ Besov spaces $B^{s}_{q,p}(\T^d)$ which can be defined as the real interpolation space $(H^{-k,q}(\T^d),H^{k,q}(\T^d))_{(s+k)/2k,p}$ for $s\in \R$, $\N\ni k>|s|$ and $1<q,p<\infty$. For explicit characterization via Littlewood-Paley decomposition, see e.g., \cite[Section 3.5.4]{schmeisser1987topics} for details on function spaces over $\T^d$. Moreover, we set $\mathcal{A}(\T^d;\R^k)\stackrel{{\rm def}}{=}(\mathcal{A}(\T^d))^k$ and $\mathcal{A}(\cdot)\stackrel{{\rm def}}{=}\mathcal{A}(\T^d;\cdot)$ where $\mathcal{A}\in \{L^q,H^{s,q},B^{s}_{q,p}\}$ and $k\in \N$.
Similar definitions hold in the case of Hilbert space-valued functions; see \cite[Chapter 14]{Analysis3}.

Next, we define the Helmholtz projection $\p$ and the complement projection 
$
\q\stackrel{{\rm def}}{=} \mathrm{Id}- \p.
$
For an $\R^d$-valued distribution $f=(f^i)_{i=1}^{d}\in \D'(\T^d;\R^d)$ on $\T^d$, let $\widehat{\,f^i\,}(k) =\langle e_k ,f^i\rangle $ be $k$-th Fourier coefficients, where $i\in \{1,\dots,d\}$, $k=(k_i)_{i=1}^d\in \Z^d$ and $e_k(x)=e^{2\pi \i k\cdot x}$. The Helmholtz projection $\p f$ for $f\in \D'(\T^d;\R^d)$ is given by
$$
(\widehat{ \p f})^i (k)\stackrel{{\rm def}}{=}\widehat{\,f^i\,}(k)- \sum_{1\leq j\leq d} \frac{k_i k_j }{|k|^2} \widehat{\,f^j\,}(k), \qquad 
(\widehat{ \p f})^i (0)\stackrel{{\rm def}}{=}\widehat{\,f^i\,}(0).
$$
Formally, $\p f$ can be written as $f- \nabla \Delta^{-1}(\nabla \cdot f)$.
From standard Fourier analysis, it follows that $\q$ and $\p$ restrict to bounded linear operators on $H^{s,q}(\T^d;\R^d)$ and $B^s_{q,p}(\T^d;\R^d)$ for $s\in \R$ and $q\in (1,\infty)$.
Finally, we can introduce function spaces of divergence-free vector fields: 
\begin{align*}
\Hs^{s,q}(\T^d)\stackrel{{\rm def}}{=}\p(H^{s,q}(\T^d;\R^d)) 
=\{f\in H^{s,q}(\T^d;\R^d)\,:\, \nabla \cdot f=0\},
\end{align*}
endowed with the natural norm.
For $s,q$ as above and $p\in (1,\infty)$, similar definitions hold for $\Ls^{q}(\T^d)$, $\Bs^{s}_{q,p}(\T^d)$ and $\ell^2$-valued function spaces. Finally, as common practice, we write $\Hs^{s}(\T^d)$ instead of $\Hs^{s,2}(\T^d)$.

\subsubsection{It\^o formulation of the stochastic 3D NSEs and energy balance}
\label{sss:ito_stratonovich_operators}
We begin by formally applying the Helmholtz projection $\p$ to the system \eqref{eq:SNS_introduction} and obtain
\begin{equation}
\label{eq:SNS_introduction_p}
\partial_t u 
= \Delta u  -\p\big[\nabla \cdot (u\otimes u)] +
\sum_{n\geq 1} \p\big[(\sigma_n\cdot \nabla) u + \cc_n\cdot u \big] \circ \dot{W}^n  \ \text{ on }\T^3,
\end{equation}
together with the initial condition $u(0)=u_0$.
In the above, we have also rewritten the transport term $(u\cdot \nabla)u$ in the conservative form $\nabla \cdot (u\otimes u)$, which formally follows from $\nabla \cdot u =0$ and better accommodates the weak PDE setting. Note that the SPDE \eqref{eq:SNS_introduction_p} (formally) preserves the divergence-free property of the initial data. Therefore, the condition $\nabla \cdot u=0$ is redundant for \eqref{eq:SNS_introduction_p} in case $\nabla \cdot u_0=0$ in $\D'(\T^d)$, as we will always assume in this section.

Next, we discuss the Stratonovich-to-It\^o transformation of \eqref{eq:SNS_introduction_p}. Set  
\begin{align*}
\L_{n} u \stackrel{{\rm def}}{=}\p[(\sigma_n\cdot\nabla )u + \cc_n \cdot u].
\end{align*}
Then, using Assumption \ref{ass:NSE_regularity_coefficients}\eqref{it:NSE_regularity_coefficients3}, formally, the solution $u$ to the SPDE \eqref{eq:SNS_introduction_p} verifies:
\begin{equation}
\label{eq:def_A_ito_stratonovich_sns0}
\sum_{n\geq 1} \p\big[(\sigma_n\cdot \nabla) u + \cc_n\cdot u \big] \circ \dot{W}^n
=\A u +
\sum_{n\geq 1} \p\big[(\sigma_n\cdot \nabla) u + \cc_n\cdot u \big] \, \dot{W}^n
\end{equation}
where
\begin{equation}
\label{eq:def_A_ito_stratonovich_sns}
\A u=\frac{1}{2}\sum_{n\geq 1}\L_n^2 u =\frac{1}{2} \sum_{n\geq 1} 
\big(\p[\nabla \cdot (\L_n u\otimes \sigma_n)]+ \p[\cc_n \L_n u]\big).
\end{equation}
Again, here we reformulate transport terms in the conservative form to accommodate the weak PDE setting. 
Thus, the It\^o formulation of \eqref{eq:SNS_introduction_p} is given by 
\begin{equation}
\label{eq:SNS_Ito}
\left\{
\begin{aligned}
\partial_t u 
&= \Delta u+\A u  -\p\big[\nabla \cdot (u\otimes u)] +
\sum_{n\geq 1} \p\big[(\sigma_n\cdot \nabla) u + \cc_n\cdot u \big]  \dot{W}^n   &\text{on }&\T^3,\\
 u(0)&=u_0 &\text{on }&\T^3.
\end{aligned}
\right.
\end{equation}
Clearly, the condition $\nabla \cdot u_0=0$ is preserved along the flow induced by \eqref{eq:SNS_Ito}. Therefore, in the rest of this section, we also understand the SPDE \eqref{eq:SNS_introduction} as the It\^o SPDE \eqref{eq:SNS_Ito} with divergence-free initial data. A rigorous definition of (weak) solution to \eqref{eq:SNS_Ito} is given in Definition \ref{def:Leray_Hopf_strong}\eqref{it:Leray_Hopf_strong1}-\eqref{it:Leray_Hopf_strong2} below. 

\smallskip

As we have seen in Subsection \ref{ss:main_results}, to apply our main results, we need strong weak-strong uniqueness. In the context of stochastic 3D NSEs, it is known that (a suitable) strong weak-strong uniqueness holds for weak solutions satisfying the strong energy inequality, see, e.g., \cite[Proposition 12.1]{LePi} and Definition \ref{def:Leray_Hopf_strong}\eqref{it:Leray_Hopf_strong1}-\eqref{it:Leray_Hopf_strong3} below. To this end, we discuss the formal energy balance for the SPDE \eqref{eq:SNS_Ito}. Let $\L_n^\top$ be the formal adjoint of $\L_n$ on $\Ls^2(\T^3)$, i.e.,
$$
\L_n^\top u\stackrel{{\rm def}}{=} - \p [((\sigma_n \cdot\nabla)- \cc_n^\top)u].
$$
By a formal integration by parts argument and using Assumption \ref{ass:NSE_regularity_coefficients}\eqref{it:NSE_regularity_coefficients3}, the following energy balance holds for sufficiently smooth solutions $u$ to \eqref{eq:SNS_Ito}:
\begin{align}
\label{eq:formal_energy_inequality}
&\frac{1}{2}\|u(t)\|_{L^2}^2
+ \int_{t_0}^t \int_{\T^3}|\nabla u|^2\,\dd x \,\dd r \\
\nonumber
&=\frac{1}{2}\|u(t_0)\|_{L^2}^2 + \sum_{n\geq 1}\int_{t_0}^t\int_{\T^3}  (\cc_n\cdot u)\cdot u\,\dd x\,\dd W^n
+ \int_{t_0}^t \int_{\T^3}  \L_n u \cdot S_n u  \,\dd x \,\dd r
\end{align}
where $S_n=(\cc_n+\cc_n^\top)/2$ and we used that, at least for smooth $u$,
\begin{align*}
\int_{\T^3} \A u\cdot u \,\dd x  + \frac{1}{2}\|(\L_n u)_{n}\|_{\calL_2(\ell^2,\Ls^2)}^2
&=
\frac{1}{2}\sum_{n\geq 1}\Big( \int_{\T^3} \big(\L_n u\cdot  \L_n^\top u + |\L_n u |^2\big)\,\dd x \Big)\\
&= 
\int_{\T^3}  \L_n u \cdot S_n u  \,\dd x.
\end{align*}
Note that the right-hand side of \eqref{eq:formal_energy_inequality} contains only lower-order terms compared to the one on the left-hand side. Therefore, as in the deterministic case, this can be exploited to construct solutions satisfying the (strong) energy \emph{inequality} via a compactness argument; see Proposition \ref{prop:existence_strong_Leray} and the comments below it for the relevant literature. The corresponding strong weak-strong uniqueness property is investigated in Proposition \ref{prop:weak_strong_uniqueness}.

\subsection{Bounds on singular times of annealed strong Leray-Hopf solutions}
\label{ss:annealed_strong_energy}
In this subsection, we state our main results on the singular times of solutions to the stochastic 3D NSEs \eqref{eq:SNS_introduction}. To this end, similar to the deterministic setting (see e.g., \cite[Theorem 13.5]{LePi}), we introduce a class of solutions to \eqref{eq:SNS_introduction} which ensures energy dissipation at any two given instances of time $0\leq t_0<t_1<\infty$. In the deterministic setting, those are usually referred to as \emph{strong Leray-Hopf solutions}. Next, we formulate one possible stochastic counterpart of such solutions, which seems to be \emph{new} in the literature. 

Below, $\Cw(I;H)$ denotes the set of all weakly continuous maps from $I\subseteq \R$ to a Hilbert space $H$.

\begin{definition}[Annealed strong stochastic Leray-Hopf solutions]
\label{def:Leray_Hopf_strong}
Fix 
\begin{equation*}
u_0\in L^0_{\F_0}(\O;\Ls^2(\T^3)), \quad \text{ and let }\quad u:[0,\infty)\times \O\to \Hs^1(\T^3) 
\end{equation*}
be a progressively measurable process. 
We say that $u$ is a \emph{annealed strong stochastic Leray-Hopf solution} if there exists a sequence $(\O_n)_n\subseteq \F_0$ such that $\O_n\uparrow \O$ for which the following hold:  
\begin{enumerate}[{\rm(1)}]
\item\label{it:Leray_Hopf_strong1} {\rm (Regularity)} 
a.s.\ $u\in \Cw([0,\infty); \Ls^2(\T^3))$, and for all $n\geq 1$ and $T<\infty$,
$$
\E\Big[\one_{\O_n}\sup_{t<T}\|u(t)\|_{L^2}^2\Big]+\E\Big[\one_{\O_n} \int_0^{T}\|\nabla u(t)\|_{L^2}^2\,\dd t\Big] <\infty.
$$
\item\label{it:Leray_Hopf_strong2} {\rm (Weak formulation)} 
a.s.\ for all $\varphi\in \Hs^1(\T^3)$ and $t>0$,
\begin{align*}
\int_{\T^3} u(t)\cdot \varphi \,\dd x 
&+\int_0^t \int_{\T^3}\nabla u: \nabla \varphi\,\dd x\,\dd s =
\int_{\T^3} u_0\cdot \varphi \,\dd x \\
& + \int_0^t \int_{\T^3} (u\otimes u): \nabla \varphi \,\dd x \, \dd s 
+\frac{1}{2}\int_{0}^t \int_{\T^3} \L_n u \cdot \L_n^\top \varphi \,\dd x\,\dd s\\
&\quad+ \sum_{n\geq 1}\int_0^t\int_{\T^3} \L_n u\cdot \varphi\, \dd x \,\dd W^n.
\end{align*}
\item\label{it:Leray_Hopf_strong3} {\rm (Annealed strong energy inequality)}
For a.a.\ $t_0\geq 0$ and all \emph{bounded} stopping times $\tau_0:\O\to [t_0,\infty)$ and $n\geq 1$, it holds that 
\begin{align*}
\frac{1}{2}\,\E\big[\one_{\O_n}\|u( \tau_0)\|_{L^2}^2]
&+ \E \Big[\one_{\O_n} \int_{t_0}^{\tau_0} \int_{\T^3}|\nabla u|^2\,\dd x \,\dd r\Big] \\
& \leq \frac{1}{2}\,\E\big[\one_{\O_n} \|u(t_0)\|_{L^2}^2\big]
+ \E\Big[\one_{\O_n} \int_{t_0}^{ \tau_0} \int_{\T^3}  \L_n u \cdot S_n u  \,\dd x \,\dd r\Big].
\end{align*}
\end{enumerate}
\end{definition}

Due to \eqref{it:Leray_Hopf_strong1} and Assumption \ref{ass:NSE_regularity_coefficients}, all the terms in \eqref{it:Leray_Hopf_strong2} and \eqref{it:Leray_Hopf_strong3} are well-defined. 
The presence of the localizing sequence $(\O_n)_{n\geq 1}$ is to accommodate initial data $u_0$ with no $\O$-integrability. 
Clearly, if $u_0\in L^2(\O;\Ls^2(\T^3))$, then from Fatou's lemma, the annealed energy inequality in \eqref{it:Leray_Hopf_strong3} holds with $\O_n$ replaced by $\O$.

The adjective `strong' applied to the energy inequality in \eqref{it:Leray_Hopf_strong3} is taken from the deterministic setting (see e.g., \cite[Proposition 12.1]{LePi}) as the corresponding bound holds for \emph{a.a.} $t_0\geq 0$. 

In case a process $u$ satisfies \eqref{it:Leray_Hopf_strong1}-\eqref{it:Leray_Hopf_strong2} and satisfies the energy inequality in \eqref{it:Leray_Hopf_strong3} for $t_0=0$, we say that $u$ is a \emph{stochastic annealed Leray-Hopf solution} to the stochastic 3D NSEs \eqref{eq:SNS_introduction} (therefore, omitting the adjective strong). 

\smallskip

The formulation of the strong annealed energy inequality in
Definition \ref{def:Leray_Hopf_strong}\eqref{it:Leray_Hopf_strong3} appears to be the weakest condition that ensures the \emph{strong weak-strong uniqueness} property for the stochastic 3D NSEs \eqref{eq:SNS_introduction}, see Proposition \ref{prop:weak_strong_uniqueness} below. In particular, we do not know if the latter strong weak-strong uniqueness property can be proven with only deterministic times in
Definition \ref{def:Leray_Hopf_strong}\eqref{it:Leray_Hopf_strong3}.

A pathwise version of the strong energy inequality in \eqref{it:Leray_Hopf_strong3} can be formulated as follows. A progressively measurable process $u$ satisfying \eqref{it:Leray_Hopf_strong1} satisfies the \emph{pathwise strong energy inequality} if for a.a.\ $t_0\geq 0$, then a.s.\ for all $t\geq t_0$,
\begin{align}
\label{eq:strong_energy_inequality}
\frac{1}{2}\,\|u( t)\|_{L^2}^2
&+ \int_{t_0}^{t} \int_{\T^3}|\nabla u|^2\,\dd x \,\dd r
 \leq \frac{1}{2}\, \|u(t_0)\|_{L^2}^2\\
\nonumber
&+ \sum_{n\geq 1}\int_{t_0}^t  \int_{\T^3}(\cc_n\cdot u)\cdot u\,\dd x\,\dd W^n
+  \int_{t_0}^{ t} \int_{\T^3} \L_n u \cdot S_n u  \,\dd x \,\dd r.
\end{align}
Note that the exceptional set on which the equality \eqref{eq:strong_energy_inequality} holds for all $t\geq t_0$ might depend on $t_0$. Clearly, the pathwise strong energy inequality \eqref{eq:strong_energy_inequality} and the bound in Definition \ref{def:Leray_Hopf_strong}\eqref{it:Leray_Hopf_strong1} imply the annealed one in Definition \ref{def:Leray_Hopf_strong}\eqref{it:Leray_Hopf_strong3}. 

\smallskip

The following result ensures the existence of annealed strong Leray-Hopf solutions satisfying the pathwise energy inequality \eqref{eq:strong_energy_inequality}.

\begin{proposition}[Existence of strong stochastic Leray-Hopf solutions]
\label{prop:existence_strong_Leray}
Let $\Lambda$ be a probability measure on $\Ls^2(\T^3)$.
Then there exist a probability space $(\O,\F, \P)$ with a filtration $(\F_t)_{t\geq 0}$ satisfying the usual conditions, a cylindrical Brownian motion $W$ in $\ell^2$, an $\F_0$-measurable random variable $u_0:\O\to \Ls^2(\T^3)$ with law $\Lambda$, and an annealed strong Leray-Hopf solution $u$ with initial data $u_0$. Moreover, $u$ also satisfies the pathwise strong energy inequality \eqref{eq:strong_energy_inequality}.
\end{proposition}

The above result might be known to experts, and except for the strong energy inequality, it is well-known; see e.g.,  \cite{FG95_PTRF,FL32_book,MR05}. Recent results including the energy inequality \eqref{eq:strong_energy_inequality} with $t_0=0$ are provided in \cite[Proposition 5.1]{gess2023landau} and \cite[Theorem 3.7]{GL26_strong}. In Subsection \ref{ss:existence_strong_Leray_proof}, we provide some details on how to deduce also the \emph{strong} energy inequality  \eqref{eq:strong_energy_inequality} following the original approach by Leray \cite{Leray}.

\smallskip

Next, we define the set of singular times for the stochastic 3D NSEs \eqref{eq:SNS_introduction} following Definitions \ref{def:singular_regular} and \ref{def:singular_regular2}. To motivate the regularity class chosen in the following, let us recall that, under Assumption \ref{ass:NSE_regularity_coefficients}, \cite[Theorem 2.7]{AV21_NS} ensures that strong solutions belong (locally in time) to the space
$$
C^{1/2-,(1+\g)-}_{\loc}((0,T)\times \T^3;\R^3),
$$
where $\g$ is as in Assumption \ref{ass:NSE_regularity_coefficients} (see Subsection \ref{ss:notation} for the notation).
Due to Lemma \ref{lem:independence} and the proofs in Subsection \ref{ss:annealed_LH}, other (equivalent) choices are possible (see Subsection \ref{sss:independence_set} and the proofs of Theorem \ref{t:singular_times_SNS} and \ref{t:singular_times_SNS_2} below).

\begin{definition}[Regular and singular times -- Stochastic 3D NSEs]
\label{def:singular_times_NS}
Let $\g>0$ be as in Assumption \ref{ass:NSE_regularity_coefficients}, and fix $\varepsilon\in (0,1)$.
Let $u:[0,\infty)\times \O\to H^1(\T^3;\R^3)$ be a progressively measurable process.
\begin{itemize}
\item {\normalfont{($\varepsilon$-Regular times)}} We say that $t_0\in (0,\infty)$ belongs to the set of  $\varepsilon$-regular times $ \re^\varepsilon$ of $u$ if there exist $t<t_0$ and a stopping time $\tau:\O\to [t,\infty]$ such that  
$\P(\tau>t_0)>1-\varepsilon$ and
\begin{align*}
u\in C^{1/2-,(1+\g)-}_{\loc}((t,\tau)\times \T^3;\R^3)& \ \text{ a.s.\ } 
\end{align*} 
\item {\normalfont{($\varepsilon$-Singular times)}} We say that $t_0\in [0,\infty)$ belongs to the set of $\varepsilon$-singular times $\re^\varepsilon$ if it is not $\varepsilon$-regular, i.e., 
$$
\si^\varepsilon=[0,\infty)\setminus \re^\varepsilon.
$$
\item {\normalfont{(Regular and singular times)}} 
The set of regular and singular times $\re$ and $\si$ are defined as 
$$
\re= \textstyle{\bigcap_{\varepsilon\in (0,1)}}\re^\varepsilon \quad \text{ and }\quad
\si=[0,\infty)\setminus \re.
$$
\end{itemize}
\end{definition}

We are now in a position to state the main results of this section concerning singular times of annealed strong stochastic Leray-Hopf solutions. The reader is referred to Subsection \ref{ss:fractal} for the notions of fractal dimensions and measures of Hausdorff and Minkowski type.

\begin{theorem}[$1/2$-bound on singular times -- 3D stochastic NSEs]
\label{t:singular_times_SNS}
Suppose that Assumption \ref{ass:NSE_regularity_coefficients} holds. Let $u$ be an annealed strong Leray-Hopf solution to the stochastic {\normalfont{3D NSEs}} \eqref{eq:SNS_introduction} in the sense of Definition \ref{def:Leray_Hopf_strong}. Then, for all $\varepsilon\in (0,1)$, 
\begin{align}
\label{eq:singular_NSE1}
\dim_{\M}(\si^\varepsilon)&\leq 1/2, \qquad\text{ and }\qquad 
\M^{1/2}(\si^\varepsilon)=0,  \\ 
\label{eq:singular_NSE2}
\dim_{\H}(\si)&\leq 1/2, \qquad\text{ and }\qquad \ 
\H^{1/2}(\si)=0.  
\end{align}
\end{theorem}

The above is an extension of the classical deterministic bounds of Leray \cite{Leray} and Scheffer \cite{S76_original}, as well as the subsequent refinements by Robinson and Sadowski \cite{RS1_07} and Kukavica \cite{K09_singular_times}, on the Hausdorff and Minkowski (or box-counting) dimension and measure of singular times.

It is worth noticing that \eqref{eq:singular_NSE1} implies \eqref{eq:singular_NSE2}. Indeed, 
by \eqref{eq:HF_comparison} it holds that
$$
\H^{1/2}(\si^\varepsilon)\leq \M^{1/2}(\si^\varepsilon)=0.
$$ 
Hence, $\H^{1/2}(\si)=0$ by $\sigma$-subadditivity of $\H^{1/2}$. 
As commented in Subsection \ref{ss:fractal}, the latter does not hold for the Minkowski content $\M^{s}$ for all $s\in (0,1)$. 

\smallskip

It is an open problem whether \eqref{eq:singular_NSE1} holds with $\si^\varepsilon$ replaced by $\si$ in the setting of Theorem \ref{t:singular_times_SNS} (i.e., non-trivial multiplicative noise including transport terms). This is due to the more complicated structure of the singular times in the stochastic setting as defined in Definition \ref{def:singular_times_NS} (indeed, note that, in the deterministic case, $\si^\varepsilon$ coincides with $\si$ for each $\varepsilon\in(0,1)$). 
For similar reasons, by invoking global existence for small data (see e.g., \cite[Theorem 4.3]{A24_global_small} or \cite[Theorem 2.11]{AV21_NS}), we expect that the set $\si^\varepsilon$ is compact for all $\varepsilon\in (0,1)$, while this is not in general the case for $\si=\bigcup_{\varepsilon\in (0,1)}\si^\varepsilon$. 
For brevity, we do not pursue this here. 

\smallskip

In the following result, we quantify the improvement of fractal dimensions of the sets of singular times in the presence of additional bounds.

\begin{theorem}[Supercritical Serrin conditions and singular times -- 3D stochastic NSEs]
\label{t:singular_times_SNS_2}
Let $u$ be an annealed strong Leray-Hopf solution to the stochastic {\normalfont{3D NSEs}} \eqref{eq:SNS_introduction} in the sense of Definition \ref{def:Leray_Hopf_strong}. Suppose that there exist parameters 
$p_0,r_0\in [2,\infty)$, $q_0 \in (3,\infty)$ and $\g_0\in [ 0,3/q_0)$ satisfying  
\begin{equation}
\label{eq:conditions_Besov_negative_smoothness}
\g_0+\frac{3}{q_0} <1 
 \qquad \text{ and } \qquad \frac{2}{p_0}+\g_0+\frac{3}{q_0}>1, 
\end{equation}
and a sequence $(\O_n)_n \subseteq \F_0$ such that $\O_n\uparrow \O$, for which the following bound holds 
\begin{equation}
\label{eq:bound_Besov_SNS}
\E\Big[\one_{\O_n}\int_0^T \| u\|_{B^{-\g_0}_{q_0,r_0}(\T^3;\R^3)}^{p_0}\,\dd t \Big]<\infty  \ \text{ for all }\ n\geq 1 ,\  T<\infty. 
\end{equation}
Finally, suppose that Assumption \ref{ass:NSE_regularity_coefficients} holds for $\g>\g_0$.  
Then, letting $\delta_0=\frac{p_0}{2}(\frac{2}{p_0}+\g_0+\frac{3}{q_0}-1)$, it holds that 
\begin{align}
\label{eq:singular_NSE1_cond}
\dim_{\M}(\si^\varepsilon)&\leq \delta_0, \qquad\text{ and }\qquad 
\M^{\delta_0}(\si^\varepsilon)=0,  \\ 
\label{eq:singular_NSE2_cond}
\dim_{\H}(\si)&\leq \delta_0, \qquad\text{ and }\qquad \ 
\H^{\delta_0}(\si)=0.  
\end{align}
In particular, if there exist $p_1\in [2,\infty)$ and $q_1\in (3,\infty)$ satisfying the following \emph{supercritical Serrin condition}:
\begin{equation}
\label{eq:supercritical_Serrin_conditions}
\frac{2}{p_1}+\frac{3}{q_1}>1 \ \ \quad \text{ and } \ \ \quad
\E\Big[ \one_{\O_n}\int_{0}^T \|u\|_{L^{q_1}(\T^3;\R^3)}^{p_1}\,\dd t\Big] <\infty  
\end{equation}
for all $n\geq 1$ and $T<\infty$, and with $(\O_n)_n\subseteq \F_0$ such that $\O_n\uparrow \O$, then the estimates \eqref{eq:singular_NSE1_cond} and \eqref{eq:singular_NSE2_cond} hold with $\g_0=0$, $p_0=p_1$ and $q_0=q_1$.
\end{theorem}

The conditions \eqref{eq:conditions_Besov_negative_smoothness} imply that $B^{-\g_0}_{q_0,r_0}$ is subcritical (or, in other words, the Sobolev index $>-1$) and a supercritical Serrin-type condition (i.e., the space-time Sobolev index of $L^{p_0}_t(B^{-\g_0}_{q_0,r_0})$ is less than $-1$), respectively. Note that, in the case $\frac{2}{p_0}+\g_0+\frac{3}{q_0}\leq 1$, then one expects \emph{global} regularity of the solutions due to Serrin-type regularity criteria; see e.g., \cite[Theorem 11.2]{LePi} and \cite[Theorem 2.9]{AV21_NS} for the stochastic setting.
Finally, from the embedding,
\begin{equation}
\label{eq:embedding_Lp_B0}
L^{q_1}(\T^3)\subseteq B^{0}_{q_1,q_1}(\T^3),
\end{equation}
as $q_1\geq 2$. Hence,
the last assertion of Theorem \ref{t:singular_times_SNS_2} is a consequence of \eqref{eq:singular_NSE1_cond}-\eqref{eq:singular_NSE2_cond}.

As below Theorem \ref{t:NSE_intro}, letting $\mathcal{S}_0=-\frac{2}{p_0}-\g_0-\frac{d}{q_0}$ the space-time parabolic Sobolev index of the space $L^{p_0}_t(B_{q_0,r_0}^{-\g_0})$ in which the bound \eqref{eq:bound_Besov_SNS} holds, then dimensional bound $\delta_0$ in \eqref{eq:singular_NSE1_cond}-\eqref{eq:singular_NSE2_cond} has the following interpretations:
$$
\delta_0=
1-\frac{p_0}{2} \Big(\underbrace{-\g_0-\frac{d}{q_0}}_{\text{Scaling of }\xx_0}- \underbrace{(-1)}_{\text{Criticality}}\Big)=
\frac{p_0}{2}\Big(\underbrace{-1}_{\text{Criticality}} -\big( \underbrace{-\frac{2}{p_0}-\g_0-\frac{d}{q_0}}_{\text{Scaling of }L^{p_0}(\xx_0)}\big) \Big).$$ 
where $\xx_0=B_{q_0,r_0}^{-\g_0}(\T^3;\R^3)$.
In particular, as in Subsection \ref{ss:NSEs_time_intro} at a fixed space-time Sobolev index, the bound on fractal dimensions improves if $p_0$ decreases. 

\smallskip

Interestingly, negative smoothness is allowed in the condition \eqref{eq:bound_Besov_SNS}. This contrasts with existing deterministic results, which typically require non-negative smoothness (see e.g., \cite{G86,L19_PhysD}). Unfortunately, it is far from being obvious how to check the condition \eqref{eq:bound_Besov_SNS} in the relevant case $\delta_0<\frac{1}{2}$, i.e., 
$
\frac{1}{p_0}+ \g_0+\frac{d}{q_0}< 1,
$
where the claims \eqref{eq:singular_NSE1_cond} and \eqref{eq:singular_NSE2_cond} improve the one in Theorem \ref{t:singular_times_SNS}.

\smallskip

We conclude this subsection with some comments on the case of NSEs on $\R^3$.

\begin{remark}[The whole space case]
\label{r:equilvalent_singular_times_SNS_3D}
One can check that using \cite[Subsection 8.4]{AV25_survey} and the arguments of Theorem \ref{t:singular_times_SNS}, the bounds \eqref{eq:singular_NSE1}-\eqref{eq:singular_NSE2} extend to the case where $\T^3$ is replaced by $\R^3$, subject to possible conditions on $(\sigma_n)_n$ for $|x|\to \infty$, see \cite[Remark 8.25]{AV25_survey}. Note that in the $\R^3$-case, regular times as in Definition \ref{def:singular_times_NS} are modelled over the space $C^{1/2-,(1+\g)-}_{\loc}((0,T)\times \R^3;\R^3)$, with no conditions at infinity. For this reason, extending Theorem \ref{t:singular_times_SNS_2} seems more challenging, as Proposition \ref{prop:singular_times_SPDEs_weaker_weak_strong} (used in the periodic setting) cannot be applied to the $\R^3$-case.

Finally, it is worth mentioning that in the absence of noise, the extension of Theorem \ref{t:singular_times_SNS_2} to $\R^3$ is possible, as a wider class of time weights $\a$ can be used (see \cite{PW18}). Moreover, Proposition \ref{prop:singular_times_SPDEs_weaker_weak_strong} is not needed in the deterministic case.
\end{remark}

\subsection{Weak-strong uniqueness and proofs of Theorems \ref{t:singular_times_SNS} and \ref{t:singular_times_SNS_2}}
\label{ss:annealed_LH}
In this subsection, we prove Theorems \ref{t:singular_times_SNS} and \ref{t:singular_times_SNS_2} by applying the abstract theory developed in Section \ref{s:singular_times_SPDEs}. In both cases, we need to verify the strong weak-strong condition of Assumption \ref{ass:abstract2} in appropriate settings. 
To this end, we introduce the relevant class of strong solutions to the stochastic 3D NSEs \eqref{eq:SNS_introduction} that we employ for weak-strong uniqueness. 

Below, $\mathcal{W}_{\ell^2}$ denotes the $\ell^2$-cylindrical Brownian motion associated to the sequence $(W^n)_n$ of standard independent Brownian motions via the formula 
\begin{equation}
\label{eq:def_Well2}
\textstyle
\mathcal{W}_{\ell^2}(f)\stackrel{{\rm def}}{=}\sum_{n\geq 1} \int_{\R_+} f_n(t)\,\dd W^n_t \ \ \text{ for } \ \ f=(f_n)_n\in L^2(\R_+;\ell^2).
\end{equation} 

\begin{definition}[$L^3$-solutions to the stochastic 3D NSEs]
\label{def:Lq_solution_SNS}
Let $\tau:\O\to [0,\infty]$ and $v: [0,\infty)\times\O\to \Hs^1(\T^3)$ be a stopping time and a progressively measurable process, respectively. Assume that $u_0\in L^0_{\F_0}(\O;\Ls^3(\T^3))$ and 
\begin{equation}
\label{eq:assumption_regularity_strong_SNS}
v\in C([0,\tau);\Ls^3(\T^3))\cap L^2_{\loc}([0,\tau);\Hs^1(\T^3)) \text{ a.s. }
\end{equation}
We say that $(v,\tau)$ is an $L^3$-solution to the {\normalfont{stochastic 3D NSEs}} \eqref{eq:SNS_introduction} if a.s.\ for all $t\in [0,\tau)$ the following identity holds:
\begin{align}
\label{eq:identity_SNS_strong}
v(t)
&=u_0+\int_{0}^t\big(
\Delta v(s) -\p[\nabla\cdot  (v(s)\otimes v(s)) ]+ \A v(s) \big) \,\dd s\\
\nonumber
&+ \int_0^t \one_{[0,\tau)}\big( \p\big[ (\sigma_n\cdot \nabla) v(s) + \cc_n\cdot v (s)\big]\big)_n \, \dd\mathcal{W}_{\ell^2}
\end{align}
in $\Hs^{-1}(\T^3)$, where $\A$ is as in \eqref{eq:def_A_ito_stratonovich_sns}.
\end{definition}

Since $H^1(\T^3)\embed L^6(\T^3)$ by Sobolev embeddings, it holds that 
\begin{equation}
\label{eq:interpolation_Lq_estimates}
\|v'\otimes v' \|_{L^2}\lesssim \|v'\|_{H^1}\|v'\|_{L^3}  \ \ \text{ for all }\ v'\in H^1(\T^3;\R^3).
\end{equation} 
In particular, \eqref{eq:assumption_regularity_strong_SNS} implies that $\nabla\cdot (v\otimes v)\in L^2_{\loc}([0,\tau);\Hs^{-1}(\T^3))$ a.s., and therefore the deterministic integral in \eqref{eq:identity_SNS_strong} is well-defined as $\Hs^{-1}(\T^3)$-valued Bochner integral. Similarly, from Assumption \ref{ass:NSE_regularity_coefficients} and \eqref{eq:assumption_regularity_strong_SNS}, the stochastic integral in \eqref{eq:identity_SNS_strong} is well-defined as an $\Ls^2(\T^3)$-valued It\^o's integral. 
Moreover, from \eqref{eq:interpolation_Lq_estimates} and the It\^o's formula (see e.g., \cite[Theorem 4.2.5]{LR15}), one can check that the energy \emph{equality} 
\eqref{eq:formal_energy_inequality} holds for $t_0=0$ and $t\leq \tau$.

\smallskip

The following weak-strong uniqueness result is the key ingredient that allows us to apply the result in Subsections \ref{ss:main_results} and \ref{ss:extension_strategy} and is of independent interest. Recall that annealed stochastic Leray-Hopf solutions to the stochastic 3D NSEs \eqref{eq:SNS_introduction} are defined in the text below Definition \ref{def:Leray_Hopf_strong}.

\begin{proposition}[Weak-strong uniqueness property for annealed Leray-Hopf solutions]
\label{prop:weak_strong_uniqueness}
Assume that $u_0\in L^0_{\F_0}(\O;\Ls^3(\T^3))$. Let $u$ and $v$ be an annealed Leray-Hopf solution and an $L^3$-solution to \eqref{eq:SNS_introduction}, respectively. 
Then $u=v$ a.e.\ on $[0,\tau)\times \O$. 
\end{proposition}

The above is an extension of the endpoint Serrin's weak-strong uniqueness result; see e.g., \cite[Theorem 12.4 and Proposition 12.2]{LePi}. Weak-strong uniqueness under the general Serrin's condition $v\in L^{p_0}_{\loc}([0,\tau);L^{q_0}(\T^3;\R^3))$ for $q_0\in (3,\infty)$ and $\frac{2}{p_0}+\frac{3}{q_0}=1$ seems more complicated. However, Proposition \ref{prop:weak_strong_uniqueness} suffices for our purposes. 

Weak-strong uniqueness for stochastic NSEs has recently attracted some attention; see, e.g., 
\cite{CK22_waek_strong,gess2023landau} and \cite[Theorem 4.9]{GL26_strong}. However, Proposition \ref{prop:weak_strong_uniqueness} seems to be the first result on weak-strong uniqueness for stochastic Leray-Hopf solutions and strong solutions with critical regularity (see also \cite[Theorem 2.2 and Remark 2.4]{GK25_conditional} for a uniqueness result where both solutions may belong to a supercritical space).

\smallskip

In the remaining part of this section, employing Proposition \ref{prop:weak_strong_uniqueness}, we show that Theorems \ref{t:singular_times_SNS} and \ref{t:singular_times_SNS_2} are consequences of the results in Subsections \ref{ss:main_results} and \ref{ss:extension_strategy}. 
The proof of Proposition \ref{prop:weak_strong_uniqueness} is postponed to Subsection \ref{sss:weak_strong_proof_SNS}.

\subsubsection{Proof of Theorem \ref{t:singular_times_SNS}}
\label{sss:proof_singular_times_SNS_1}
Here, we apply Theorem \ref{t:singular_times_SPDEs} in the weak PDE setting, i.e., $\set=\setfull$ where 
\begin{equation}
\label{eq:weak_setting_NSEs}
X_0=\Hs^{-1,q}(\T^3),  \ \ X_1=\Hs^{1,q}(\T^3), \ \  q\in [2,\infty), \ \  p\in (2,\infty), \ \ \a\in [0,\tfrac{p}{2}-1),
\end{equation}
and for $v\in X_1$,
\begin{equation}
\label{eq:weak_setting_NSEs_2}
A v = -\Delta v - \A v, \ \ B v = \big(\p[(\sigma_n\cdot\nabla )v + \cc_n \cdot v]\big), \ \ F(v)=-\p[\nabla\cdot (v\otimes v)], 
\end{equation}
and $G(v)=0$.
In the above, we used the notation introduced in Subsection \ref{sss:function_spaces_divfree} and \eqref{eq:def_A_ito_stratonovich_sns}.
From Assumption \ref{ass:NSE_regularity_coefficients}, all the operators in \eqref{eq:weak_setting_NSEs_2} are well-defined.  
Further assumptions on the parameters $(q,p,\a)$ are given below when needed.
Of course, the standard case $q=p=2$ is excluded as well-posedness does not hold in this setting for the 3D NSEs.

The reason for this choice is that the weak setting allows for the least regularity for the noise coefficients $(\sigma_n)_n$ and $(\cc_n)_n$, cf.\ \cite[Section 3]{AV21_NS}. This is of central importance to accommodate rough Kraichnan and Lie transport noises (see also Remark \ref{rem:necessity_Lp_setting}). Moreover, it will turn out that different choices of the parameters $(q,p,\a)$ lead to the same bound on the fractal dimensions of singular times for the stochastic 3D NSEs \eqref{eq:SNS_introduction}. In particular, the $1/2$-bound is, to some extent, \emph{universal}.

\smallskip

We are now in a position to prove Theorem \ref{t:singular_times_SNS}.

\begin{proof}[Proof of Theorem \ref{t:singular_times_SNS}]
We begin by recalling that, by standard interpolation results (see e.g., \cite[Chapter 6]{BeLo}), for all $\theta\in (0,1)$ and $p\in (1,\infty)$,
\begin{equation}
\label{eq:interpolation_X0_weak_spaces}
[X_0,X_1]_\theta= \Hs^{-1+2\theta,q}(\T^3) \qquad \text{ and }\qquad
(X_0,X_1)_{\theta,p}= \Bs^{-1+2\theta}_{q,p}(\T^3).
\end{equation}
We now divide this proof into several steps.
\smallskip

\emph{Step 1: (Maximal $L^p$-regularity).\ The couple $(A,B)$ in \eqref{eq:weak_setting_NSEs_2} has the stochastic maximal $L^p$-regularity in the $\set$-setting.} 
Let $(A_0, B_0)$ be the operators in \eqref{eq:weak_setting_NSEs_2} with $\mu_n=0$. Now, it follows from \cite[Theorem 3.2]{AV21_NS} that $(A_0,B_0)$ has the stochastic maximal $L^p_\a$-regularity property. To conclude, we now apply a perturbation argument by using the results in \cite[Section 3]{AV_torus}.
Note that $(A,B)=(A_0,B_0)+(A-A_0,B-B_0)$ and $(A-A_0,B-B_0)$ contains only lower-order differential operators due to Assumption \ref{ass:NSE_regularity_coefficients}. For instance, the operator $v\mapsto \sum_{n\geq 1}\p\big[\cc_n \cdot\p[(\sigma_n\cdot\nabla v)]\big]$ is contained in $A-A_0$, and for all $\delta\in (0,\gamma)$ and $v\in X_1$ satisfies,
\begin{align*}
\|\textstyle{\sum_{n\geq 1}}\,\p\big[\mu_n \cdot\p[(\sigma_n\cdot \nabla )v]\big]\|_{\Hs^{-1,q}}
&\lesssim
\|\textstyle{\sum_{n\geq 1}}\,\mu_n \cdot \p[(\sigma_n\cdot \nabla )v]\|_{H^{-\delta,q}}\lesssim
\| v\|_{H^{1-\delta,q}},
\end{align*}
where the last inequality follows from pointwise multiplication results in Sobolev spaces of negative smoothness \cite[Proposition 4.1(4)]{AV_torus}. 
Since $\|v\|_{H^{1-\delta,q}}\lesssim \|v\|_{X_0}^{\delta/2}\|v\|_{X_1}^{1-\delta/2}$ for $\delta\in (0,1)$, the claim of Step 1 is a consequence of \cite[Theorem 3.2]{AV_torus}.

\smallskip

\emph{Step 2: (Nonlinear estimates).\ For all $v,v'\in H^{\mu,q}(\T^3;\R^3)$, it holds that}
\begin{equation}
\label{eq:nonlinear_estimate_weak_setting_Step1}
\|\nabla \cdot(v\otimes v')\|_{H^{-1,q}(\T^3;\R^3)}
\lesssim_q \|v\|_{H^{\mu,q}(\T^3;\R^3)}\|v'\|_{H^{\mu,q}(\T^3;\R^3)}.
\end{equation} 
\emph{where $\mu=3/(2q)$. In particular, local well-posedness for the stochastic {{\normalfont 3D NSEs}} holds in the setting $\set=\setfull$ in \eqref{eq:weak_setting_NSEs} provided}
$$
\frac{1+\a}{p}+\frac{3}{2q}\leq 1.
$$ 
This is a special case of 
\cite[Lemma 4.2]{AV21_NS}. For the reader's convenience, we include some details. 
We first prove \eqref{eq:nonlinear_estimate_weak_setting_Step1}. Note that, by $H^{\mu,q}(\T^3)\embed L^{2q}(\T^3)$,
\begin{align*}
\|\nabla \cdot (v\otimes v')\|_{\Hs^{-1,q}}
\lesssim \|v\|_{L^{2q}} \|v'\|_{L^{2q}}
\lesssim \|v\|_{H^{\mu,q}}\|v'\|_{H^{\mu,q}}.
\end{align*}
Letting $\beta=\frac{1}{2}+\frac{3}{4q}$, the last claim follows by the results in Subsection \ref{ss:critical_stochastic_evol}, see Theorem \ref{t:lwp} there.

\smallskip

\emph{Step 3: (Conclusion).} Here, we apply Theorem \ref{t:singular_times_SPDEs}. 
Note that the singular times for the $\set$-setting considered here are given as in Definition \ref{def:singular_regular} with the choice \eqref{eq:weak_setting_NSEs}. From \cite[Theorem 2.7]{AV21_NS} and Lemma \ref{lem:instantaneous_reg} applied with $\mathscr{X}_{(s,t)}=C^{1/2-,(1+\g)-}_{\loc}((s,t)\times \T^3)$, it follows that the latter coincide with the one given in Definition \ref{def:singular_times_NS}.
First, recall that, by definition of annealed stochastic Leray-Hopf solutions, there exists $(\O_n)_n\subseteq \F_0$ such that $\O_n\uparrow \O$ and 
\begin{equation}
\label{eq:energy_bound_recalled_weak_setting}
\E\Big[\one_{\O_n}\int_0^T \|u\|_{\Hs^1(\T^3)}^2\,\dd t \Big] <\infty\  \text{ for all $n\geq 1$ and $T<\infty$.}  
\end{equation}
By Sobolev embeddings, it follows that 
\begin{equation}
\label{eq:embedding_xx_into_trace_space}
\Hs^1(\T^3) \embed \Xap= \Bs^{1-2(1+\a)/p}_{q,p}(\T^3),
\end{equation}
if, for instance, the parameters $(q,p,\a)$ satisfy
\begin{equation}
\label{eq:choice_parameters_weak_setting}
q\in (3,6), \ \  
1-2\frac{1+\a}{p}-\frac{3}{q}= 1-\frac{3}{2}\quad \Longrightarrow \quad \frac{1+\a}{p}= \frac{3}{2}\Big(\frac{1}{2}-\frac{1}{q}\Big).
\end{equation}
Note that the equality in the above ensures that the embedding in \eqref{eq:embedding_xx_into_trace_space} is \emph{sharp} (from a scaling viewpoint). 
Note that, if $q>2$, then $ \frac{3}{4}-\frac{3}{2q}>0$ and therefore the above equality can be verified for $p$ large and $\a\geq 0$. Moreover, note that $\a<\frac{p}{2}-1$ provided $ \frac{3}{4}-\frac{3}{2q}<\frac{1}{2}$, which implies $q<6$. 

The above observations and the second assertion in Step 1 show that Assumptions \ref{ass:standing_assumption} and \ref{ass:abstract1} hold. Moreover, note that solutions $(v,\tau)$ in the $\set$-setting satisfies
$$
v\in C([s,\tau);\Bs^{1-2(1+\a)/p}_{q,p}(\T^3))\subseteq C([s,\tau); \Ls^3(\T^3))\ \text{ a.s., }
$$
where the last embedding follows from the choice of the parameters in \eqref{eq:choice_parameters_weak_setting}. In particular, Proposition \ref{prop:weak_strong_uniqueness} shows that Assumption \ref{ass:abstract2} also holds in the current situation. Thus, it remains to compute the excess of $\Xap$ in \eqref{eq:embedding_xx_into_trace_space} from the criticality, $\gap_\set$. From Step 1, we know that 
$\rho=1$ and $\beta=\frac{\mu+1}{2}=\frac{1}{2}+\frac{3}{4q}$. Thus, from \eqref{eq:choice_parameters_weak_setting}, we obtain
$$
\gap_\set =\frac{\rho+1}{\rho} (1-\beta)-\frac{1+\a}{p}
= \frac{1}{4}.
$$
In particular, the excess from criticality is \emph{independent} of $(q,p,\a)$.

Theorem \ref{t:singular_times_SNS} now follows from Theorem \ref{t:singular_times_SPDEs} by recalling $\r=2$ due to \eqref{eq:energy_bound_recalled_weak_setting}.
\end{proof}

\begin{remark}[Noise regularity $\&$ $L^p$-setting]
\label{rem:necessity_Lp_setting}
In the usual approach to singular times to 3D NSEs (see e.g., \cite[Theorem 13.5]{LePi}), the typical choice is $X_0=L^2(\T^3)$, $X_1=\Hs^2(\T^3)$, $p=2$ and $\a=0$, i.e., the \emph{strong $L^2$-setting}. 
The same choice in our situation is also possible at the expense of additionally requiring $\g>1$ in Assumption \ref{ass:NSE_regularity_coefficients}. 
In particular, this excludes the case of \emph{rough} Kraichnan noise $\g\in (0,1)$, including the one reproducing the Kolmogorov spectrum of turbulence for which $\g=\frac{2}{3}$ (see the text below Assumption \ref{ass:NSE_regularity_coefficients}). 
The necessity of $\g>1$ in the case of the strong $L^2$-setting is needed for the stochastic maximal $L^2$-regularity in the strong setting (see \cite[Theorem 3.2]{AV21_NS}). A similar situation can be found in \cite[Subsection 1.1]{Primitive3}; see the comments below eq.\ (1.12) in the latter reference.
\end{remark}

\subsubsection{Proof of Theorem \ref{t:singular_times_SNS_2}}
\label{sss:proof_singular_times_SNS_2}
In this subsection, we again apply the results of Section \ref{s:singular_times_SPDEs}, in the form of Proposition \ref{prop:singular_times_SPDEs_weaker_weak_strong}. To this end, let $(r_0,q_0,\g_0)$ be as in Theorem \ref{t:singular_times_SNS_2}, and let $r_1\geq r_0$ be decided later. Recall that $\g>0$ is given in Assumption \ref{ass:NSE_regularity_coefficients}
Let $\set_0=\setfullzero$ where 
\begin{equation}
\label{eq:weak_setting_NSEs0}
X_j=\Hs^{-1+2j+\nu_0,q_0}(\T^3), \ \ \nu_0\in (-(\g\wedge 1),0),  \ \  r_1\in [2\vee r_0,\infty), \ \ \a_1\in [0,\tfrac{r_1}{2}-1),
\end{equation} 
for $j\in \{0,1\}$, 
and $(A,B,F,G)$ be as in \eqref{eq:weak_setting_NSEs_2}. The assumption $\nu_0>-\g$ ensures that the operators in \eqref{eq:weak_setting_NSEs_2} are well-defined in the $\set_0$-setting due to pointwise multiplier results \cite[Proposition 4.1]{AV_torus}, see \cite[eq.\ (3.4)-(3.5)]{AV21_NS}.

As mentioned in Subsection \ref{sss:compatibility_bounds}, on the one hand, solutions to \eqref{eq:SEE} in the $\set_0$-setting given in \eqref{eq:weak_setting_NSEs0} do \emph{not} in general belong $L^2_{t}(H_x^{1})$ but only to $L^{r_1}_{t}(H_x^{1+\nu_0,q_0})$. In particular, weak-strong uniqueness as in Proposition \ref{prop:weak_strong_uniqueness} is \emph{not} directly applicable. On the other hand, \cite[Theorems 2.4 and 4.1]{AV21_NS} ensures that paths of solutions to \eqref{eq:SEE} in the $\set_0$-setting with $(A,B,F,G)$ as in \eqref{eq:weak_setting_NSEs_2} instantaneously regularize to $C_t (H_x^{1,r})$ for all $r<\infty$. 
This ensures the \emph{compatibility} of the setting $\set_0$ and the one in \eqref{eq:weak_setting_NSEs} (see \eqref{eq:setting_compatibility} for the definition), hence allowing us to overcome the difficulties in applying the weak-strong uniqueness by means of Proposition \ref{prop:singular_times_SPDEs_weaker_weak_strong}.

\begin{proof}[Proof of Theorem \ref{t:singular_times_SNS_2}]
We begin by recalling that, similar to \eqref{eq:interpolation_X0_weak_spaces}, by interpolation, it holds that, for all $\theta\in (0,1)$ and $r\in (1,\infty)$,
$$
[X_0,X_1]_\theta= \Hs^{-1+\nu_0+2\theta,q_0}(\T^3) \quad \text{ and }\quad
(X_0,X_1)_{\theta,r}= \Bs^{-1+\nu_0+2\theta}_{q_0,r}(\T^3).
$$
Arguing as in Step 1 in the proof of Theorem \ref{t:singular_times_SNS}, from \cite[Theorem 3.1]{AV21_NS}, \cite[Theorem 3.1 and Proposition 4.1]{AV_torus} and $\nu_0\in (-\g,0)$, it follows that $(A,B)$ in \eqref{eq:weak_setting_NSEs_2} has stochastic maximal $L^{r}$-regularity on $(X_0,X_1,r,\alpha)$ for all $r\in (2,\infty)$ and $\alpha\in [0,\frac{r}{2}-1)$. 
Similar to Step 2 of Theorem \ref{t:singular_times_SNS}, \cite[Lemma 4.2]{AV21_NS} ensures that
\begin{equation}
\label{eq:inequality_F_G_Lq_SNS}
\|F(v)-F(v')\|_{X_0}\lesssim (\|v\|_{X_{\beta_0}}
+\|v'\|_{X_{\beta_0}})\|v-v'\|_{X_{\beta_0}} \ \ \text{ for all } \ v,v'\in X_1.
\end{equation}
where $\beta_0=\frac{1}{2}(1-\frac{\nu_0}{2}+\frac{3}{2q_0})$ provided
$
\frac{3}{2+\nu_0}<q_0<\tfrac{3}{-\nu_0}.
$
Note that the lower bound in the latter condition is automatically satisfied as $q_0>3$ and $\nu_0> -1$ by assumption.

\smallskip

Pick $\nu_0\in (-(\tfrac{3}{q_0}\wedge \g),-\g_0)$ (this choice is possible as $\g_0<3/q_0$ and $\g>\g_0$).
Clearly, $q_0<\frac{3}{-\nu_0}$, and $1+\nu_0+\g_0\in (0,1)$ since $-1<\nu_0<-\g_0$.
Fix $r_1\in [2\vee r_0,\infty)$ such that $1+\nu_0+\g_0>\frac{2}{r_1}$. In particular, there exists $\kappa_1\in [0,\frac{r_1}{2}-1)$ such that
\begin{equation}
\label{eq:choice_weight_time}
\frac{1+\kappa_1}{r_1}=\frac{1}{2}(1+\nu_0+\g_0).
\end{equation}
Note also that $\beta_0>1-\frac{1+\kappa_1}{r_1}$,
since $\nu_0>-\frac{3}{q_0}$. Since
$$
\xx_0=\Bs^{-\g_0}_{q_0,r_0}(\T^3)
\embed \Bs^{1+\nu_0-2\frac{1+\kappa_1}{r_1}}_{q_0,r_1}(\T^3),
$$
by \eqref{eq:inequality_F_G_Lq_SNS} and the comments below it, the excess from criticality in the $\set_0$-setting is given by (where, clearly, $\rho_0=1$)
$$
\gap_{\set_0}= \frac{\rho_0+1}{\rho_0}(1-\beta_0)-\frac{1+\a_1}{r_1}
=\frac{1}{2} \Big(1-\frac{3}{q_0}-\g_0\Big).
$$
As expected, $\gap_{\set_0}$ is given by the difference of the Sobolev indices of Besov $\Bs^{-\g_0}_{q_0,r_1}(\T^3)$ and $\Bs^{\frac{3}{q_0}-1}_{q_0,r_1}(\T^3)$, rescaled by the $1/2$ due to the parabolic scaling (or due to the order of the operator). 
Finally, due to the assumed energy bound in Theorem \ref{t:singular_times_SNS_2}, letting $\r_0=p_0$, we have
$$
1-\r_0 \,\gap_{\set_0}
= \frac{p_0}{2}\Big(\frac{2}{p_0}+\frac{3}{q_0}+\g_0-1\Big).
$$
Now, \eqref{eq:singular_NSE1_cond} and \eqref{eq:singular_NSE2_cond} follow from Proposition \ref{prop:singular_times_SPDEs_weaker_weak_strong} applied with $\set_0=\set$ and $\sety$ as in \eqref{eq:weak_setting_NSEs}, as strong weak-strong uniqueness holds in the $\sety$-setting due to Proposition \ref{prop:weak_strong_uniqueness}. 
Let us point out that, as in Theorem \ref{t:singular_times_SNS_2}, we used that \cite[Theorem 2.7]{AV21_NS} and Lemma \ref{lem:instantaneous_reg} applied with $\mathscr{X}_{(s,t)}=C^{1/2-,(1+\g)-}_{\loc}((s,t)\times \T^3)$ imply that the singular times in the $\set_0$-setting (see Definition \ref{def:singular_regular}) in \eqref{eq:weak_setting_NSEs0} coincide with the one given in Definition \ref{def:singular_times_NS}. 
\end{proof}

\begin{remark}[Necessity of time weights]
\label{rem:necessity_time_weights}
Weights in time are essential in the proof of Theorem \ref{t:singular_times_SNS_2} to allow:
\begin{itemize}
\item for arbitrary large exponent $r_0$ in \eqref{eq:bound_Besov_SNS};
\item for optimal results under $L^{q_1}$-bounds as in \eqref{eq:supercritical_Serrin_conditions}. 
\end{itemize}
To see the latter, note that due to \eqref{eq:embedding_Lp_B0}, forcing $\a_1=0$ in the condition \eqref{eq:choice_weight_time} leads to unnatural restrictions on $q_1$ when $\g$ is small (and consequently, for $\nu_0$ and $\g_0$ as well). 
\end{remark}

\subsubsection{Proof of Proposition \ref{prop:weak_strong_uniqueness}}
\label{sss:weak_strong_proof_SNS}
Before diving into the proof of Proposition \ref{prop:weak_strong_uniqueness}, we collect some facts on the regularity of the various solutions to the stochastic NSEs \eqref{eq:SNS_introduction} introduced above. 
Let $u$ be an annealed stochastic Leray-Hopf solution to \eqref{eq:SNS_introduction}. Then, it solves the identity in Definition \ref{def:Leray_Hopf_strong}\eqref{it:Leray_Hopf_strong2}. From the regularity in Definition \ref{def:Leray_Hopf_strong}\eqref{it:Leray_Hopf_strong1} and interpolation, it follows that 
$$
u\in L^{4}_{\loc}([0,\tau);H^{1/2}(\T^3))\subseteq L^4_{\loc}([0,\tau);L^3(\T^3)) \text{ a.s. }
$$ 
In particular, the convective term satisfies $\nabla \cdot (u\otimes u)\in L^2_{\loc}([0,\tau);H^{-1,3/2}(\T^3))$ a.s.\ 
Moreover, from the assumed regularity in Definition \ref{def:Leray_Hopf_strong}\eqref{it:Leray_Hopf_strong1} and Assumption \ref{ass:NSE_regularity_coefficients}\eqref{it:NSE_regularity_coefficients2}, it follows that $(\L_n u)_n \in L^2_{\loc}([0,\tau);\Ls^2(\T^3;\ell^2))$. Hence, from Definition \ref{def:Leray_Hopf_strong}\eqref{it:Leray_Hopf_strong1} and then Hahn-Banach theorem, $u$ solves, a.s.\ for all $t\in [0,\tau)$,
\begin{equation}
\label{eq:identity_u_H_LH}
u(t)=u_0+\int_0^t \big(\Delta u -\p[\nabla \cdot(u\otimes u)] + \A u )\,\dd s + 
\int_0^t (\L_n u)_n \,\dd \mathcal{W}_{\ell^2},
\end{equation}
in $\Hs^{-1,3/2}(\T^3)$, and where $\mathcal{W}_{\ell^2}$ is as in \eqref{eq:def_Well2}. 
Next, let $v$ be an $L^3$-solution to \eqref{eq:SNS_introduction}, see Definition \ref{def:Lq_solution_SNS}. 
For all $k\geq 1$, let $\tau_k$ be the stopping time 
\begin{equation}
\label{eq:tau_k_SNS}
\tau_k\stackrel{{\rm def}}{=}
\inf\Big\{t\in [0,\tau)\,:\, \|v(t)\|_{L^{3}}+\int_0^t \|\nabla v(r)\|_{L^2}^2\,\dd r\geq k\Big\}\wedge k, 
\end{equation}
where $\inf\emptyset \stackrel{{\rm def}}{=}\tau\wedge k$.
In particular, for all $k\geq 1$, it holds that $\|v(t)\|_{L^3(\T^3)}+\int_0^t \|\nabla v(r)\|_{L^2}^2\,\dd r\leq k$ a.s.\ for $t< \tau_k$. Hence, from \eqref{eq:assumption_regularity_strong_SNS} and the estimate \eqref{eq:interpolation_Lq_estimates},
\begin{equation}
\label{eq:L2_of_strong_solutions}
\nabla \cdot (v\otimes v) \in L^2(\O\times (0,\tau_k);H^{-1}(\T^3;\R^3)).
\end{equation}
From the stochastic maximal $L^2$-regularity (see e.g., \cite[Lemma 4.1]{AV24_variational}, it follows that $v$ has an extension on $[0,\tau_k]\times \O$ with values in $\Ls^2(\T^3)$, which we still denote by $v$. Hence, the evaluation $v(\tau_k)$ is well-defined for all $k\geq 1$.

With the above observations at our disposal, we now prove Proposition \ref{prop:weak_strong_uniqueness}.

\begin{proof}[Proof of Proposition \ref{prop:weak_strong_uniqueness}]
In this proof, we extend the argument in the deterministic case (see e.g., \cite[Theorem 12.4]{LePi}), which we adapt to accommodate the annealed version of the energy inequality as in Definition \ref{def:Leray_Hopf_strong}\eqref{it:Leray_Hopf_strong3}. 
Clearly, to prove Proposition \ref{prop:weak_strong_uniqueness}, it suffices to show that, for all $k\geq 1$, 
\begin{equation}
\label{eq:equality_v_u_tau_k}
v=u \text{ a.e. on }[0,\tau_k]\times \O.
\end{equation} 
Therefore, in the following, we fix $k\geq 1$. Let $\chi\in C^\infty([0,\infty))$ be non-decreasing, with values in $[0,1]$ and $\chi|_{[0,1/2]}=0$ and $\chi|_{[1,\infty)}=1$. Let $\eta>0$ be sufficiently small. For $k,R\geq 1$, set
$$
\mu_{k,R}\stackrel{{\rm def}}{=}\tau_k\wedge
\inf\big\{t\in[0,\tau_k)\,:\, \|\chi(|v(t)|/R)v(t)\|_{L^3}\geq\eta\big\},
$$
where $\inf\emptyset\stackrel{{\rm def}}{=}\infty$. One can readily check that $\mu_{k,R}$ is a stopping time. For brevity, we let 
$$
\mu\stackrel{{\rm def}}{=}\mu_{k,R}\quad \text{ and } \quad w\stackrel{{\rm def}}{=}u-v.
$$ 
Next, for simplicity, we only consider the case $u_0\in L^2(\O;\Ls^2(\T^3))$. The general case follows similarly by using the localizing sequence $(\O_n)_n$ associated with the annealed energy inequality for $u$; see Definition \ref{def:Leray_Hopf_strong}\eqref{it:Leray_Hopf_strong3}. For notational convenience, for $\phi,\psi\in H^1(\T^3;\R^3)$, we denote by $\Abin$ the bilinear form associated to the annealed energy dissipation of the linear operators in \eqref{eq:SNS_Ito}:
\begin{align*}
\Abin(\phi,\psi)
&\stackrel{{\rm def}}{=}\int_{\T^3} \big(-\nabla \phi:\nabla \psi +\sum_{n\geq 1} \L_n \phi\cdot S_n \psi  \big) \,\dd x.
\end{align*}

Fix $\varphi\in C^{\infty}(\T^3)$ with $\int_{\T^3}\varphi\,\dd x =1$ and $\textrm{supp}\,\varphi\subseteq \{|x|<1/2\}$. Let $\varphi_{\varepsilon}= \varepsilon^{-d}\varphi(\cdot/\varepsilon)$ be the corresponding smooth approximation of the identity. 
For brevity, let us set 
$$
w^\mu(t)\stackrel{{\rm def}}{=} w(t\wedge \mu),
$$
and similar for $u^\mu$ and $v^\mu$. Thus, we can write, a.s.\ for all $t> 0$, 
\begin{align*}
\E\|w^\mu(t)\|_{L^2}^2
& \textstyle
=\E\|u^\mu(t)\|_{L^2}^2-\E\|v^\mu(t)\|_{L^2}^2 - 2\,\E\int_{\T^3}v^\mu(t)\cdot w^\mu(t)\,\dd x\\
& \textstyle
=\E\|u^\mu(t)\|_{L^2}^2- \E\|v^\mu(t)\|_{L^2}^2 
-2\, \lim_{\varepsilon\downarrow 0} \E \int_{\T^3} (\varphi_\varepsilon *v^\mu)(t) \cdot w^\mu(t)\,\dd x,
\end{align*}
where $*$ denotes the convolution on $\T^3$.
Next, we estimate each term on the previous identity separately. Firstly, from the annealed energy inequality for $u$ (i.e., Definition \ref{def:Leray_Hopf_strong}\eqref{it:Leray_Hopf_strong3} with $t_0=0$, $\O_n\equiv \O$ and $\tau_0=\mu\wedge t$), we obtain, for all $t>0$, 
\begin{align*}
\frac{1}{2}\,\E\|u^\mu(t)\|_{L^2}^2
\leq 
\frac{1}{2}\,  \E \|u_0\|_{L^2}^2& + \E \int_{0}^{t\wedge \mu} \Abin (u(r),u(r))\,\dd r .
\end{align*}
Moreover, from \eqref{eq:L2_of_strong_solutions} and the discussion below it, it is possible to apply It\^o's formula in \cite[Theorem 4.2.5]{LR15}, and thus, for all $t>0$, we obtain the following energy equality for $v^\mu$:
\begin{align*}
\frac{1}{2}\,\E\|v^\mu(t)\|_{L^2}^2
=
\frac{1}{2} \, \E \|u_0\|_{L^2}^2& + \E \int_{0}^{t\wedge \mu} \Abin(v(r),v(r)) \,\dd r .
\end{align*}
Putting together the previous observations, we have proved that, a.s.\ for all $t>0$,
\begin{align}
\label{eq:energy_inequality_difference}
\frac{1}{2}\,
\E\|w^\mu(t)\|_{L^2}^2
&\leq \E \int_{0}^{t\wedge \mu} \big(\Abin (u(r),u(r))-\Abin (v(r),v(r))\big)\,\dd r
-\lim_{\varepsilon\downarrow 0} 
\reim(t),\\
\nonumber
\reim(t)&\stackrel{{\rm def}}{=}
\E \int_{\T^3} (\varphi_\varepsilon *v^\mu)(t) \cdot w^\mu(t)\,\dd x.
\end{align}
For exposition convenience, we now split the proof into three steps.

\smallskip

\emph{Step 1: For all $t>0$, the following identity holds:}
\begin{align*}
\lim_{\varepsilon\downarrow 0}
\reim(t)
=
\E \int_{0}^{t\wedge \mu} \big(\Abin(v,w)+ \Abin(w,v)\big) \,\dd r-
\E \int_{0}^{t\wedge \mu}\int_{\T^3} v\cdot (w\cdot \nabla) w\,\dd x \,\dd r .
\end{align*}
From \eqref{eq:identity_u_H_LH} and Definition \ref{def:Lq_solution_SNS}, the It\^o formula applied to the bilinear formulation $(U,U')\mapsto \int_{\T^3} (\varphi_\varepsilon * U )\cdot U' \,\dd x$ for $\varepsilon>0$ (see e.g., \cite[Corollary 2.6]{NVW2}), one can check that, for all $t>0$,
\begin{align*}
&\lim_{\varepsilon\downarrow 0}\reim(t)
=\E \int_{0}^{t\wedge \mu} \big(\Abin(v,w)+ \Abin(w,v)\big) \,\dd r\\
\nonumber
&-\lim_{\varepsilon\downarrow 0}
\E \int_{0}^{t\wedge \mu}\int_{\T^3} \big( (\varphi_\varepsilon * [\nabla \cdot (v\otimes v)])\cdot  w
+(\varphi_\varepsilon *  v )\cdot [\nabla \cdot(u\otimes u-v\otimes v)]\big)\,\dd x \,\dd r.
\end{align*}
Note that we used the assumption that $u,v\in L^2(\O\times (0,\mu);H^1)$ to obtain the first term on the right-hand side.

Recall that $|\int_{\T^3} (\psi\otimes \psi):  \nabla \phi\,\dd x |\leq \|\psi\|_{L^3}\|\phi\|_{H^1}\|\psi\|_{L^6}$ and $H^1(\T^3)\embed L^6(\T^3)$ by Sobolev embeddings. Since $\|v\|_{L^\infty(0,\mu;L^3)}\leq k $ and $v,w\in L^2(\O;L^\infty( 0,\mu;L^2))\cap L^2(\O\times (0,\mu);H^1)$, it follows from integration by part that, for all $t>0$,
\begin{align*}
&\lim_{\varepsilon\downarrow 0}\E\int_{0}^{t\wedge \mu} \int_{\T^3}(\varphi_\varepsilon * [\nabla \cdot (v\otimes v)])\cdot  w\,\dd x\,\dd r \\
&=-\lim_{\varepsilon\downarrow 0}\E
\int_{0}^{t\wedge \mu} \int_{\T^3} (\varphi_\varepsilon *  v\otimes v) :\nabla w\,\dd x\,\dd r
=-\,
\E\int_{0}^{t\wedge \mu} \int_{\T^3} v\cdot (v\cdot\nabla) w \,\dd x\,\dd r.
\end{align*}
With similar reasoning, it holds that, for all $t>0$, 
\begin{align*}
\lim_{\varepsilon\downarrow 0}
\E \int_{0}^{t\wedge \mu}\int_{\T^3} (\varphi_\varepsilon *  v) 
\cdot [\nabla \cdot(u\otimes u-v\otimes v)]\,\dd x\,\dd r
&= 
\E \int_{0}^{t\wedge \mu}\int_{\T^3} v\cdot [(u\cdot \nabla) u]\,\dd x \,\dd r\\
&= 
\E \int_{0}^{t\wedge \mu}\int_{\T^3} v\cdot (u\cdot \nabla) w\,\dd x \,\dd r,
\end{align*}
where we used the standard cancellations 
$$
\int_0^{t\wedge \mu}\int_{\T^3} v\cdot [(v\cdot \nabla)v] \,\dd x \,\dd r
= \int_0^{t\wedge \mu}\int_{\T^3} v\cdot [(u\cdot \nabla )v] \,\dd x\,\dd r=0 \ \text{ a.s., }
$$ 
which, again, hold since $v\in L^\infty(\O\times (0,\mu);L^3)\cap L^2(\O\times (0,\mu);H^1)$ and $u\in L^2(\O;L^\infty((0,\mu);L^2))\cap L^2(\O\times (0,\mu);H^1)$.
Therefore, the claim of Step 1 readily follows by collecting the above identities.

\smallskip

\emph{Step 2: For $\eta>0$ sufficiently small,
there exists a deterministic constant $C_R>0$
such that, for all $t>0$,}
$$
\E\Big|\int_0^{t\wedge\mu}\int_{\T^3}
v\cdot[(w\cdot\nabla)w]\,\dd x\,\dd r\Big|
\leq \frac12\,\E\int_0^{t\wedge\mu}
\|\nabla w\|_{L^2}^2\,\dd r
+C_R\,\E\int_0^t\|w^\mu(r)\|_{L^2}^2\,\dd r.
$$

Indeed, write $ v=\one_{\{|v|\leq R\}}v+ \one_{\{|v|>R\}}v. $
By construction of $\mu$, a.s.\ for all $s<\mu$,
$$
\|\one_{\{|v|>R\}}v(s)\|_{L^3}\leq\eta,
\qquad
\|\one_{\{|v|\leq R\}}v(s)\|_{L^\infty}
\leq R.
$$
Hence, by the H\"older inequality and the Sobolev embedding $H^1(\T^3)\embed L^6(\T^3)$ as well as Young's inequality, it follows that 
\begin{align*}
\Big|\int_{\T^3}v\cdot[(w\cdot\nabla)w]\,\dd x\Big|
&\leq \eta\|w\|_{L^6}\|\nabla w\|_{L^2} +R\|w\|_{L^2}\|\nabla w\|_{L^2}\\
&\leq \eta C\|\nabla w\|_{L^2}^2 +C_R\|w\|_{L^2}^2.
\end{align*}
The claim of Step 2 follows by choosing $\eta$ sufficiently small and integrating on $(0, t\wedge\mu)$.

\smallskip

\emph{Step 3: Conclusion.}
From the identity
$
\Abin(\phi-\psi,\phi-\psi)=
\Abin(\phi,\phi)- 
\Abin(\psi,\psi)-
\Abin(\psi,\phi-\psi)-
\Abin(\phi-\psi,\psi)
$
for $\phi,\psi\in H^1(\T^3;\R^3)$, using the estimates of Steps 1 and 2 in \eqref{eq:energy_inequality_difference}, we obtain
\begin{align*}
\frac{1}{2}\,\E\|w^\mu(t)\|_{L^2}^2
&\leq \E \int_{0}^{t\wedge \mu} \int_{\T^3}\Big(- \frac{1}{2}|\nabla w|^2+\sum_{n\geq 1} \L_n w\cdot S_n w\Big)\,\dd x \,\dd r
+C_R \,\E \int_0^{t} \| w^\mu(r)\|_{L^2}^2\,\dd r.
\end{align*}
Note that Assumption \ref{ass:NSE_regularity_coefficients} implies, for all $\delta>0$,
$$
\Big|\E \int_{0}^{t\wedge \mu}\int_{\T^3} \sum_{n\geq 1} \L_n w\cdot S_n w\,\dd x \,\dd r\Big|
\leq\delta \,\E \int_0^{t\wedge \mu} \int_{\T^3} |\nabla w|^2 \,\dd x \,\dd r
+ C_\delta \,\E \int_0^{t} \| w^\mu(r)\|_{L^2}^2\,\dd r .
$$
Choosing $\delta=1/4$ and applying Gr\"onwall's lemma, we obtain $w^\mu(t)=0$ a.s.\ for every $t\geq 0$.
By restricting to rational times and using weak $L^2$-continuity, 
on a set of full probability, we have, for all integers $k,R\geq 1$,
\begin{equation}
\label{eq:u_equal_v_on_muR_weak_strong}
u=v \ \text{ on } [0,\mu_{k,R}].
\end{equation} 
Fix $a<\tau_k$. Since $v$ has continuous $L^3$-paths and $\chi$ is non-decreasing, the functions
$
t\mapsto \|\chi(|v(t)|/R)v(t)\|_{L^3}
$
are continuous on $[0,a]$ and decrease pointwise to zero as $R\to\infty$. 
By Dini's theorem on uniform convergence, the convergence is uniform on $[0, a]$. Thus, for $R$ sufficiently large, we have $\mu_R\geq a$. Since $a<\tau_k$ is arbitrary and $\mu=\mu_{k,R}$ is non-decreasing in $R$, it follows that $\mu_R\uparrow\tau_k$ a.s.\ as $R\to\infty$. Letting $R\to\infty$ and then $k\to\infty$ in \eqref{eq:u_equal_v_on_muR_weak_strong}, we obtain \eqref{eq:equality_v_u_tau_k}. This concludes the proof.
\end{proof}

\subsection{Stochastic strong Leray-Hopf solutions -- Proof of Proposition \ref{prop:existence_strong_Leray}}
\label{ss:existence_strong_Leray_proof}
In this subsection, we sketch the existence of stochastic Leray-Hopf solutions to \eqref{eq:SNS_introduction}. 
As commented below Proposition \ref{prop:existence_strong_Leray}, their existence might be known to experts, and follows the standard compactness argument as in, e.g., \cite{FG95_PTRF,FL32_book,MR05}. 
Here, we mainly focus on the proof of the strong pathwise annealed energy inequality \eqref{eq:strong_energy_inequality}. 
To this end, we need to go over again the main stochastic compactness argument.
Here, we content ourselves to prove the existence and the corresponding pathwise strong energy inequality \eqref{eq:strong_energy_inequality} on $[0,T]$, where $T>0$ is arbitrary. For the general case, one applies a further compactness argument as outlined in \cite[Appendix B]{DGGS26}. 

Let $\Lambda$ be a probability measure on $\Ls^2(\T^3)$. By the Skorokhod representation theorem, there exists a probability space and an $\Ls^2$-valued random variable $u_0$ satisfying $\text{Law}(u_0)=\Lambda$.   
Pick a complete filtered probability space $(\O,(\F_t)_t,\F,\P)$ satisfying the usual conditions endowed with a sequence of standard independent $(\F_t)_{t}$-Brownian motions $(B^n)_{n}$, and such that $u_0:\O\to \Ls^2(\T^3)$ is $\F_0$-measurable. 
Following Leray's approach \cite{Leray} (see also \cite[Subsection 12.2]{LePi}), we consider the following approximation of the stochastic 3D NSEs in It\^o's form \eqref{eq:SNS_Ito}:
\begin{equation}
\label{eq:SNS_Ito_moll}
\left\{
\begin{aligned}
\partial_t u_\varepsilon 
&= \Delta u_\varepsilon-\p\big[\nabla \cdot ( u_\varepsilon\otimes [\moll_\varepsilon u_\varepsilon]) \big]+ \A u_\varepsilon 
+\sum_{n\geq 1} \p\big[ (\sigma_n\cdot \nabla) u_\varepsilon + \cc_n \cdot u_\varepsilon \big] \, \dot{B}^n,\\
u_\varepsilon(0)&=u_0,
\end{aligned}
\right.
\end{equation}
both on $\T^3$; and where
$
\moll_\varepsilon= \varphi_\varepsilon*
$
is a standard mollifiers, that is $\varphi_\varepsilon= \varepsilon^{-d}\varphi(\cdot /\varepsilon)$ for some $\varphi\in C^\infty(\T^d)$ and $\int_{\T^d} \varphi\,\dd x=1$ satisfying $\textrm{supp}\,\varphi \subseteq \{|x|< 1/2\}$. 

The Stratonovich integration in \eqref{eq:SNS_Ito_moll} is understood as an It\^o-integral plus a correction as in \eqref{eq:def_A_ito_stratonovich_sns0}-\eqref{eq:def_A_ito_stratonovich_sns}.
Moreover, the SPDE \eqref{eq:SNS_Ito_moll} is understood as a stochastic evolution equation on the Hilbert space $\Hs^{-1}(\T^3)$. 
Indeed, recall from Subsection \ref{sss:function_spaces_divfree} that $\Hs^{-1}(\T^3)=(\Hs^{1}(\T^3))^*$. In particular, we can define the operator $\A$ acting from $\Hs^1(\T^3)$ into $\Hs^{-1}(\T^3)$ as follows: (see Subsection \ref{sss:ito_stratonovich_operators} for the notation)
$$
\langle \A u  , v \rangle =\frac{1}{2}\sum_{n\geq 1} \int_{\T^3} \L_n u\cdot \L_n^\top v\, \dd x .
$$
A \emph{global unique solution $u_\varepsilon$ to \eqref{eq:SNS_Ito_moll}} is a progressively measurable process 
$
u_\varepsilon:[0,\infty)\times \O\to \Hs^1(\T^3)
$
such that 
\begin{equation}
\label{eq:regularity_uepsilon_regularity}
u_\varepsilon\in L^2_{\loc}([0,\infty);\Hs^1(\T^3))\cap C([0,\infty);\Ls^2(\T^3)) \text{ a.s.},
\end{equation} 
and a.s.\ for all $t>0$ satisfies
\begin{align*}
u_\varepsilon(t)=u_0 + \int_0^t \big(\Delta u_\varepsilon - \p[\nabla \cdot ( u_\varepsilon \otimes \moll_\varepsilon u_\varepsilon)]+ \A u_\varepsilon\big)\,\dd s
+ \int_0^t (\L_n u_\varepsilon)_n  \,\dd \mathcal{B}_{\ell^2}
\end{align*}
in $\Hs^{-1}(\T^3)$. Here, $\mathcal{B}_{\ell^2}$ is the $\ell^2$-cylindrical Brownian motion associated to the sequence of standard independent Brownian motions $(B^n)_n$.
We emphasize that in the above, the deterministic and stochastic integrals are understood as an $\Hs^{-1}(\T^3)$-valued and $\Ls^2(\T^3)$-valued Bochner and It\^o integrals, respectively. 
Their existence readily follows from \eqref{eq:regularity_uepsilon_regularity}, the fact that $\moll_\varepsilon: L^2(\T^3)\to L^\infty(\T^3) $ is bounded, and Assumption \ref{ass:NSE_regularity_coefficients}.

The existence of such a global unique solution $u_\varepsilon$ to \eqref{eq:SNS_Ito_moll} is standard, and it readily follows from the variational approach to SPDEs (see e.g., \cite[Chapter 4]{LR15} or \cite[Theorem 3.4]{AV24_variational}). 
In the following, we collect the needed $\varepsilon$-uniform bounds for the unique global solution $u_\varepsilon$, which allow us to use stochastic compactness. 

\begin{lemma}[Uniform-in-$\varepsilon$-estimates]
\label{lem:epsilon_estimates_sketch_NS}
Fix $p\in [2,\infty)$ and $T\in (0,\infty)$. For $n\geq 1$, let $\O_n\stackrel{{\rm def}}{=}\{\|u_0\|_{L^2}\leq n\}$. Then, there exist constants $r>2$, $\alpha\in (\frac{1}{r},\frac{1}{2})$ and $\beta,C>0$ independent of $n\geq 1,\varepsilon>0$ and $u_0$ such that
\begin{align*}
\E\Big[\one_{\O_n}\sup_{t\in [0,T]} \|u_\varepsilon(t)\|_{L^2}^p\Big] +  \E\Big[\one_{\O_n}\Big(\int_0^T \|\nabla u_\varepsilon\|_{L^2}^2\,\dd t \Big)^{p/2}\Big]
\leq C(1+ \E[\one_{\O_n}\|u_0\|_{L^2}^p]),\\
\E\big[\one_{\O_n}\|u_\varepsilon\|_{W^{\alpha,r}(0,T;H^{-\beta})}^2\big] 
\leq C(1+ \E[\one_{\O_n}\|u_0\|_{L^2}^{2r}]).
\end{align*}
\end{lemma}

Recall for a Hilbert space $H$ and $\alpha\in (0,1)$, the fractional Sobolev $W^{\alpha,r}(0,T;H)$ denotes the set of maps $v\in L^r(0,T;H)$ such that  
$$
[v]_{W^{\alpha,r}(0,T;H)}
\stackrel{{\rm def}}{=}\Big(\int_0^T
\int_0^T\frac{\|v(t)-v(s)\|_H^r}{|t-s|^{1+r\alpha}}\,\dd t \,\dd s \Big)^{1/r}<\infty,
$$
endowed with the natural norm. The proof of the above is standard. Indeed, the first claimed estimate in Lemma \ref{lem:epsilon_estimates_sketch_NS} follows from the It\^o's formula applied to compute $\|u_\varepsilon(t)\|_{L^2}^2$ (see e.g., \cite[Theorem 4.2.5]{LR15}) and the standard cancellation $\int_{\T^3} \nabla \cdot ( v\otimes [\moll_\varepsilon v])\cdot v\,\dd x =0 $ for all $\varepsilon>0$ and $v\in \Hs^{1}(\T^3)$. Finally, the claimed estimate in Lemma \ref{lem:epsilon_estimates_sketch_NS} is a consequence of the first and well-known results on time regularity of stochastic integrals; see e.g., \cite[Lemma 2.1]{FG95_PTRF} or \cite[Lemma 2.7]{VP18}.

\smallskip

We are now ready to prove Proposition \ref{prop:existence_strong_Leray}.

\begin{proof}[Proof of Proposition \ref{prop:existence_strong_Leray} -- Sketch]
By the stochastic compactness method (see e.g., \cite[Section 2]{BFH18_book}), there exists a sequence $(\varepsilon_j)_{j\to \infty}$ such that $\varepsilon_j\to 0$, random variables $(\wt{u}_{j},(\wt{B}^{n,j})_n)\stackrel{{\rm Law}}{=}(u_{\varepsilon_j},(B^n)_n)$, and complete filtered probability space $(\wt{\O},\wt{\F},\wt{\P})$ endowed with a filtration satisfying the usual condition such that $\wt{\P}$-a.s.\  $(\wt{B}^{n,j})_n\to (\wt{B}^n)_n$ in $C([0,T];\ell^2_0)$ where $\ell^2_0$ is an auxiliary Hilbert space such that the embedding $\ell^2\embed \ell^2_0$ is Hilbert-Schmidt, and
\begin{align}
\label{eq:convergenza_u_uepsilon}
&\wt{u}_{\varepsilon_j} \to \wt{u} 
\text{
weakly-$^*$ in 
$
L^\infty(0,T;L^2)\cap L^2(0,T;H^1)$,}\\
\nonumber
& \text{and strongly in $W^{\alpha_0,r}([0,T];H^{-\beta_0})$ 
for some $r>2$, $\alpha_0>1/r$ and $\beta_0>0$.}
\end{align} 
This clearly implies (see the text above Definition \ref{def:Leray_Hopf_strong} for the notation)
\begin{equation}
\label{eq:proof_strong_annealed_L2weakly}
\wt{u}\in C_{{\rm w}}([0,T];L^2)\text{ $\wt{\P}$-a.s.\ }
\end{equation}
By standard arguments (see e.g., \cite{FG95_PTRF}), one can check that \eqref{eq:convergenza_u_uepsilon} and \eqref{eq:proof_strong_annealed_L2weakly} are sufficient to prove that $\wt{u}$ satisfies the conditions in Definition \ref{def:Leray_Hopf_strong}\eqref{it:Leray_Hopf_strong1}-\eqref{it:Leray_Hopf_strong2}. In the remaining part of the proof, we prove the strong pathwise energy inequality \eqref{eq:strong_energy_inequality}.

By interpolation, it follows that $\wt{\P}$-a.s.\ $\wt{u}_{j}\to \wt{u}$ weakly in $L^{2/\theta}(0,T;H^{\theta})$ for all $\theta\in (0,1)$. The latter, the $\wt{\P}$-a.s.\ strong convergence in $W^{\alpha_0,2}([0,T];H^{-\beta_0})$ and again interpolation yields $\wt{\P}$-a.s.\
$\wt{u}_{\varepsilon_j}\to \wt{u}$ strongly in $L^{r}(0,T;L^2)$ for all $r<\infty$.
After extracting a further subsequence, the strong $L^2(0,T;L^2)$ convergence yields
\begin{equation}
\label{eq:strong_convergence_L2}
\wt{u}_j(t,\wt{\om})\stackrel{j\to \infty}{\to} \wt{u}(t,\wt{\om}) \text{ strongly in }L^2(\T^3).
\end{equation}
In particular, by Fubini's theorem, there exists a \emph{full measure} set $I_*\subseteq[0,T]$ such that the convergence in \eqref{eq:strong_convergence_L2} holds for almost all $\wt{\om}\in\wt{\O}$ and for \emph{all} $t\in I_*$.
For every $t\in I_*$, the convergence holds for $\wt{\P}$-a.e. $\wt{\om}$, with the exceptional set possibly depending on $t$.

Now, fix $t_0\in I_*$. Let $J_{*}$ be a dense countable subset of $I_*$.   
Due to the cancellation mentioned below Lemma \ref{lem:epsilon_estimates_sketch_NS} and the considerations in Subsection \ref{sss:ito_stratonovich_operators}, for all $j\geq 1$, it holds that, $\wt{\P}$-a.s.\ for all $0\leq t_0\leq t\leq T$,
\begin{align}
\label{eq:energy_balance_tilde}
\frac{1}{2}\|\wt{u}_j(t)\|_{L^2}^2
&+ \int_{t_0}^t \int_{\T^3}|\nabla \wt{u}_j|^2\,\dd x \,\dd r=\frac{1}{2}\|\wt{u}_j(t_0)\|_{L^2}^2 \\
\nonumber
& + \sum_{n\geq 1}\int_{t_0}^t \int_{\T^3} (\cc_n\cdot \wt{u}_j)\cdot \wt{u}_j\,\dd x\,\dd \wt{B}^{n,j}
+\sum_{n\geq 1} \int_{t_0}^t \int_{\T^3}  \L_n \wt{u}_j \cdot S_n \wt{u}_j  \,\dd x \,\dd r.
\end{align}
Recall that $t_0\in I_*$, and fix $t\in J_{*}\cap [t_0,T]$. Letting $j\to \infty$, by \eqref{eq:strong_convergence_L2} and the comments below it, it follows from \eqref{eq:convergenza_u_uepsilon} that \eqref{eq:energy_balance_tilde} passes to the limit a.s., yielding the corresponding inequality for $(\wt{u}, \wt{B}^n)$ for the chosen $t_0$ and $t$.
Here, the convergence of the stochastic integral follows from \eqref{eq:convergenza_u_uepsilon}, Assumption \ref{ass:NSE_regularity_coefficients}\eqref{it:NSE_regularity_coefficients2} and e.g., \cite[Lemma 2.6.5]{BFH18_book} or \cite[Lemma 1.1]{DGHT11}.

To deduce now the pathwise strong energy inequality for \emph{all} $t\geq t_0$, recall that \eqref{eq:proof_strong_annealed_L2weakly} implies the existence of a set $\wt{\O}_0$ of full probability such that, for all $t\geq t_0$,
$$
\|\wt{u}(t)\|_{L^2}^2\leq \liminf_{J_* \ni t_*\to t} \|\wt{u}(t_*)\|_{L^2}^2 \  \text{ on }\wt{\O}_0.
$$ 
Combining the above and the comments below \eqref{eq:energy_balance_tilde}, we obtain \eqref{eq:strong_energy_inequality} with $u$ and $ (W^n)_n$ replaced by $\wt{u}$ and $(\wt{B}^n)_n$, respectively.
\end{proof}

\smallskip

\noindent
\emph{Acknowledgments.} The author is grateful to Federico Cornalba, Max Sauerbrey, and Esm\'ee Theewis for inspiring suggestions and discussions.

\medskip

\noindent
\textbf{Data availability.} This manuscript has no associated data.

\medskip

\noindent
\textbf{Declaration -- Conflict of interest.} The author has no conflict of interest.

\def\polhk#1{\setbox0=\hbox{#1}{\ooalign{\hidewidth
  \lower1.5ex\hbox{`}\hidewidth\crcr\unhbox0}}} \def\cprime{$'$}


\begin{thebibliography}{100}

\bibitem{Primitive3}
A.~Agresti.
\newblock The primitive equations with rough transport noise: Global
  well-posedness and regularity.
\newblock {\em arXiv preprint arXiv:2310.01193}, 2023.

\bibitem{A24_global_small}
A.~Agresti.
\newblock Global smooth solutions by transport noise of 3{D} {N}avier-{S}tokes
  equations with small hyperviscosity.
\newblock {\em arXiv preprint arXiv:2406.09267}, 2024.
\newblock To appear in {A}nn. {P}robab.

\bibitem{Primitive1}
A.~Agresti, M.~Hieber, A.~Hussein, and M.~Saal.
\newblock The stochastic primitive equations with transport noise and turbulent
  pressure.
\newblock {\em Stoch. Partial Differ. Equ. Anal. Comput.}, 12(1):53--133, 2024.

\bibitem{Primitive2}
A.~Agresti, M.~Hieber, A.~Hussein, and M.~Saal.
\newblock The stochastic primitive equations with nonisothermal turbulent
  pressure.
\newblock {\em Ann. Appl. Probab.}, 35(1):635--700, 2025.

\bibitem{ALV21}
A.~Agresti, N.~Lindemulder, and M.C. Veraar.
\newblock On the trace embedding and its applications to evolution equations.
\newblock {\em Mathematische Nachrichten}, 296(4):1319--1350, 2023.

\bibitem{AgrSau}
A.~Agresti and M.~Sauerbrey.
\newblock Well-posedness of the stochastic thin-film equation with an interface
  potential.
\newblock {\em Comm. Math. Phys.}, 407(8):Paper No. 158, 66, 2026.

\bibitem{ASV25}
A.~Agresti, M.~Sauerbrey, and M.C. Veraar.
\newblock A stochastic flow approach to {D}e {G}iorgi-{N}ash-{M}oser estimates
  for {SPDE}s with smooth transport noise.
\newblock {\em arXiv preprint arXiv:2511.12692}, 2025.

\bibitem{AV19_QSEE1}
A.~Agresti and M.C. Veraar.
\newblock Nonlinear parabolic stochastic evolution equations in critical spaces
  part {I}. {S}tochastic maximal regularity and local existence.
\newblock {\em Nonlinearity}, 35(8):4100--4210, 2022.

\bibitem{AV19_QSEE2}
A.~Agresti and M.C. Veraar.
\newblock Nonlinear parabolic stochastic evolution equations in critical spaces
  part {II}.
\newblock {\em J. Evol. Equ.}, 22(2):Paper No. 56, 2022.

\bibitem{RD_AV23}
A.~Agresti and M.C. Veraar.
\newblock Reaction-diffusion equations with transport noise and critical
  superlinear diffusion: {L}ocal well-posedness and positivity.
\newblock {\em Journal of Differential Equations}, 368:247--300, 2023.

\bibitem{AV24_variational}
A.~Agresti and M.C. Veraar.
\newblock The critical variational setting for stochastic evolution equations.
\newblock {\em Probab. Theory Related Fields}, 188(3-4):957--1015, 2024.

\bibitem{AV23}
A.~Agresti and M.C. Veraar.
\newblock Reaction-diffusion equations with transport noise and critical
  superlinear diffusion: {G}lobal well-posedness of weakly dissipative systems.
\newblock {\em SIAM J. Math. Anal.}, 56(4):4870--4927, 2024.

\bibitem{AV_torus}
A.~Agresti and M.C. Veraar.
\newblock Stochastic maximal ${L}^p({L}^q)$-regularity for second order systems
  with periodic boundary conditions.
\newblock {\em Ann. Inst. Henri Poincar\'{e} Probab. Stat.}, 60(1):413--430,
  2024.

\bibitem{AV21_NS}
A.~Agresti and M.C. Veraar.
\newblock Stochastic {N}avier--{S}tokes equations for turbulent flows in
  critical spaces.
\newblock {\em Communications in Mathematical Physics}, 405(2):43, 2024.

\bibitem{AV25_survey}
A.~Agresti and M.C. Veraar.
\newblock Nonlinear {SPDE}s and {M}aximal {R}egularity: {A}n {E}xtended
  {S}urvey.
\newblock {\em NoDEA Nonlinear Differential Equations Appl.}, 32(6):Paper No.
  123, 2025.

\bibitem{ABC22_annals}
D.~Albritton, E.~Bru\'{e}, and M.~Colombo.
\newblock Non-uniqueness of {L}eray solutions of the forced {N}avier-{S}tokes
  equations.
\newblock {\em Ann. of Math. (2)}, 196(1):415--455, 2022.

\bibitem{B08_local_energy}
A.~Basson.
\newblock Spatially homogeneous solutions of 3{D} stochastic {N}avier-{S}tokes
  equations and local energy inequality.
\newblock {\em Stochastic Process. Appl.}, 118(3):417--451, 2008.

\bibitem{BeLo}
J.~Bergh and J.~L{\"o}fstr{\"o}m.
\newblock {\em Interpolation spaces. {A}n introduction}.
\newblock Springer-Verlag, Berlin, 1976.
\newblock Grundlehren der Mathematischen Wissenschaften, No. 223.

\bibitem{BR09_growth}
D.~Bl\"{o}mker and M.~Romito.
\newblock Regularity and blow up in a surface growth model.
\newblock {\em Dyn. Partial Differ. Equ.}, 6(3):227--252, 2009.

\bibitem{BR15_growth}
D.~Bl\"{o}mker and M.~Romito.
\newblock Stochastic {PDE}s and lack of regularity: a surface growth equation
  with noise: existence, uniqueness, and blow-up.
\newblock {\em Jahresber. Dtsch. Math.-Ver.}, 117(4):233--286, 2015.

\bibitem{Breit_partial}
D.~Breit.
\newblock Partial boundary regularity for the {N}avier-{S}tokes equations in
  irregular domains.
\newblock {\em J. Funct. Anal.}, 289(12):Paper No. 111188, 39, 2025.

\bibitem{BFH18_book}
D.~Breit, E.~Feireisl, and M.~Hofmanov\'{a}.
\newblock {\em Stochastically forced compressible fluid flows}, volume~3 of
  {\em De Gruyter Series in Applied and Numerical Mathematics}.
\newblock De Gruyter, Berlin, 2018.

\bibitem{BPW25}
D.~Breit, A.~Prohl, and J.~Wichmann.
\newblock Numerical analysis of the stochastic {N}avier-{S}tokes equations.
\newblock {\em arXiv preprint arXiv:2508.05564}, 2025.

\bibitem{BrCaFl}
Z.~Brze{\'z}niak, M.~Capi{\'n}ski, and F.~Flandoli.
\newblock Stochastic partial differential equations and turbulence.
\newblock {\em Math. Models Methods Appl. Sci.}, 1(1):41--59, 1991.

\bibitem{BzMo13}
Z.~Brze\'{z}niak and E.~Motyl.
\newblock Existence of a martingale solution of the stochastic
  {N}avier-{S}tokes equations in unbounded 2{D} and 3{D} domains.
\newblock {\em J. Differential Equations}, 254(4):1627--1685, 2013.

\bibitem{NVW2}
Z.~Brze{\'z}niak, J.M.A.M.~van Neerven, M.C. Veraar, and L.W. Weis.
\newblock It\^o's formula in {UMD} {B}anach spaces and regularity of solutions
  of the {Z}akai equation.
\newblock {\em J. Differential Equations}, 245(1):30--58, 2008.

\bibitem{BCV22}
T.~Buckmaster, M.~Colombo, and V.~Vicol.
\newblock Wild solutions of the {N}avier-{S}tokes equations whose singular sets
  in time have {H}ausdorff dimension strictly less than 1.
\newblock {\em J. Eur. Math. Soc.}, 24(9):3333--3378, 2022.

\bibitem{CKN82}
L.~Caffarelli, R.~Kohn, and L.~Nirenberg.
\newblock Partial regularity of suitable weak solutions of the
  {N}avier-{S}tokes equations.
\newblock {\em Comm. Pure Appl. Math.}, 35(6):771--831, 1982.

\bibitem{Can04}
M.~Cannone.
\newblock Harmonic analysis tools for solving the incompressible
  {N}avier-{S}tokes equations.
\newblock In {\em Handbook of mathematical fluid dynamics. {V}ol. {III}}, pages
  161--244. North-Holland, Amsterdam, 2004.

\bibitem{CK22_waek_strong}
A.~Chaudhary and U.~Koley.
\newblock On weak-strong uniqueness for stochastic equations of incompressible
  fluid flow.
\newblock {\em J. Math. Fluid Mech.}, 24(3):Paper No. 62, 33, 2022.

\bibitem{CD_suitable_Weak}
W.~Chen and Z.~Dong.
\newblock Martingale suitable weak solutions of 3-{D} stochastic
  {N}avier-{S}tokes equations with vorticity bounds.
\newblock {\em J. Funct. Anal.}, 289(9):Paper No. 111081, 62, 2025.

\bibitem{CY19_1}
H.J. Choe and M.~Yang.
\newblock The {M}inkowski dimension of boundary singular points in the
  {N}avier-{S}tokes equations.
\newblock {\em J. Differential Equations}, 267(8):4705--4718, 2019.

\bibitem{CDLM20}
M.~Colombo, C.~De~Lellis, and A.~Massaccesi.
\newblock The generalized {C}affarelli-{K}ohn-{N}irenberg theorem for the
  hyperdissipative {N}avier-{S}tokes system.
\newblock {\em Communications on Pure and Applied Mathematics}, 73(3):609--663,
  2020.

\bibitem{DG17_continuity}
K.~Dareiotis and M.~Gerencs\'{e}r.
\newblock Local {$L_\infty$}-estimates, weak {H}arnack inequality, and
  stochastic continuity of solutions of {SPDE}s.
\newblock {\em J. Differential Equations}, 262(1):615--632, 2017.

\bibitem{DGGS26}
K.~Dareiotis, B.~Gess, M.~Gnann, and M.~Sauerbrey.
\newblock Solutions to the stochastic thin-film equation for initial values
  with non-full support.
\newblock {\em Transactions of the American Mathematical Society}, 2026.

\bibitem{DGHT11}
A.~Debussche, N.~Glatt-Holtz, and R.~Temam.
\newblock Local martingale and pathwise solutions for an abstract fluids model.
\newblock {\em Phys. D}, 240(14-15):1123--1144, 2011.

\bibitem{DM25_PhysD}
A.~Debussche and E.~M\'{e}min.
\newblock Variational principles for fully coupled stochastic fluid dynamics
  across scales.
\newblock {\em Phys. D}, 481:Paper No. 134777, 11, 2025.

\bibitem{DP22_two_scale}
A.~Debussche and U.~Pappalettera.
\newblock Second order perturbation theory of two-scale systems in fluid
  dynamics.
\newblock {\em Journal of the European Mathematical Society}, 2024.

\bibitem{EG_measure}
L.C. Evans and R.F. Gariepy.
\newblock {\em Measure theory and fine properties of functions}.
\newblock Textbooks in Mathematics. CRC Press, Boca Raton, FL, revised edition,
  2015.

\bibitem{Falconer_book}
K.~Falconer.
\newblock {\em Fractal geometry: Mathematical foundations and applications}.
\newblock John Wiley \& Sons, Ltd., Chichester, third edition, 2014.

\bibitem{F15_renormalized}
J.~Fischer.
\newblock Global existence of renormalized solutions to entropy-dissipating
  reaction-diffusion systems.
\newblock {\em Arch. Ration. Mech. Anal.}, 218(1):553--587, 2015.

\bibitem{F17_weak_strong}
J.~Fischer.
\newblock Weak-strong uniqueness of solutions to entropy-dissipating
  reaction-diffusion equations.
\newblock {\em Nonlinear Anal.}, 159:181--207, 2017.

\bibitem{FG95_PTRF}
F.~Flandoli and D.~Gatarek.
\newblock Martingale and stationary solutions for stochastic {N}avier-{S}tokes
  equations.
\newblock {\em Probab. Theory Related Fields}, 102(3):367--391, 1995.

\bibitem{FL19}
F.~Flandoli and D.~Luo.
\newblock High mode transport noise improves vorticity blow-up control in 3{D}
  {N}avier-{S}tokes equations.
\newblock {\em Probab. Theory Related Fields}, 180, 2021.

\bibitem{FL32_book}
F.~Flandoli and E.~Luongo.
\newblock {\em Stochastic partial differential equations in fluid mechanics},
  volume 2330.
\newblock Springer Nature, 2023.

\bibitem{FlaPa21}
F.~Flandoli and U.~Pappalettera.
\newblock From additive to transport noise in {2D} fluid dynamics.
\newblock {\em Stochastics and Partial Differential Equations: Analysis and
  Computations}, pages 1--41, 2022.

\bibitem{FR_partial}
F.~Flandoli and M.~Romito.
\newblock Partial regularity for the stochastic {N}avier-{S}tokes equations.
\newblock {\em Trans. Amer. Math. Soc.}, 354(6):2207--2241, 2002.

\bibitem{gess2023landau}
B.~Gess, D.~Heydecker, and Z.~Wu.
\newblock {L}andau-{L}ifshitz-{N}avier-{S}tokes equations: {L}arge deviations
  and relationship to the energy equality.
\newblock {\em arXiv preprint arXiv:2311.02223}, 2023.

\bibitem{GL26_strong}
B.~Gess and R.~Lasarzik.
\newblock Probabilistically strong solutions to stochastic {E}uler equations.
\newblock {\em arXiv preprint arXiv:2601.22073}, 2026.

\bibitem{GY25}
B.~Gess and I.~Yaroslavtsev.
\newblock Stabilization by transport noise and enhanced dissipation in the
  {K}raichnan model.
\newblock {\em J. Evol. Equ.}, 25(2):Paper No. 42, 63, 2025.

\bibitem{Currents1}
M.~Giaquinta, G.~Modica, and J.~Sou\v{c}ek.
\newblock {\em Cartesian currents in the calculus of variations. {I}},
  volume~37 of {\em Ergebnisse der Mathematik und ihrer Grenzgebiete. 3. Folge.
  A Series of Modern Surveys in Mathematics}.
\newblock Springer-Verlag, Berlin, 1998.

\bibitem{Currents2}
M.~Giaquinta, G.~Modica, and J.~Sou\v{c}ek.
\newblock {\em Cartesian currents in the calculus of variations. {II}},
  volume~38 of {\em Ergebnisse der Mathematik und ihrer Grenzgebiete. 3. Folge.
  A Series of Modern Surveys in Mathematics}.
\newblock Springer-Verlag, Berlin, 1998.

\bibitem{G86}
Y.~Giga.
\newblock Solutions for semilinear parabolic equations in {$L^p$} and
  regularity of weak solutions of the {N}avier-{S}tokes system.
\newblock {\em J. Differential Equations}, 62(2):186--212, 1986.

\bibitem{G26_smooth}
D.~Goodair.
\newblock High order smoothness for stochastic {N}avier-{S}tokes equations with
  transport and stretching noise on bounded domains.
\newblock {\em Nonlinear Anal.}, 267:Paper No. 114054, 2026.

\bibitem{GC24_Lie}
D.~Goodair and D.~Crisan.
\newblock On the 3{D} {N}avier-{S}tokes equations with stochastic {L}ie
  transport.
\newblock In {\em Stochastic transport in upper ocean dynamics {II}}, volume~11
  of {\em Math. Planet Earth}, pages 53--110. Springer, Cham, 2024.

\bibitem{GK25_conditional}
I.~Gy{\"o}ngy and N.V. Krylov.
\newblock On conditional uniqueness of solutions to stochastic
  {N}avier-{S}tokes {E}quations.
\newblock {\em J. Dynam. Differential Equations}, 38(suppl. 1):349--375, 2026.

\bibitem{HLN19}
M.~Hofmanov\'{a}, J.-M. Leahy, and T.~Nilssen.
\newblock On the {N}avier-{S}tokes equation perturbed by rough transport noise.
\newblock {\em J. Evol. Equ.}, 19, 2019.

\bibitem{HLN21}
M.~Hofmanov\'{a}, J.-M. Leahy, and T.~Nilssen.
\newblock On a rough perturbation of the {N}avier-{S}tokes system and its
  vorticity formulation.
\newblock {\em Ann. Appl. Probab.}, 31(2):736--777, 2021.

\bibitem{H15_SVP}
D.D. Holm.
\newblock Variational principles for stochastic fluid dynamics.
\newblock {\em Proceedings of the Royal Society A: Mathematical, Physical and
  Engineering Sciences}, 471(2176):20140963, 2015.

\bibitem{Hopf}
E.~Hopf.
\newblock \"{U}ber die {A}nfangswertaufgabe f\"{u}r die hydrodynamischen
  {G}rundgleichungen.
\newblock {\em Math. Nachr.}, 4:213--231, 1951.

\bibitem{hou2025nonuniqueness}
T.~Hou, Y.~Wang, and C.~Yang.
\newblock Nonuniqueness of {L}eray-{H}opf solutions to the unforced
  incompressible {3D} {N}avier-{S}tokes equation.
\newblock {\em arXiv preprint arXiv:2509.25116}, 2025.

\bibitem{HWZ17}
E.~P. Hsu, Y.~Wang, and Z.~Wang.
\newblock Stochastic {D}e {G}iorgi iteration and regularity of stochastic
  partial differential equations.
\newblock {\em Ann. Probab.}, 45(5):2855--2866, 2017.

\bibitem{Analysis3}
T.~Hyt\"{o}nen, J.M.A.M. van Neerven, M.C. Veraar, and L.~Weis.
\newblock {\em Analysis in {B}anach spaces. {V}ol. {III}. {H}armonic analysis
  and spectral theory}, volume~76 of {\em Ergebnisse der Mathematik und ihrer
  Grenzgebiete. 3. Folge.}
\newblock Springer, 2023.

\bibitem{Analysis2}
T.P. Hyt\"onen, J.M.A.M.~van Neerven, M.C. Veraar, and L.~Weis.
\newblock {\em Analysis in {B}anach spaces. {V}ol. {II}. {P}robabilistic
  {M}ethods and {O}perator {T}heory.}, volume~67 of {\em Ergebnisse der
  Mathematik und ihrer Grenzgebiete. 3. Folge.}
\newblock Springer, 2017.

\bibitem{Analysis1}
T.P. Hyt\"onen, J.M.A.M.~van Neerven, M.C. Veraar, and L.W. Weis.
\newblock {\em Analysis in {B}anach spaces. {V}ol. {I}. {M}artingales and
  {L}ittlewood-{P}aley theory}, volume~63 of {\em Ergebnisse der Mathematik und
  ihrer Grenzgebiete. 3. Folge.}
\newblock Springer, 2016.

\bibitem{JW14}
Q.~Jiu and Y.~Wang.
\newblock On possible time singular points and eventual regularity of weak
  solutions to the fractional {N}avier-{S}tokes equations.
\newblock {\em Dyn. Partial Differ. Equ.}, 11(4):321--343, 2014.

\bibitem{KY16_1}
Y.~Koh and M.~Yang.
\newblock The {M}inkowski dimension of interior singular points in the
  incompressible {N}avier-{S}tokes equations.
\newblock {\em J. Differential Equations}, 261(6):3137--3148, 2016.

\bibitem{Kry96}
N.V. Krylov.
\newblock On {$L\sb p$}-theory of stochastic partial differential equations in
  the whole space.
\newblock {\em SIAM J. Math. Anal.}, 27(2):313--340, 1996.

\bibitem{Kry}
N.V. Krylov.
\newblock An analytic approach to {SPDE}s.
\newblock In {\em Stochastic partial differential equations: six perspectives},
  volume~64 of {\em Math. Surveys Monogr.}, pages 185--242. Amer. Math. Soc.,
  Providence, RI, 1999.

\bibitem{KryOverview}
N.V. Krylov.
\newblock A brief overview of the {$L_p$}-theory of {SPDE}s.
\newblock {\em Theory Stoch. Process.}, 14(2):71--78, 2008.

\bibitem{K09_singular_times}
I.~Kukavica.
\newblock The fractal dimension of the singular set for solutions of the
  {N}avier-{S}tokes system.
\newblock {\em Nonlinearity}, 22(12):2889--2900, 2009.

\bibitem{K09_singular_times2}
I.~Kukavica and Y.~Pei.
\newblock An estimate on the parabolic fractal dimension of the singular set
  for solutions of the {N}avier-{S}tokes system.
\newblock {\em Nonlinearity}, 25(9):2775--2783, 2012.

\bibitem{KX26_JDE}
I.~Kukavica and F.~Xu.
\newblock On the almost global existence of the stochastic {N}avier-{S}tokes
  equations in {$L^3$} with small data.
\newblock {\em J. Differential Equations}, 457:Paper No. 113987, 27, 2026.

\bibitem{LePi}
P.~G. Lemari\'{e}-Rieusset.
\newblock {\em The {N}avier-{S}tokes problem in the 21st century}.
\newblock CRC Press, Boca Raton, FL, 2016.

\bibitem{Leray}
J.~Leray.
\newblock Sur le mouvement d'un liquide visqueux emplissant l'espace.
\newblock {\em Acta Math.}, 63(1):193--248, 1934.

\bibitem{L98_1}
F.~Lin.
\newblock A new proof of the {C}affarelli-{K}ohn-{N}irenberg theorem.
\newblock {\em Comm. Pure Appl. Math.}, 51(3):241--257, 1998.

\bibitem{LR15}
W.~Liu and M.~R\"{o}ckner.
\newblock {\em Stochastic partial differential equations: an introduction}.
\newblock Universitext. Springer, Cham, 2015.

\bibitem{InterpolationLunardi}
A.~Lunardi.
\newblock {\em Interpolation theory}.
\newblock Appunti. Scuola Normale Superiore di Pisa. Edizioni della Normale,
  Pisa, second edition, 2009.

\bibitem{L19_PhysD}
X.~Luo.
\newblock On the possible time singularities for the 3{D} {N}avier-{S}tokes
  equations.
\newblock {\em Phys. D}, 395:37--42, 2019.

\bibitem{MK99_simplified}
A.J. Majda and P.R. Kramer.
\newblock Simplified models for turbulent diffusion: theory, numerical
  modelling, and physical phenomena.
\newblock {\em Phys. Rep.}, 314(4-5):237--574, 1999.

\bibitem{M14_derivation}
E.~M\'{e}min.
\newblock Fluid flow dynamics under location uncertainty.
\newblock {\em Geophys. Astrophys. Fluid Dyn.}, 108(2):119--146, 2014.

\bibitem{MR01}
R.~Mikulevicius and B.~Rozovskii.
\newblock On equations of stochastic fluid mechanics.
\newblock In {\em Stochastics in finite and infinite dimensions}, Trends Math.,
  pages 285--302. Birkh\"{a}user Boston, Boston, MA, 2001.

\bibitem{MR05}
R.~Mikulevicius and B.~Rozovskii.
\newblock Global {$L_2$}-solutions of stochastic {N}avier-{S}tokes equations.
\newblock {\em Ann. Probab.}, 33(1):137--176, 2005.

\bibitem{MR04}
R.~Mikulevicius and B.L. Rozovskii.
\newblock Stochastic {N}avier--{S}tokes equations for turbulent flows.
\newblock {\em SIAM Journal on Mathematical Analysis}, 35(5):1250--1310, 2004.

\bibitem{NVW1}
J.M.A.M.~van Neerven, M.C. Veraar, and L.W. Weis.
\newblock Stochastic integration in {UMD} {B}anach spaces.
\newblock {\em Ann. Probab.}, 35(4):1438--1478, 2007.

\bibitem{MaximalLpregularity}
J.M.A.M.~van Neerven, M.C. Veraar, and L.W. Weis.
\newblock Stochastic maximal {$L^p$}-regularity.
\newblock {\em Ann. Probab.}, 40(2):788--812, 2012.

\bibitem{NVW13}
J.M.A.M.~van Neerven, M.C. Veraar, and L.W. Weis.
\newblock Stochastic integration in {B}anach spaces---a survey.
\newblock In {\em Stochastic analysis: a series of lectures}, volume~68 of {\em
  Progr. Probab.}, pages 297--332. Birkh\"{a}user/Springer, Basel, 2015.

\bibitem{OR19}
W.~S. O\.{z}a\'{n}ski and J.~C. Robinson.
\newblock Partial regularity for a surface growth model.
\newblock {\em SIAM J. Math. Anal.}, 51(1):228--255, 2019.

\bibitem{VP18}
P.~Portal and M.C. Veraar.
\newblock Stochastic maximal regularity for rough time-dependent problems.
\newblock {\em Stoch. Partial Differ. Equ. Anal. Comput.}, 7(4):541--597, 2019.

\bibitem{CriticalQuasilinear}
J.~Pr\"{u}ss, G.~Simonett, and M.~Wilke.
\newblock Critical spaces for quasilinear parabolic evolution equations and
  applications.
\newblock {\em J. Differential Equations}, 264(3):2028--2074, 2018.

\bibitem{addendum}
J.~Pr\"{u}ss and M.~Wilke.
\newblock Addendum to the paper ``{O}n quasilinear parabolic evolution
  equations in weighted {$L_p$}-spaces {II}''.
\newblock {\em J. Evol. Equ.}, 17(4):1381--1388, 2017.

\bibitem{PW18}
J.~Pr\"{u}ss and M.~Wilke.
\newblock On critical spaces for the {N}avier-{S}tokes equations.
\newblock {\em J. Math. Fluid Mech.}, 20(2):733--755, 2018.

\bibitem{RRS_3D}
J.C. Robinson, J.L. Rodrigo, and W.~Sadowski.
\newblock {\em The three-dimensional {N}avier-{S}tokes equations}, volume 157
  of {\em Cambridge Studies in Advanced Mathematics}.
\newblock Cambridge University Press, Cambridge, 2016.
\newblock Classical theory.

\bibitem{RS1_07}
J.C. Robinson and W.~Sadowski.
\newblock Decay of weak solutions and the singular set of the three-dimensional
  {N}avier-{S}tokes equations.
\newblock {\em Nonlinearity}, 20(5):1185--1191, 2007.

\bibitem{RS2_09}
J.C. Robinson and W.~Sadowski.
\newblock Almost-everywhere uniqueness of {L}agrangian trajectories for
  suitable weak solutions of the three-dimensional {N}avier-{S}tokes equations.
\newblock {\em Nonlinearity}, 22(9):2093--2099, 2009.

\bibitem{S76_partial}
V.~Scheffer.
\newblock Partial regularity of solutions to the {N}avier-{S}tokes equations.
\newblock {\em Pacific J. Math.}, 66(2):535--552, 1976.

\bibitem{S76_original}
V.~Scheffer.
\newblock Turbulence and {H}ausdorff dimension.
\newblock In {\em Turbulence and {N}avier-{S}tokes equations ({P}roc. {C}onf.,
  {U}niv. {P}aris-{S}ud, {O}rsay, 1975)}, Lecture Notes in Math., Vol. 565,
  pages 174--183. Springer, Berlin-New York, 1976.

\bibitem{Scheffer_partial_regularity}
V.~Scheffer.
\newblock Hausdorff measure and the {N}avier-{S}tokes equations.
\newblock {\em Comm. Math. Phys.}, 55(2):97--112, 1977.

\bibitem{schmeisser1987topics}
H.J. Schmeisser and H.~Triebel.
\newblock {\em Topics in Fourier Analysis and Function Spaces}.
\newblock Wiley, 1987.

\bibitem{WW24_1}
Y.~Wang and G.~Wu.
\newblock Fractal dimension of potential singular points set in the
  {N}avier-{S}tokes equations under supercritical regularity.
\newblock {\em Proc. Roy. Soc. Edinburgh Sect. A}, 154(3):727--745, 2024.

\end{thebibliography}
\end{document}